\documentclass[review,onefignum,onetabnum]{siamart190516}

\newcommand\scalemath[2]{\scalebox{#1}{\mbox{\ensuremath{\displaystyle #2}}}}

\usepackage{amsmath}
\usepackage{geometry}
\usepackage{amssymb,latexsym}
\usepackage{verbatim}
\usepackage{extarrows}
\usepackage{enumerate}
\usepackage{txfonts}
\usepackage{mathtools}
\usepackage{booktabs}
\usepackage{bbm}

\usepackage[font=small,labelfont=bf, labelsep = quad]{caption}
\usepackage{subfigure}
\usepackage{graphicx}
\usepackage{caption}
\usepackage{mwe}    
\usepackage{hyperref}
\hypersetup{
    colorlinks= True,
    linkcolor=blue,
    filecolor=magenta,      
    urlcolor=cyan,
}
\usepackage{xurl}
\usepackage{xcolor}

\usepackage{mathtools}
\usepackage[tableposition=top]{caption}
\usepackage{booktabs,dcolumn}

\usepackage{bbm} 

\newcommand\R{\mathbb{R}}

\usepackage{url}

\usepackage{xparse}
\usepackage{tikz}
\usetikzlibrary{matrix,backgrounds}
\pgfdeclarelayer{myback}
\pgfsetlayers{myback,background,main}

\tikzset{mycolor/.style = {line width=1bp,color=#1}}%
\tikzset{myfillcolor/.style = {draw,fill=#1}}%

\NewDocumentCommand{\highlight}{O{blue!40} m m}{%
\draw[mycolor=#1] (#2.north west)rectangle (#3.south east);
}

\newcommand*\stref[1]{Table S\ref{#1}}

\newcommand*\ssref[1]{Sec. S\ref{#1}}

\begin{document}
\nolinenumbers
\title{Optimal Vaccine Allocation for Pandemic Stabilization} 
\author{Qianqian Ma\thanks{Department of Electrical and Computer Engineering, Boston
  University, Boston, MA, USA } \and Yang-Yu Liu\thanks{Channing Division of Network Medicine,
  Brigham and Women's Hospital, Harvard Medical School, Boston, MA
  02115, USA} \and Alex Olshevsky\thanks{Department of Electrical and Computer Engineering and Division
  of Systems Engineering, Boston University, Boston, MA, USA}}
\maketitle

\begin{abstract} 
How to strategically allocate the available vaccines is a crucial issue for pandemic control. In this work, we propose a mathematical framework for optimal stabilizing vaccine allocation, where our goal is to send the infections to zero as soon as possible with a fixed number of vaccine doses.  This framework allows us to efficiently compute the optimal vaccine allocation policy for general epidemic spread models including SIS/SIR/SEIR and a new model of COVID-19 transmissions. By fitting the real data in New York State to our framework, we found that the optimal stabilizing vaccine allocation policy suggests offering vaccines priority to locations where there are more susceptible people and where the residents spend longer time outside the home. Besides, we found that offering vaccines priority to young adults (20-29) and middle-age adults (20-44) can minimize the cumulative infected cases and the death cases. Moreover, we compared our method with five age-stratified strategies in \cite{bubar2021model} based on their epidemics model. We also found it's better to offer vaccine priorities to young people to curb the disease and minimize the deaths when the basic reproduction number $R_0$ is moderately above one, which describes the most world during COVID-19. Such phenomenon has been ignored in \cite{bubar2021model}.
\end{abstract} 

\tableofcontents
\section{Introduction}
The COVID-19 pandemic has caused almost 187M cases and 3.7M (June 2021) deaths worldwide, and an unprecedented social and economic cost. Untill now, FDA has approved three types of COVID-19 vaccines, and around 40\% people in USA has been fully vaccinated for coronavirus. However, it may still takes months until sufficient vaccines are available to overcome the pandemic. Therefore, it is important to strategically allocate the available vaccines such that the number of the infected cases as well as the death cases remains as small as possible.

In this work, we consider the optimal stabilizing vaccine allocation issue: how to allocate a fixed number of vaccines to different locations and different age groups so that the number of infections goes to zero as soon as possible. Here we propose a framework to design optimal stabilizing vaccine allocation policy for a COVID-19 transmission model with symptomatic and asymptomatic compartments. We consider two different scenarios. In our first scenario, the network model we consider consists of locations, where the demographic difference is ignored. Next, we also study the scenario where the demographic structure of each location is considered. 

We apply the proposed algorithm to design optimal stabilizing vaccine allocation policy on both synthetic and real data (using  data from SafeGraph \cite{travel} to fit a county-level model of New York State) for epidemic spread models of COVID-19 using disease parameters from CDC\cite{old_people, cdc_data_tracker,reporting_rate,symptom_rate}. 

Notation: $A \succcurlyeq B$ implies matrix $A - B$ is positive semi-definite. $A \preccurlyeq B$ implies matrix $B - A$ is positive semi-definite. $I$ represents an identity matrix. $\lambda_{\rm max}$ is the largest eigenvalue of matrix $A$.

\section{Results}
All the epidemic spread models considered in this work are compartmental or network models \cite{birge2020controlling} with ``locations'' corresponding to  neighborhoods,  counties, or other geographic subdivisions. 
Our framework can be applied to general epidemic spread models. For demonstration purpose, here we consider a simple model of COVID-19 spreading which contains the classical Susceptible-Infectious-Recovered (SIR) model and the Susceptible-Exposed-Infectious-Recovered (SEIR) model as special cases.

\subsection{A network model of  COVID-19.}
We consider a simple model (similar to models in literature \cite{khanafer2014optimal,giordano2020modelling,birge2020controlling, pagliara2020adaptive, carli2020model, zino2020assessing}) of COVID-19 spreading that breaks infected individuals into two types: asymptomatic and symptomatic. This model allows individuals transmit the infection at different rates: 
\begin{equation}\label{eq:COVID}
	\begin{aligned}
		\dot{s}_i & =  - s_i \sum_{j=1}^n a_{ij} (\beta^{\text a} x^{\text a}_j + \beta^{\text s} x^{\text s}_j) \\ 
		\dot{x}^{\text a}_i & =  s_i \sum_{j=1}^n a_{ij} (\beta^{\text a} x^{\text a}_j + \beta^{\text s} x^{\text s}_j) - (\epsilon + r^{\text a}) x^{\text a}_i \\ 
		\dot{x}^{\text s}_i & =  \epsilon x^{\text a}_i - r^{\text s} x^{\text s}_i  - \kappa x^{\text s}_i  \\
		 \dot{e}_i &= \kappa x^{\text s}_i \\
		 \dot{h}_i &= r^{\text a} x^{\text a}_i + r^{\text s} x^{\text s}_i
	\end{aligned}.
\end{equation}
Here $s_i$ ($x_i^{\text a}$ or $x_i^{\text s}$) stands for the proportion of susceptible (asymptomatic or symptomatic infected, respectively) population at location $i$, $a_{ij}$ captures the rate at which infection flows from location $j$ to location $i$, $e_i$ (or $h_i$) stands for the proportion of deceased (or recovered) population at location $i$,  $\beta^{\text a}$ (or $\beta^{\text s}$) is the transmission rate of asymptomatic (or symptomatic) infected individuals, $r^\text{a}$ (or $r^\text{s}$) is the recovery rate of asymptomatic (or symptomatic) infected individuals. We assume that infected individuals are asymptomatic at first and $\epsilon$ is the rate at which they develop symptoms, $\kappa$ is the rate at which the symptomatic patients die. We use different parameters for symptomatic and asymptomatic individuals because a recent study \cite{Kissler2020.10.21.20217042} reported that asymptomatic individuals have viral load that drops more quickly, so they not only recover faster, but also are probably less contagious. 

Note that our model of COVID-19 spreading can be considered as a generalization of the classical SIR model and the SEIR model of epidemic spread. Indeed, by setting $\beta^{\text s}=\epsilon= r^{\text s} =\kappa= 0$, we recover the SIR model; and by setting $\beta^{\text a}=r^{\text a} =\kappa= 0$, we recover the SEIR model. However, neither the SIR nor the SEIR model captures the existence of two classes of individuals who transmit infections at different rates as above. 

We follow the method in \cite{ma2020optimal} to define the quantities $a_{ij}$ as
\begin{equation}\label{eq: a_ij}
	a_{ij} = \sum_{l=1}^n  \tau_{il} \tau_{jl} \frac{N^{*}_j}{\sum_{k=1}^n N^{*}_k \tau_{kl}}, 
\end{equation}
where $N^{*}_j$ denotes the resident population at location $j$, and people travel from location $i$ to location $j$ at rate $\tau_{ij}$. Since people travel between different locations, the total population of a location is time-varying, here $N^{*}_j$ represents the population of the long-time residents of location $j$.

In matrix form, we can write Eq. \eqref{eq:COVID} as 
\begin{equation} \label{COVID_19}
	\renewcommand{\arraystretch}{0.75}
	\left( 
	\begin{array}{c} 
		\dot{s} \\
		\dot{x}^{\text a} \\ 
		\dot{x}^{\text s}
	\end{array} 
	\right) = 
	\begin{tikzpicture}[baseline=-\the\dimexpr\fontdimen22\textfont2\relax ]
		\matrix (m)[matrix of math nodes, left delimiter=(,right delimiter=), nodes={minimum width=5em, minimum height=1.4em}]
		{
			0 & -\beta^{\text a} {\rm diag}(s) A &  - \beta^{\text s} {\rm diag}(s) A \\
			0 & \beta^{\text a} {\rm diag}(s) A - (\epsilon + r^{\text a}) &  \beta^{\text s} {\rm diag}(s) A \\
			0 & \epsilon & - {\rm diag}(r^{\text s} + \kappa)\\
		};
		
		\begin{pgfonlayer}{myback}
			\highlight[gray]{m-2-2}{m-3-3}
		\end{pgfonlayer}
	\end{tikzpicture}
	\left( 
	\begin{array}{c} 
		s \\
		x^{\text a} \\ 
		x^{\text s}
	\end{array} 
	\right),
\end{equation}
where scalars in the matrix should be understood as multiplying the identity matrix and
\begin{equation}
	A = \tau {\rm diag}\left(\sum_{k} N^{*}_k \tau_{kl}\right)^{-1} \tau^\top {\rm diag}(N_i^*)
\end{equation}
where $\tau = (\tau_{ij})$.  Let us write $M(t)$ for the bottom right $2n \times 2n$ submatrix (outlined by a  box) in Eq. \eqref{COVID_19}. According to Proposition 2 in \cite{ma2020optimal}, if we want the number of infections at each location (or a linear combination of those numbers) to go to zero at a prescribed rate $\alpha$, we just need to ensure that the linear eigenvalue condition $\lambda_{\rm max}(M(t_0)) \leq -\alpha$ holds.

\subsection{The COVID-19 model with demographic structure.}
In this section, we consider the COVID-19 model with demographic structures. The population of each location is partitioned into six mutually exclusive age groups \cite{enayati2020optimal}: preschool children (0-4 years), school children (5-19 years), young adults (20-29 years), middle age adults (30-44 years), middle age adults (45-64 years), and seniors (65 years and over). Contact intensities between different age groups can be different, which in turn can lead to difference of the infection flows.

To construct the epidemic models with demographic structure, we will first introduce the contact matrix  $C$ \cite{mossong2008social,arregui2018projecting,prem2017projecting}, where $C_{ij}$ is the mean number of contacts that an individual of group $i$ has with other individuals of group $j$ during a day. In our model, we will not directly use the contact matrix $C$, instead, we will use the intrinsic connectivity matrix $\Gamma$ \cite{arregui2018projecting}\cite{Brittoneabc6810,de2018impact,prem2017projecting}, which is defined as
\[\Gamma_{ij} = M_{ij} \frac{N}{N_j}, \]
where $N$ is the total population, $N_j$ is the population of age group $j$. $\Gamma_{ij}$ corresponds to the contact pattern in a ``rectangular" demography \cite{arregui2018projecting}(a population structure where all age groups have the same density). We will use the matrix $\Gamma$ to quantify the infection flows between different age groups.

We can write the COVID-19 model with demographic structures as
\begin{equation} \label{eq: covid 19, demographic}
	\renewcommand{\arraystretch}{0.75}
	\left( 
	\begin{array}{c} 
		\dot{s} \\
		\dot{x}^{\text a} \\ 
		\dot{x}^{\text s}
	\end{array} 
	\right) = 
\left( 
\begin{matrix}
0 & -{\rm diag}(\beta^{\text a'}) {\rm diag}(s) A' &  -{\rm diag} (\beta^{\text s'} ){\rm diag}(s) A' \\
0 & -{\rm diag}(\beta^{\text a'}) {\rm diag}(s) A' - (\epsilon + r^{\text a}) &  -{\rm diag}(\beta^{\text s'}) {\rm diag}(s) A' \\
0 & \epsilon & - {\rm diag}(r^{\text s} + \kappa)
\end{matrix}\right)
	\left( 
	\begin{array}{c} 
		s \\
		x^{\text a} \\ 
		x^{\text s}
	\end{array} 
	\right).
\end{equation}
Here $s, x^{\text a}, x^{\text s} \in \mathbb{R}^{6n\times 1}$,  $s_i(a)$ ($x_i^{\text a}(a)$ or $x_i^{\text s}(a)$) stands for the proportion of susceptible (asymptomatic or symptomatic infected, respectively) population of age group a) at location $i$, 
\begin{equation}\label{eq: A prime definition}
	A' = (\bar{A} \otimes \Gamma) {\rm diag}(N^*) \in \mathbb{R}^{6n\times 6n},
\end{equation}
where
\begin{equation}\label{eq: bar_A}
\bar{A} = \tau {\rm diag}\left(\sum_{k} N^{*}_k \tau_{kl}\right)^{-1} \tau^\top.	
\end{equation}
The details about how the matrix $A'$ is constructed is presented in SI Sec. \ref{sec: construction A'}. $\beta^{\text a'}$ (or $\beta^{\text s'}$) is the transmission risk of asymptomatic (or symptomatic) infected individuals. The difference between $\beta^{\text a}$ (or $\beta^{\text s}$) and $\beta^{\text a'}$ (or $\beta^{\text s'}$) is that the former is the probability that a susceptible individual get infected by an asymptomatic (or symptomatic) individual in a day, while the latter is the probability that a susceptible individual get infected from a meeting with an asymptomatic (or symptomatic) individual. Such meetings can happen multiple times in a day.  Note that $\kappa$, $\beta^{\rm a'}$, $\beta^{\rm s'}\in \mathbb{R}^{6n\times 1}$ in \eqref{eq: covid 19, demographic},  as we assume people in different age groups have different values of mortality rate and transmission risk,  as reported in COVID-19\cite{old_people}\cite{davies2020age}.

\subsection{optimal stabilizing vaccine allocation design for the COVID-19 model.} The optimal stabilizing vaccine allocation problem we consider can be summarized as follows: suppose the number of the available vaccine doses is fixed, we want to send the infections going to zero as soon as possible by allocating the vaccines to different locations in a non-uniform way. This is equivalent to fix the decay rate of the epidemics and minimize the number of the vaccine doses used. If we can solve the second problem, it is convenient for us to solve the first problem by using the binary search method. 

Suppose the vaccines are given to people at time $t_0$, and the vaccines will be effective immediately. 
The vaccinated people are no longer susceptible, then the initial susceptible rate of location $i$ for the COVID-19 model is $s_i(t_0) - \psi v_i$, 
where $v_i$ is the proportion of the vaccinated population at location $i$,  $\psi$ is the efficacy of the vaccines.  According to Proposition 2 in \cite{ma2020optimal}, if $\lambda_{\rm max}(M(t_0)) \leq \alpha$ ($M$ is a submatrix in the COVID-19 model outlined by a box in Eq. \eqref{COVID_19}), then there exists a positive linear combination of the quantities $x_i(t)$ that decays to zero at rate $\alpha$ starting at any time $t_0$. Thus the optimal stabilizing vaccine allocation problem can be formulated as a convex optimization problem as follows 
\begin{equation}\label{eq: vaccine initial problem}
	\begin{aligned}
		\min_{v_i}\quad & \sum_i N_i^*v_i\\
		s.t.\quad &\lambda_{\rm max}(M(t_0)) \leq -\alpha\\
		&0\leq v_i \leq s_i(t_0),~ i = 1,\ldots, n.
	\end{aligned}
\end{equation}
Let $N_ib_1(s_i(t_0) - \psi v_i) = u_i$, after some reductions (see details in SI Section \ref{sec: reduction1}),
the optimal stabilizing vaccine allocation problem \eqref{eq: vaccine initial problem} can be written as a Semidefinite Programming (SDP) problem as follows,
\begin{equation}\label{eq: SDP}
	\begin{aligned}
		\min_{u_i}\quad & - \sum_{i} u_i\\
		s.t.\quad & {\rm diag}(u_i)\preccurlyeq \bar{A}^{-1}\\
		&(1- \psi)s_i(t_0)N_ib_1 \leq u_i \leq s_i(t_0)N_ib_1,~ i = 1,\ldots, n.
	\end{aligned}
\end{equation}

From the discussion above, we can see that our method is trying to design a vaccine allocation policy to enforce decay of the infections with a prescribed decay rate by modifying the initial susceptible rate $s(t_0)$ in matrix $M(t_0)$ to meet an eigenvalue bound. This strategy is different from the traditional optimal control approaches \cite{NBERw27102,bock2018optimal,alvarez2020simple,fajgelbaum2020optimal} for the network epidemic models in the following two major ways. First, our method provides a fixed vaccine allocation policy while the traditional optimal control approaches provide time-varying policies (the policy can be different in every time $t$).  Such time-varying policies are not realistic. On the other hand, if the time-varying policy is approximated by a series of fixed allocation policies, the optimality of the approach can not be guaranteed. Second, our main result is a SDP algorithm, which is scalable. However, the traditional optimal control approaches can not guarantee the scalability or sometimes even the convergence.

\subsection{optimal stabilizing vaccine allocation design for COVID-19 model with demographic structure.} If we consider the demographic structure of the COVID-19 model, the optimal stabilizing vaccine allocation design is similar to the cases for COVID-19 model without demographic structures. We can simply replace matrix $\bar{A}$ with matrix $\bar{A}\otimes \Gamma$, replace scalar $\beta^{\text a}$, $\beta^{\text s}$ with $\beta^{\text a'}$, $\beta^{\text s'}$, then follow the same method to solve this problem. Here, a major concern is that $\bar{A}\otimes \Gamma$ is not necessarily positive definite, as the intrinsic connectivity matrix $\Gamma$ is not necessarily positive definite. If $\bar{A}\otimes \Gamma$ is not positive definite, we can not use the trick that $A \succcurlyeq B$ is equivalent to $B^{-1} \succcurlyeq A^{-1}$, and the optimal stabilizing vaccine allocation design problem can not be written as a SDP problem.

Fortunately, the contact matrix obtained by gathering empirical social contacts usually shows a pattern\cite{mossong2008social, mistry2021inferring, Brittoneabc6810, prem2017projecting} that the diagonal elements are greater than off-diagonal elements. This implies that people contact more frequently with the ones that from the same age group. When the number of age groups is small, such pattern can be strengthened, therefore it is very likely that the contact matrix is positive definite. For instance, if we divide the population into six age groups as we discussed, the contact matrix for each country (8 in total) in \cite{mossong2008social} is positive definite, so as the contact matrix for New York State in \cite{mistry2021inferring}. As $\Gamma = C {\rm diag}(N/N_i)$, matrix $\Gamma$ will be positive-definite if the contact matrix $C$ is positive definite. 

If $\Gamma$ is not positive definite, we still can formulate the optimal stabilizing vaccine allocation problem as following 
\begin{equation}\label{eq: vaccine problem general 0}
	\begin{aligned}
		\min_{v_i}\quad & \sum_i N_i^*v_i\\
		s.t.\quad & \lambda_{\rm max}({\rm diag}(s(t_0) - \psi v)Ab_1) \leq 1 \\
		&0\leq v_i \leq s_i(t_0),~ i = 1,\ldots, n.
	\end{aligned}
\end{equation}
The reduction of this problem can be found in SI Sec. \ref{sec: reduction2}.

\section{Empirical analysis} We now apply the algorithms we've developed to design an optimal stabilizing vaccine allocation policy for the 62 counties in the State of New York (NY).

\subsection{COVID-19 Model Without Demographic structures.}
First we consider the network consists of locations, where demographic structures in each county are ignored. All the parameters and data sources we employed are presented in SI Sec. \ref{sec: data covid_19}.

\textbf{Comparison with other allocation policies.} We used the data of the 62 counties in NY on Dec. 1st, 2020 as initialization and estimated the number of the new cases, cumulative cases and death cases over $500$ days  with different vaccine allocation policies. We consider two different scenarios where the number of the available vaccine doses is limited (5\% of the population) and unlimited (100\% of the population), respectively. The vaccines are supplied daily at a speed of 0.33\% of the population in NY per day, where the number 0.33\% is estimated from the data in \cite{NY_vaccine_num_tracker}. The simulation results are presented in Fig. S\ref{fig: location_curves_small} and Fig. S\ref{fig: location_curves_large}.
We compared the optimal stabilizing vaccine allocation policy calculated by our method with three other benchmark policies: (1) no vaccine: $v_i=0$ for all locations; (2) population weighted: the number of the vaccine doses allocated to location $i$ is proportional to the population of location $i$; (3) infection weighted: the number of the vaccine doses allocated to location $i$ is proportional to the number of the cumulative cases at location $i$.
It can be observed that in the two scenarios,  our policy outperforms all other polices in terms of the new cases, cumulative cases, as well as the death cases.

\textbf{Optimal vaccine rate $v_i$ (\# vaccine doses) for each county.}  Fig. S\ref{fig: location_daily_small} and Fig. S\ref{fig: location_daily_large} show the vaccine allocation rate $v_i$ and the number of vaccine doses of each county calculated by these methods we discussed, where the vaccine supply is $5\%$ and $100\%$ of the population, respectively.
It can be observed from Fig. S\ref{fig: location_daily_small}a,d that the counties in the sounthernmost of NY (mainly the counties in NYC and Long Island) are allocated with zero vaccines by our method when the vaccine supply is limited. This is a counter-intuitive result: even though the epidemics as well as the population was largely localized in the NYC and Long Island, the calculated optimal stabilizing vaccine allocation rate indicates that it is more efficient to reduce the spread of COVID-19 by allocating more vaccines to counties with smaller infections and populations. This is also quite different from the actual vaccine allocation policy\cite{NY_vaccine_tracker} applied, where the majority of the vaccines provided for NY was allocated to NYC, Long Island in the first month after the vaccines are becoming available.

There are two possible reasons for this phenomenon. First, the susceptible rates of these counties are relatively smaller than other counties (see Fig. S\ref{fig: home_rate}b), which means there are more residents of these counties that are immune to the disease. Second, according to data provided by Safegraph \cite{travel}, residents of these counties have higher values of daily home-dwell-time (see Fig. S\ref{fig: home_rate}a), which means they tend to spend longer time at home and therefore are less likely to be infected. In this case, if we give vaccine priority to the other counties in NY, it would be more efficient to curb the epidemics.

In \ssref{sec: two nodes}, we further replicate the same finding in a much simpler two-node network model: the optimal stabilizing vaccine allocation policy tends to assign zero vaccines to location with larger value of the home-dwell-time or smaller value of the initial susceptible rate. We also found that the value of $v_i^*$ is not sensitive to the population.

\textbf{Additional observations.} The effective reproduction number ($R_t$) is the average number of individuals infected by a single infected individual in the population which consists of the susceptible and non-susceptible people. It is an important metric to follow up the growth of epidemics. Meanwhile, the number of the vaccine supply and the time interval between two vaccine supplies can also impact the allocation of the vaccines.
To fully understand the effect of these parameters to the performance of the policies we discussed, we implemented additional numerical experiments. The results are shown in Fig. S\ref{fig: location_V_num_varying}, Fig. S\ref{fig: location_period_varying}, Fig. S\ref{fig: location_Rt_varying}. It can be observed that our policy outperforms all the other policies regardless of the number of the available vaccine doses, the value of $R_t$, and the time interval between two vaccine supplies.

\subsection{COVID-19 model with demographic structures} Next we consider the cases where the demographic structure of each county is considered. People from different age groups may have different values of the transmission risk $\beta^{\rm a'}(a)$, $\beta^{\rm s'}(a)$, and different values of mortality rate $\kappa(a)$. All the parameters and data sources we used are presented in SI Sec. \ref{sec: data covid_19_demo}.

\textbf{Comparison with other lockdown policies.} We still used the data of COVID-19 break in NY on Dec. 1st, 2021 as initilization and estimated the number of new, cumulative and death cases over 500 days with different vaccine allocation policies. The basic setting is similar to the simulations for the COVID-19 model without demographic structures.
The simulation results are presented in Fig. \ref{fig: demo_allocation_small} and Fig. S\ref{fig: demo_curves_large}.
%
It can be observed that our policy outperforms all the other policies no matter the vaccine supply is limited (5\%) or unlimited (100\%). Note that the mortality rate of old people in this pandemic is much higher than young people\cite{old_people,cdc_data_tracker}, however, it can be seen from Fig. \ref{fig: demo_allocation_small}d that our policy gives almost all the available vaccines to young (20-29) and middle-age adults(30-44) when the vaccine supply is limited. This is because people of these two groups have relatively higher contact rates and the transmission risks $\beta_0$ (see Fig. S\ref{fig: contact matrix}), which means it is more likely for them to transmit the disease.  Offering vaccine priority to people between 20-44  is the most efficient way to curb the pandemic, as a consequence less seniors will be infected, and less of them will die of this disease.

\textbf{Optimal vaccine rate $v_i$ (\# vaccine doses) for each county.} Fig. \ref{fig: demo_daily_small} and Fig. S\ref{fig: demo_daily_large} show the vaccine allocation rate $v_i$ and the number of the vaccine doses for all these polices we discussed, where the vaccine supply is $5\%$ and 100\% of the population, respectively. Here $v_i$ (\# vaccine doses) is the sum of vaccine rate $v_i$ (\# vaccine doses) for all six age groups of location $i$. Similar to the scenario which ignores the demographic structure, the optimal stabilizing vaccine allocation policy suggests to allocate more vaccine doses to counties outside of NYC, Long Island, while all the other policies does not show this pattern. The reason is similar as before.

\textbf{Impact of the number of the available vaccine doses.} 
To check how will the the distribution of the vaccines in the six age groups suggested by our method change with the vaccine supply changes, we varied the number of the vaccine doses from $1\%$ to $50\%$ of the population in NY, and recorded the dynamical vaccine allocation policy for each day. Then we computed the vaccine distribution in the six age groups. Note that we only observe to 50\% as the number of infections drops to 0 after around 50\% of people in NY getting vaccinated.
Fig. \ref{fig: allocation_vary_vaccine_num} shows the simulation results. We found that the order of vaccines priorities suggested by our method when the vaccine supply increases is: young adults (20-29), middle age adults (30-44), school-age children (5-19), middle age adults(45-64), seniors (65+), and preschool children (0-4). Such order is closely related to the contact intensity and the transmission risk (see Fig. S\ref{fig: contact matrix}), which decides the transmission rate. As our method is designed to maximize the decay rate of the epidemics, the vaccines will be allocated firstly to the group which has the highest value of the transmission rate.

{\bf Additional Observations.} Similar to before, we also implemented additional sensitivity analysis experiments in terms of the number of the available vaccines, the value of $R_t$ and the time interval between two vaccine supplies for COVID-19 model with demographic structures. The experimental results are shown in Fig. S\ref{fig: demo_V_num_varying}, Fig. S\ref{fig: demo_period_varying}, and Fig. S\ref{fig: demo_Rt_varying}. It can be observed that our optimal stabilizing allocation policy outperforms all the other policies regardless of the number of the available vaccines, the value of $R_t$ and the time interval between two vaccine supplies.

\subsection{Results on another model about COVID-19.} In literature \cite{bubar2021model}, the authors studied five age-stratified COVID-19 vaccine prioritization strategies based on a mathematical model, and some observations and suggestions about the prioritization strategies have been proposed. To further verify the effectiveness of our proposed method, we design an optimal stabilizing vaccine allocation policy for the model in \cite{bubar2021model} (see details in SI section \ref{sec: comparion with bubar}). Then we compare the proposed policy with the five age-stratified policies provided in \cite{bubar2021model}. All the parameters and data as well as the epidemics model we used are identical to the ones in \cite{bubar2021model}.

The basic reproduction number $R_0$ is also an important parameter to follow up the growth of the epidemics. It is very similar to $R_t$, the only difference is that $R_t$ assumes the population consists of both the susceptible and non-susceptible individuals, while $R_0$ assumes the population only consists of susceptible individuals. Since the model in \cite{bubar2021model} is designed to match the value of $R_0$, we will also consider $R_0$ here. Besides, the vaccine supply is another important parameter which determines the allocation of the vaccines. Therefore, we experimented with different values of $R_0$ and different numbers of vaccine supply.

In Fig. \ref{fig: R0_115_5_to_50}, Fig. S\ref{fig: R0_105_5_to_50} and Fig. S\ref{fig: R0_125_5_to_50}, we show the estimated percentage of infected cases as well as the cumulative mortality cases in the population over 500 days, where $R_0 = 1.15$, $R_0 = 1.05$ and $R_0 = 1.25$, and the vaccine supply is $5\%$, $20\%$, and $50\%$, respectively. Meanwhile, we also show the distribution of vaccines provided by our method and the five age-stratified strategies from \cite{bubar2021model} in Fig. \ref{fig: R0_115_5_to_50}, Fig. S\ref{fig: R0_105_5_to_50} and Fig. S\ref{fig: R0_125_5_to_50}. It can be seen that our method outperforms all the five age-stratified prioritization strategies in \cite{bubar2021model}. Moreover, we can see that the vaccine distribution suggested by our method is different from any of the distributions in \cite{bubar2021model}. Particularly, we can observe that our method suggests to offer vaccine priority to adults between 30-40 when the vaccine supply is small (5\%). This is because people in this group have the highest value of transmission rate, allocating vaccines to them firstly can help curb the epidemics fast. Paradoxically, this also causes fewer deaths than giving vaccines to elderly people.

Until now, vaccine supply is no longer an issue in the United States. Therefore, we also experimented with unlimited vaccine supplies (100\% vaccine supply). As our method is designed to minimize the decay rate, and the infections would drop to 0 before using up all the vaccines, we will allocate the leftover vaccines evenly to all the age-groups after the vaccines disappears. Fig. \ref{fig: R0_105_115_125_100_percent}, Fig. S\ref{fig: R0_100_110_120_100_percent}, and Fig. S\ref{fig: R0_130_135_100_percent} show the experimental results with $R_0 = 1.0$, 1.05, 1.10, 1.15, 1.20, 1.25, 1.30, 1.35, respectively. It can be seen that our method outperforms all the five age-stratified prioritization strategies in \cite{bubar2021model} when $ 1.0 < R_0 < 1.30$. When $R_0 = 1.0$, our method still outperforms the all the strategies in \cite{bubar2021model} in terms of the infected cases, 
the estimated mortality cases for our method is close to the strategy which offers vaccine priority to elderly people. When $R_0 > 1.30$, the prioritization strategy for seniors is the best when consider the mortality cases.

In summary, if $R_0$ is at one or moderately above one -- which describes most of the world during COVID \cite{global} -- it's better to offer vaccine priorities to young people to curb disease spread in every way.
Paradoxically, this also causes fewer deaths than offering vaccine priorities to older people.

\begin{figure}[!htb]
	\centering
	\begin{minipage}[b]{0.95\linewidth}
		\centering
		\includegraphics[width=1.0\linewidth]{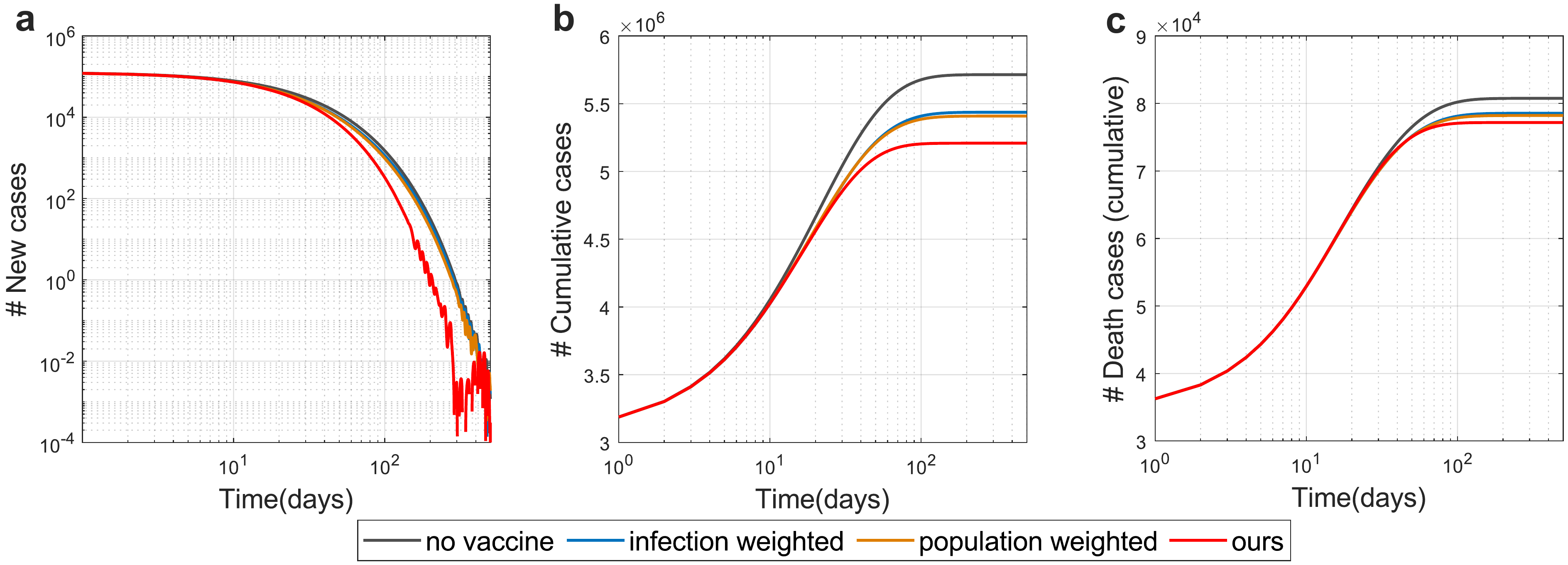}			
	\end{minipage}\hfill

	\begin{minipage}[b]{0.6\linewidth}
		\centering
		\includegraphics[width=1.0\linewidth]{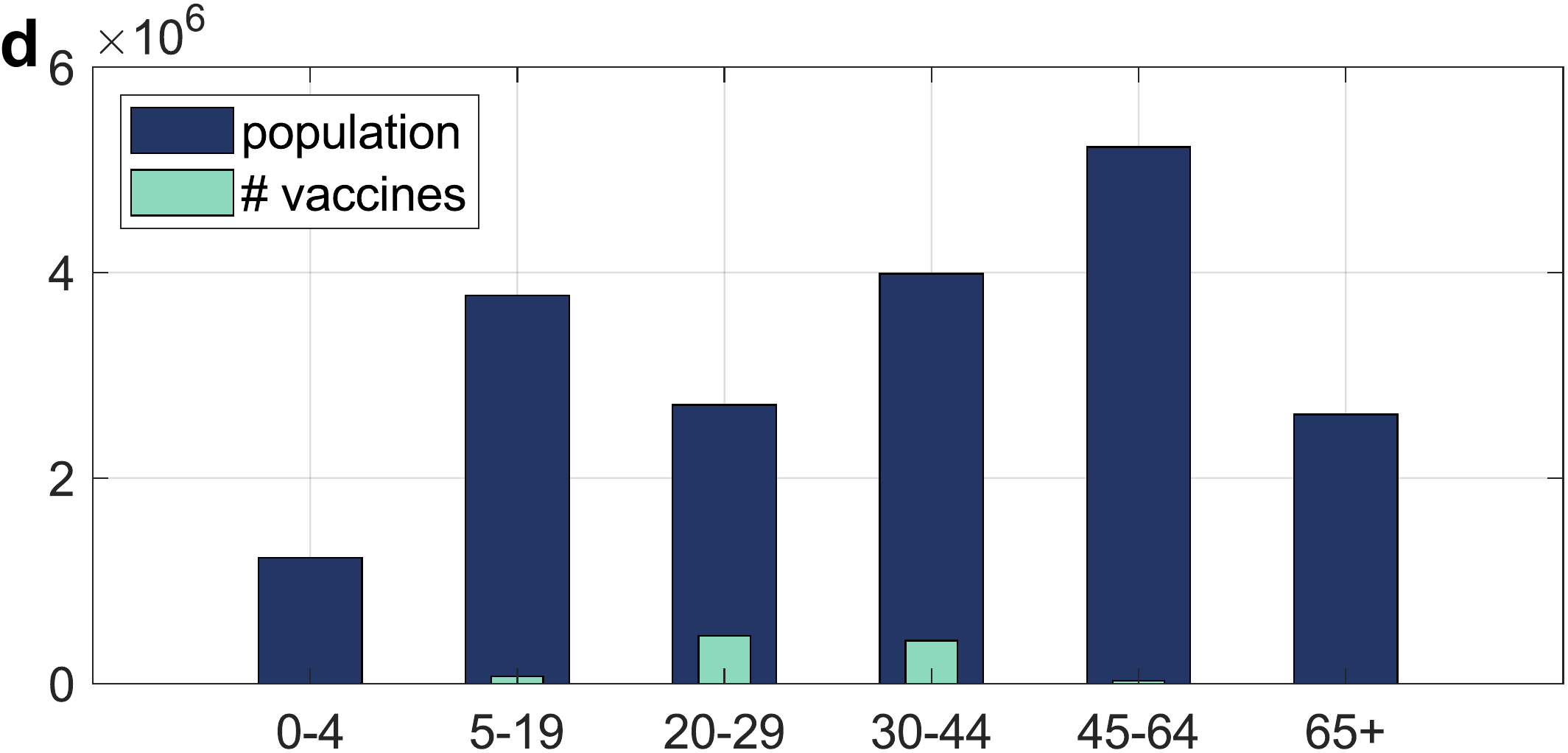}
	\end{minipage}%
	\centering
	\caption{{\bf When the vaccine supply is limited, our method suggests to allocate almost all the vaccines to young adults in 20-44, such strategy outperforms all the other comparison methods.}  {\bf a}, the estimated number of daily new cases.  {\bf b}, the estimated number of cumulative cases. {\bf c}, the estimated number of cumulative death cases. {\bf d}, the optimal stabilizing vaccine allocation number for each age group calculated by our method. The vaccines are supplied daily with
	a speed of $0.33\%$ of total population per day. The total number of the available vaccine doses for these policies are the same, that is $5\%$ of the total population. ``Population weighted" implies the policy where the number of the vaccine doses allocated to county $i$ is proportional to the population $N_i$. ``infection weighted" implies the policy where the number of the vaccine doses allocated to county $i$ is proportional to the cumulative cases at location $N_i$. ``No vaccine" implies the policy where the vaccines are not applied. The data applied is about COVID-19 outbreak in NY on December 1st, 2020.
	}\label{fig: demo_allocation_small}
\end{figure}
\begin{figure}[!htb]
	\centering
	\begin{minipage}[b]{0.95\linewidth}
		\centering
		\includegraphics[width=1.0\linewidth]{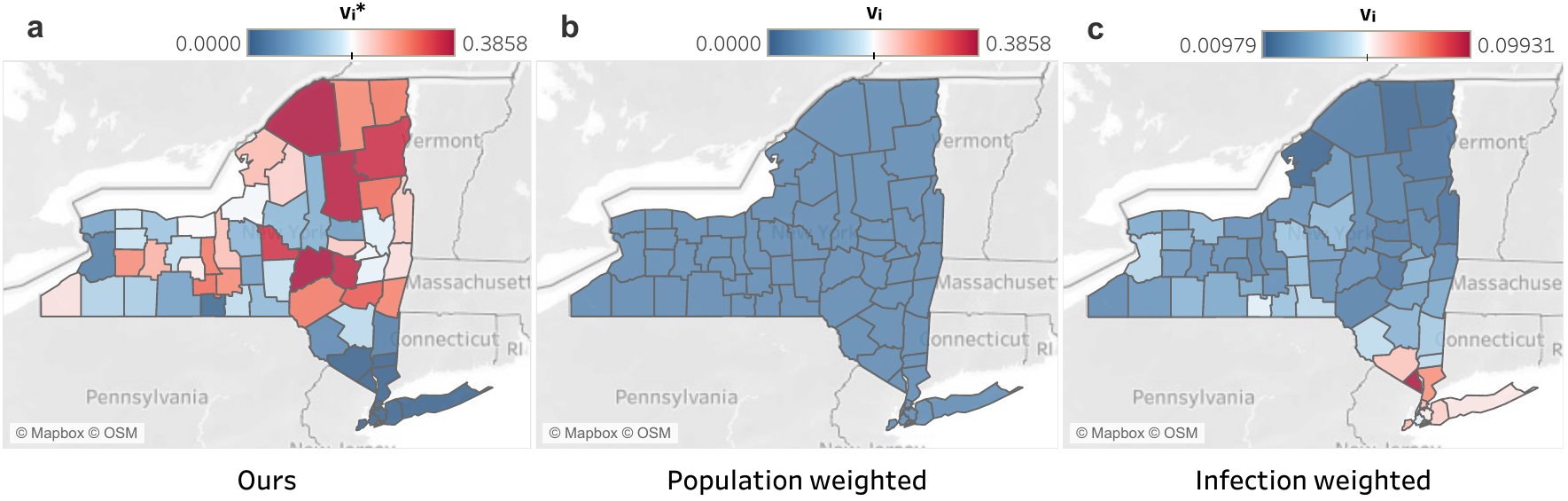}			
	\end{minipage}\hfill
	
	\vspace{3mm}
	
	\begin{minipage}[b]{0.95\linewidth}
		\centering
		\includegraphics[width=1.0\linewidth]{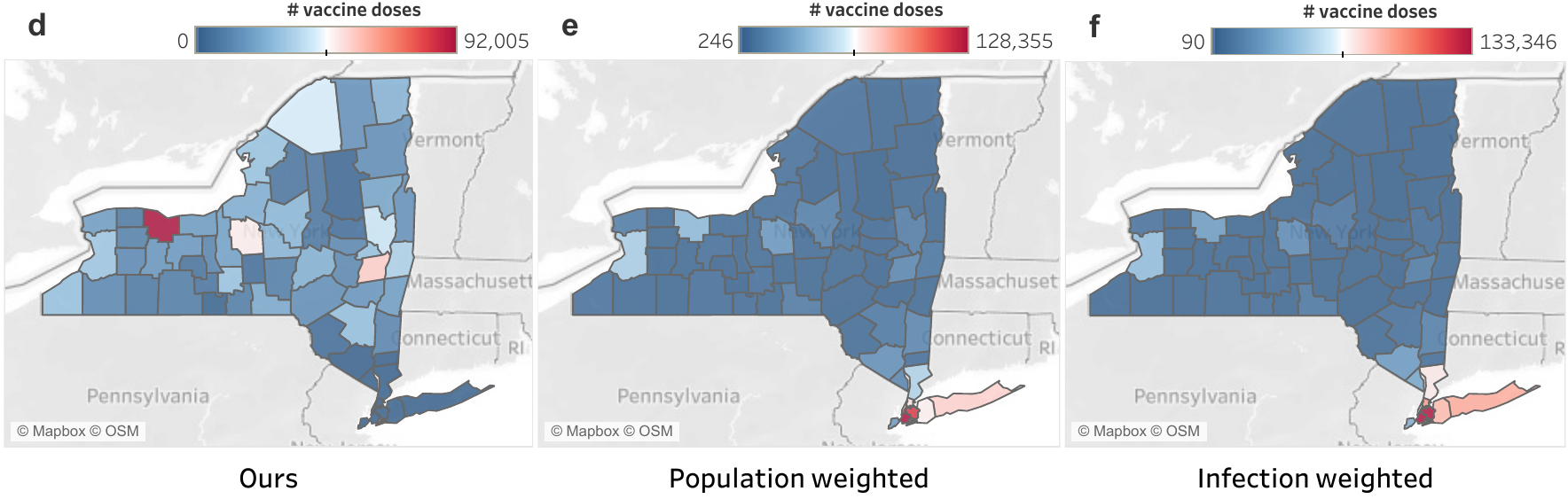}
	\end{minipage}%
	\centering
	\caption{ {\bf When vaccine supply is limited, our method suggests to allocate zero vaccines to counties in NYC and Long Island.} {\bf a-c}, vaccine allocation rate $v_i$ of each county given by different policies. {\bf d-f}, the number of vaccine doses allocated to each county by different policies. The vaccines are supplied daily with
		a speed of $0.33\%$ of total population per day. The total number of the available vaccine doses for these policies are the same, that is $5\%$ of the total population. ``Population weighted", ``infection weighted", and ``no vaccine" polices are defined as in Fig. \ref{fig: demo_allocation_small}. The data applied is about COVID-19 outbreak in NY on December 1st, 2020. The values of $v_i^*$, $v_i$ in this figure corresponds to the results in Fig. \ref{fig: demo_allocation_small}.
	}\label{fig: demo_daily_small}
\end{figure}
\begin{figure}[!htb]
	\includegraphics[width=1.0\linewidth]{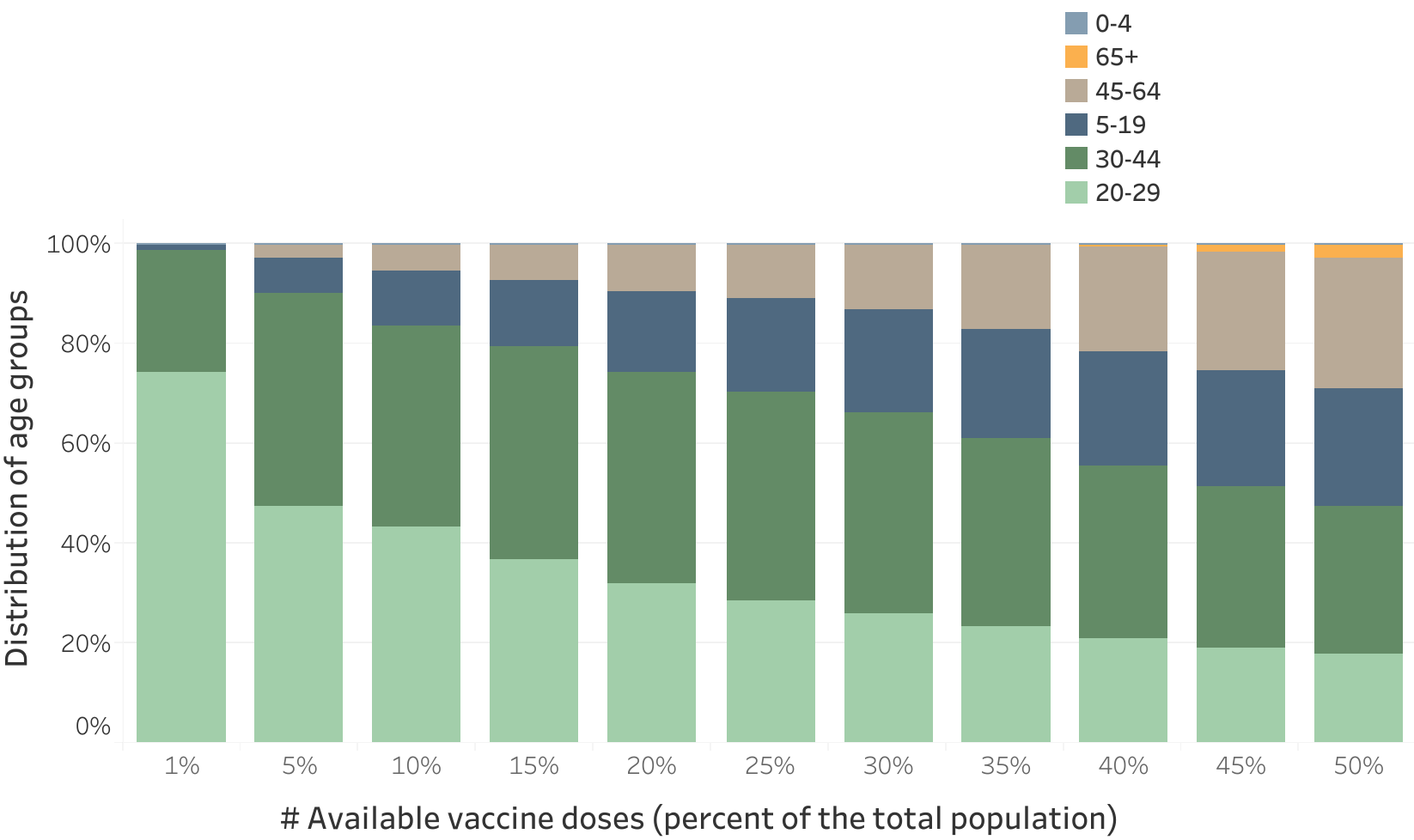}
	\centering
	\caption{{\bf The order of the vaccine priority suggested by our method when the vaccine supply increases is: 20-19, 30-44, 5-19, 45-64, 65+, 0-4.} The vaccines are supplied daily with a speed of $0.33\%$ of total population per day. The total number of the available vaccine doses ranges from 1\% to 50\% of the population. We only observe to $50\%$ as the epidemic disappears after supplying around $50\%$ people with vaccines. ``Population weighted" and ``infection weighted" polices are defined as in Fig. \ref{fig: demo_allocation_small}. The data applied is about COVID-19 outbreak in NY on December 1st, 2020.  It can be observed that the optimal stabilizing vaccine allocation policy suggests to provide the vaccines to young adults (20-29) and middle age adults (30-44) firstly. The reason is that people of these two groups have relatively higher values of contact rate and transmission risk,  therefore they will be more likely to spread the epidemic. Thus giving vaccine priority for these groups can curb the epidemic most efficiently.
	}\label{fig: allocation_vary_vaccine_num}
\end{figure}


\begin{figure}[!htb]
	\centering
	\begin{minipage}[b]{0.22\linewidth}
		\centering
		\includegraphics[width=1.0\linewidth]{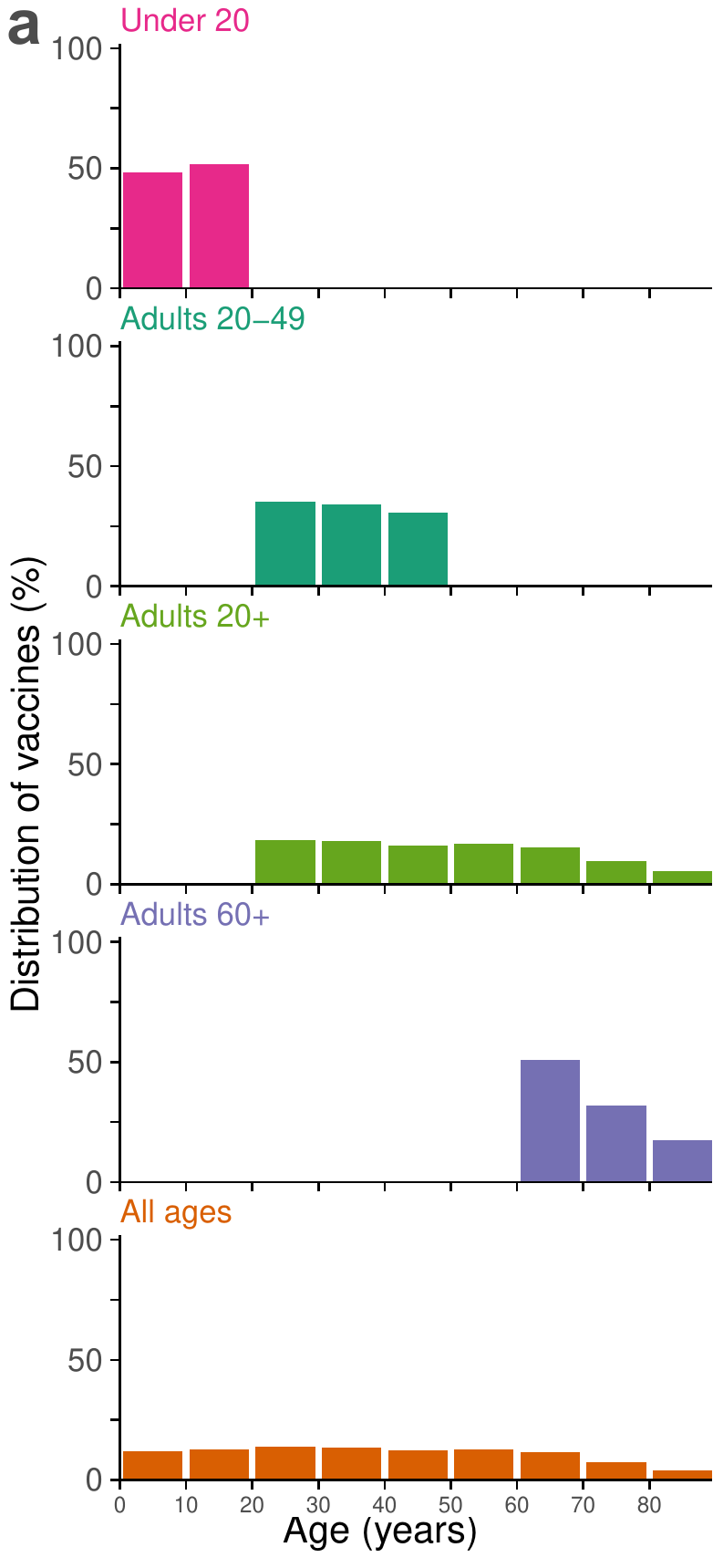}			
	\end{minipage}\hfill
	\begin{minipage}[b]{0.78\linewidth}
		\centering
		\includegraphics[width=1.0\linewidth]{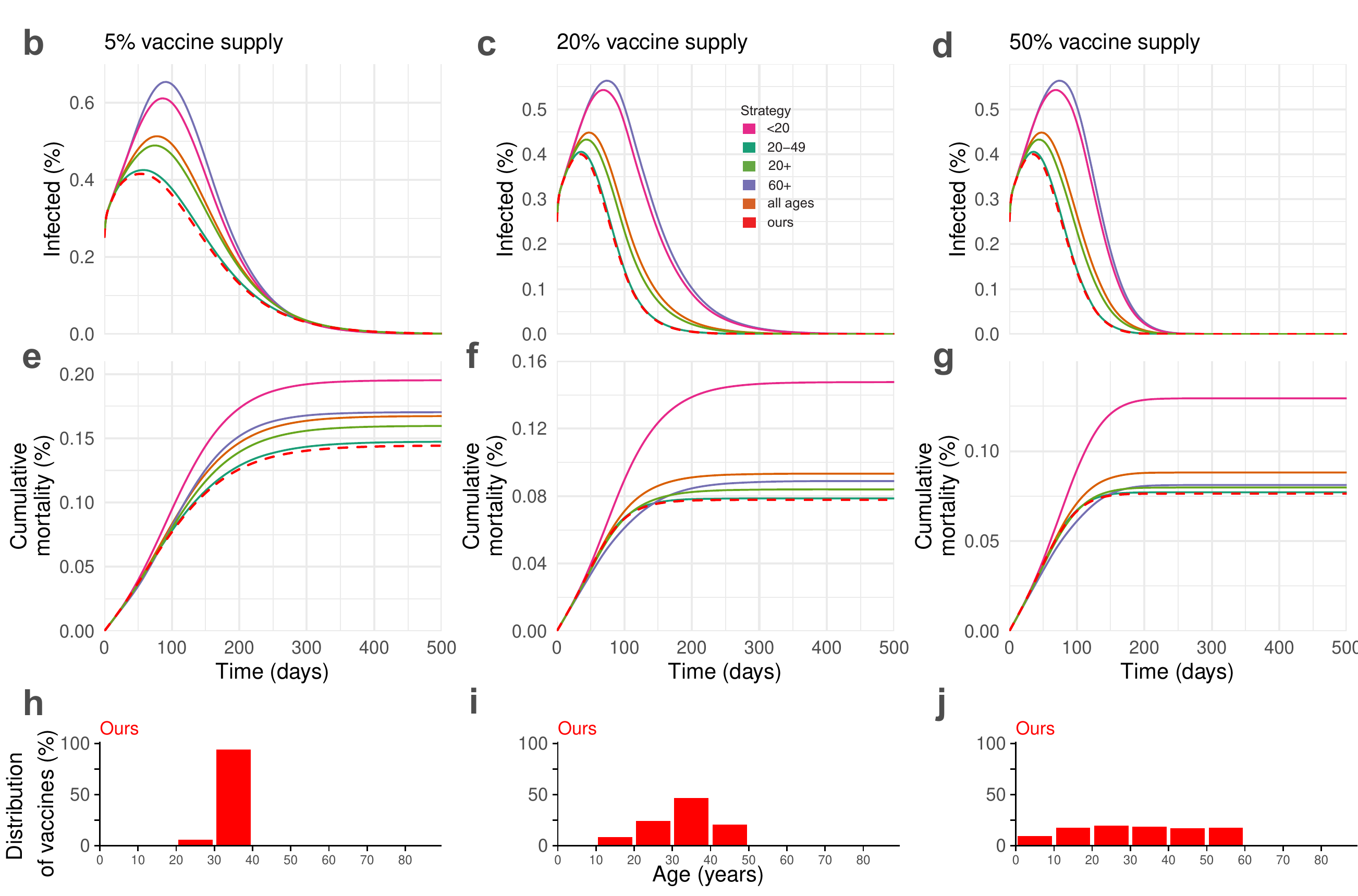}
	\end{minipage}%
	\centering
	\caption{{\bf Our method outperforms all the five age-stratified prioritization strategies in \cite{bubar2021model} when $R_0 = 1.15$.} {\bf a}, distribution of vaccines for the five age-stratified prioritization policies in \cite{bubar2021model}. {\bf b-g}, the estimated percentage of the infected cases and the cumulative mortality cases in the total population over 500 days. {\bf h-j}, distribution of vaccines of our method. In {\bf b-g}, the colors of the lines match with the polices, the red dashed line represents our policy. In {\bf b, e, h}, the total number of the available vaccine doses is 5\% of the population. In {\bf c, f, i}, the total number of the available vaccine doses is 20\% of the population. In {\bf d, g, j}, the total number of the available vaccine doses is 50\% of the population. All these results based on data and parameters for United States in \cite{bubar2021model}. The vaccines are supplied at 0.2\% of the total population per day. The vaccines are assumed to be all-or-nothing, transmission-blocking with 90\% efficacy.
	}\label{fig: R0_115_5_to_50}
\end{figure}

\begin{figure}[!htb]
	\begin{minipage}[b]{1.0\linewidth}
		\centering
		\includegraphics[width=1.0\linewidth]{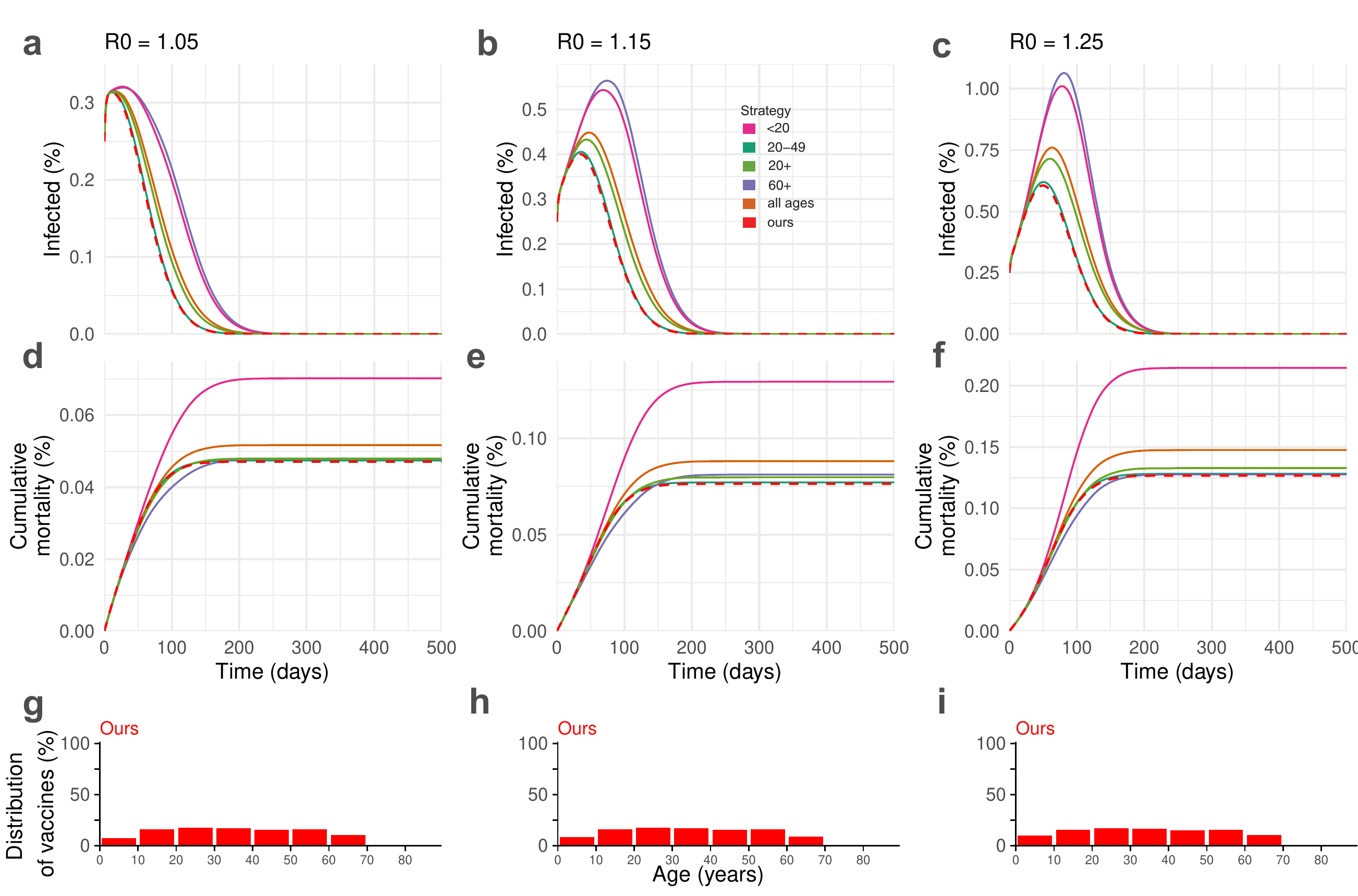}
	\end{minipage}%
	\centering
	\caption{{\bf If the vaccine supply is unlimited, and $R_0$ = 1.05, 1.15, or 1.25, our method outperforms all the five age-stratified prioritization strategies in \cite{bubar2021model}.} {\bf a-f}, the estimated percentage of the infected cases and the cumulative mortality cases in the total population over 500 days. {\bf g-i}, the distribution of vaccines of our method. In {\bf a-f}, the colors of the lines match with the polices, the red dashed line represents our policy. In {\bf a, d, g}, $R_0$ = 1.05. In {\bf b, e, h}, $R_0$ = 1.15. In {\bf c, f, i}, $R_0$ = 1.25. In this figure, the total vaccine supply is assumed to be 100\% of the population, in {\bf g-i}, we only show the distribution of vaccines before the disappearing of the epidemics, after that, the leftover vaccines would be evenly distributed to these age groups.  All these results based on data and parameters for United States in \cite{bubar2021model}. The vaccines are supplied at 0.2\% of the total population per day. The vaccines are assumed to be all-or-nothing, transmission-blocking with 90\% efficacy.
	}\label{fig: R0_105_115_125_100_percent}
\end{figure}

\clearpage

\noindent {\bf Author contributions.} All authors designed and did the research. Q.M. performed all the calculations and wrote the manuscript. Y.-Y.L and A.O. edited the manuscript.  
\bigskip

\noindent {\bf Competing interests statement.} The authors declare no competing interests. 

\nocite{*}
\bibliography{pandemic}{}
\bibliographystyle{abbrv}


\clearpage

\begin{center}
  \vskip 60pt
\large \MakeUppercase{\textbf{Optimal Stabilizing Vaccine Allocation for Epidemic Control}}\\
\MakeUppercase{\textbf{---Supplementary Information---}} 
  \vskip 1em
  \large QIANQIAN MA\footnote{Department of Electrical and Computer Engineering, Boston
University, Boston, MA USA},  \ 
YANG-YU LIU\footnote{Channing Division of Network Medicine, Department of Medicine, Brigham and Women's
Hospital, Harvard Medical School, Boston, MA 02115, USA}, 
and ALEX OLSHEVSKY \footnote{Department of Electrical and Computer Engineering and Division of
Systems Engineering, Boston University, Boston, MA USA}
  \vskip 20pt
\end{center}

\setcounter{section}{0}
\setcounter{figure}{0}

\section{Related work}
Our work is related to a number of recent papers motivated by the spread COVID-19, as well as some older work. This paper builds on our previous work \cite{ma2020optimal}, where we considered the problem of sending the epidemics to 0 with a specified decay rate while minimizing the economic cost. We used the same COVID-19 model as in \cite{ma2020optimal}, besides, we fix the decay rate of epidemics in both papers. There are two major differences between these two works. First, we aim to control the epidemics through the vaccines in this paper while in \cite{ma2020optimal} we used the lockdown policy. Second, we consider the demographic structure of the population in this work which was ignored in our previous work \cite{ma2020optimal}. 

 Our work has some similarities with the recent paper \cite{enayati2020optimal}, which considered the optimal vaccine distribution issue in a heterogeneous population with multiple age groups. They solved this issue by minimizing the number of the vaccines used while enforcing the effective reproduction number $R_t $ be bounded by 1. This is similar to what we study. The major difference is that they write this problem to a bilinear program,  and solved this using an iterating algorithm which involves two mixed-integer programs in each iteration. While in this paper, we wrote a similar problem as a SDP problem, which can be solved more efficiently.
 
 The recent work \cite{bubar2021model} studied the vaccine prioritization issue for different age groups, which is similar to the problem we consider. In this work, the authors estimated the cumulative cases, death cases of  several simple prioritization strategies based on SEIR model. One of their observations is similar to us, i.e., giving vaccines priority to young people can minimize the cumulative cases. The difference is that we also observed that the death cases can be minimized by giving young people vaccines firstly if  the number of the available vaccines is not very large. Besides, in work \cite{bubar2021model}, the strategy they consider is to give vaccines to a specific age group population, while in our case, we allow different age groups have different fraction of vaccines and we have proved our policy can achieve minimum decay rate of the epidemics.

\section{Mathematical Background} 
A matrix is called continuous time stable if all of its eigenvalues have nonpositive real parts. A matrix is called discrete time stable if all of its eigenvalues are upper bounded by one in magnitude. A central concern of this paper is to get certain quantities of interest (e.g., number of infected individuals) to decay at prescribed exponential rates. We will say that {\em $y(t)$ decays at rate $\alpha$ beginning at $t_0$} if $y(t) \leq y(t_0) e^{-\alpha t}$ for all $t \geq t_0$. Note that the decay in this definition  is not asymptotic but results in a decrease starting  at time $t_0$. 

We will associate to every matrix $A \in \R^{n \times n}$ the graph $G(A)$ corresponding to its nonzero entries: the vertex set of $G(A)$ will be $\{1, \ldots, n\}$ while $(i,j) \in G(A)$ if and only if $A_{ji}\neq 0$. Informally, $(i,j)$ is an edge in $G(A)$ when the variable $j$ ``is influenced by'' variable $i$. We will say that $A$ is strongly or weakly
connected if the graph $G(A)$ has this property.

\section{Analytical calculations}

\subsection{Construction of matrix $A'$}\label{sec: construction A'}
\quad

In this section, we present how we constructed the matrix $A'$ in the COVID-19 model \eqref{eq: covid 19, demographic} with demographic structures. First, we consider the infection rate of a susceptible individual of age group $a$ at location $i$. This individual can travel to any location of the network we consider, the fraction of the time this individual spent at location $l$ in a day is $\tau_{il}$. During this period, the individual may contact people from different age groups, suppose $C^l_{ab}$ is the mean number of contacts that this individual has with people of age group $b$ during a day at location $l$, $\beta^{\text s'}_i(a)$ is the probability with which a meeting between a susceptible individual of group $a$ from location $i$ and an symptomatic individual results in an infection. Then the rate that this individual be infected by symptomatic individuals of age group $b$ at location $l$ is proportional to $\tau_{il}C^l_{ab}\beta^{\text s'}_i(a)$.

Moreover, since the population of age group $b$ at location $l$ consists of healthy people and infected people. It is natually to assume that the infection rate is proportional to the fraction of infected people in this specific population, i.e., 
\[\frac{\hat{\alpha} \sum_{j = 1}^n N_j^*(b) \tau_{jl}x_j^{\rm a}(b) + \sum_{j = 1}^n N_j^*(b) \tau_{jl} x_j^{\rm s}(b)}{\sum_{k = 1}^{n} N_k^*(b)\tau_{kl}}, \]
where $\beta^{\text a'} = \hat{\alpha} \beta^{\text s'}$, $\hat{\alpha}$ is the discount factor, which captures the reduced risk of infection in meetings between suaceptible individual and asymptomatic (compare to symptomatic individual) individual, $N_j^*(b)$ is the population of the long-term residents of group $b$ at location $j$. 

As this individual can travel to any location of the network and can be infected by meeting with people from any age groups, then we can rewrite $\dot{x}^{\text a}_i$ in  \eqref{eq:COVID} as
\begin{equation}\label{eq: covid_19 demo 1}
	\dot{x}^{\text a}_i(a) = \sum_{l = 1}^n \sum_{b = 1}^{6}s_i(a) \tau_{il} \frac{\hat{\alpha} \sum_{j = 1}^n N_j^*(b) \tau_{jl}x_j^{\rm a}(b) + \sum_{j = 1}^n N_j^*(b) \tau_{jl} x_j^{\rm s}(b)}{\sum_{k = 1}^{n} N_k^*(b)\tau_{kl}}C_{ab}^l \beta^{\text s'}_i(a) - (\epsilon + r^{\text a})x_i^{\text a}.
\end{equation}

Next we consider the issues of projecting contact matrices to different demographic stuctures. The contact rate between different age groups are usually obtained via gathering empirical social contacts. Such empirical contact matrix is measured on a specific population, and should not be used directly. Literature \cite{arregui2018projecting} provided a method to transform the empirical contact matrix $C_{ij}$, which is measured for a specific demographic structure $N_i$, into a contact matrix $C'_{ij}$ that is compatible with a different demographic structure $N_i'$ as follows
\begin{equation}\label{eq: contact matrix}
	C'_{ij} = C_{ij} \frac{NN_j'}{N_jN'}.
\end{equation}
In the COVID-19 model, the demographic structure of each location is location-dependent, we will employ \eqref{eq: contact matrix} to construct the time-varying contact matrix for each location.

Suppose $\bar{C}$ is the empirical contact matrix with demographic structure $\bar{N}$. By using \eqref{eq: contact matrix}, the middle term inside the sum of \eqref{eq: covid_19 demo 1} can be rewritten as
\begin{equation}\label{eq: contact matrix mid term}
	\begin{aligned}
& \frac{\hat{\alpha} \sum_{j = 1}^n N_j^*(b) \tau_{jl}x_j^{\rm a}(b) + \sum_{j = 1}^n N_j^*(b) \tau_{jl} x_j^{\rm s}(b)}{\sum_{k = 1}^{n} N_k^*(b)\tau_{kl}}C_{ab}^l \\
	&=  \frac{\hat{\alpha} \sum_{j = 1}^n N_j^*(b) \tau_{jl}x_j^{\rm a}(b) + \sum_{j = 1}^n N_j^*(b) \tau_{jl} x_j^{\rm s}(b)}{\sum_{k = 1}^{n} N_k^*(b)\tau_{kl}} \bar{C}_{ab}\frac{\bar{N}}{\bar{N}_b}\frac{\sum_{k = 1}^{n} N_k^*(b)\tau_{kl}}{\sum_{b= 1}^6\sum_{k = 1}^{n} N_k^*(b)\tau_{kl}}\\
	&=  \bar{C}_{ab}\frac{\bar{N}}{\bar{N}_b} \frac{\hat{\alpha} \sum_{j = 1}^n N_j^*(b) \tau_{jl}x_j^{\rm a}(b) + \sum_{j = 1}^n N_j^*(b) \tau_{jl} x_j^{\rm s}(b)}{\sum_{b= 1}^6\sum_{k = 1}^{n} N_k^*(b)\tau_{kl}}
	\end{aligned}
\end{equation}

Let
\[
\bar{C}_{ab} \frac{\bar{N}}{\bar{N}(b)} = \Gamma_{ab},\quad
\sum_{b = 1}^{6}\sum_{k = 1}^{n} N_k^*(b)\tau_{kl} = m(l),
\]
then combine \eqref{eq: covid_19 demo 1}, \eqref{eq: contact matrix mid term} together, we have
\begin{equation}\label{eq: covid_19 demo 2}
	\begin{aligned}
			\dot{x}^{\text a}_i(a) &= \sum_{l = 1}^n \sum_{b = 1}^{6}s_i(a) \tau_{il}  \Gamma_{ab} \frac{\sum_{j = 1}^n \tau_{jl}(\hat{\alpha}x_j^{\rm a}(b) + x_j^{\rm s}(b))N_j^*(b)}{m(l)}  \beta^{\text s'}_i(a) - (\epsilon + r^{\text a})x_i^{\text a}\\
		&= s_i(a)\sum_{b = 1}^{6}\sum_{j = 1}^n \Gamma_{ab}\bar{A}_{ij}N_j^*(b)(\hat{\alpha}x_j^{\rm a}(b) + x_j^{\rm s}(b))\beta^{\text s'}_i(a) - (\epsilon + r^{\text a})x_i^{\text a},
	 \end{aligned}
\end{equation}
where $\bar{A}_{ij} = \sum_{l = 1}^{n} \frac{\tau_{il}\tau_{jl}}{m(l)}$. 

In matrix form, we can write \eqref{eq: covid_19 demo 2} as
\begin{equation*}
\dot{x}^{\text a} = {\rm diag}(\beta^{\text a '}){\rm diag}(s) A'x^{\text a} + {\rm diag}(\beta^{\text s '}){\rm diag}(s) A'x^{\text s} - (\epsilon + r^{\text a}) x^{\text a},
\end{equation*}
where
\textsf{\begin{equation*}
	A' = (\bar{A} \otimes \Gamma) {\rm diag}(N^*).
\end{equation*}}

\textbf{An illustrtive example.} To clearly demonstrate the construction of the matrix $A'$, we provide a numerical example on a small network consists of two nodes. Suppose the network consists of node 1 and node 2, and the population of each node belongs to either age group $a$ or age group $b$.  Let the data of the population be
\[N_1^*(a) = 80, ~ N_1^*(b)= 20, ~ N_2^*(a) = 100, N_2^*(b) = 100,\]
the trvel rate matrix $\tau$ be 
\[\tau = \left(\begin{matrix} 0.4 & 0.1\\
	0.1 & 0.4 \end{matrix} \right ),\]
the intrinsic connectivity matrix $\Gamma$ be
\[\Gamma = \left(\begin{matrix} 20 & 2\\
	2 & 4 \end{matrix} \right ).\]
In this case,  we have 
\begin{equation*}
	\begin{aligned}
		m(1) &= 80 \times 0.4 +20 \times 0.4 + 100\times 0.1 + 100 \times 0.1=60,\\
		m(2) &= 80 \times 0.1 +20 \times 0.1 + 100\times 0.4 + 100 \times 0.4=90.\\
	\end{aligned}
\end{equation*}
Next we can compute the matrix $\bar{A}$ as
\begin{align*}
\bar{A}
	=
	\begin{pmatrix} 
		 \frac{0.4\times 0.4}{60} +  \frac{0.1\times 0.1}{90} & \frac{0.4\times 0.1}{60} +  \frac{0.1\times 0.4}{90} \\[0.5cm]
		\frac{0.4\times 0.1}{60} +  \frac{0.1\times 0.4}{90} & \frac{0.1\times 0.1}{60} +  \frac{0.4\times 0.4}{90}
	\end{pmatrix} = \frac{1}{180}\begin{pmatrix} 
0.5 & 0.2 \\[0.3cm]
0.2 & 0.35
\end{pmatrix}.
\end{align*}
Thus we can obtain the matrix $A'$ as
\begin{align*}
A' &= (\bar{A} \otimes \Gamma) {\rm diag}(N^*)\\
&=\frac{1}{180}\begin{pmatrix}
10 & 1 & 4 & 0.4\\
1 & 2 & 0.4 & 0.8 \\
4 & 0.4 & 7 & 0.7\\
0.4 & 0.8 & 0.7 & 1.4
\end{pmatrix} \begin{pmatrix} 
	80 &  &  & \\
	 & 20 &  & \\
	 &  &  100 & \\
	 &  &  & 100
\end{pmatrix}\\
& = \frac{1}{9}\begin{pmatrix}
	40 & 1 & 20 & 2 \\
	4 & 2 & 2 & 4\\
	16 & 0.4 & 35 & 3.5 \\
	1.6 & 0.8 & 3.5 & 7
\end{pmatrix}.
\end{align*}

\subsection{The reduction of the optimal stabilizing vaccine allocation problem \eqref{eq: vaccine initial problem}.}\label{sec: reduction1}\quad 

In this section, we present the details about how to reduce the problem \eqref{eq: vaccine initial problem} to a SDP problem \eqref{eq: SDP}. Apply the result in the middle of proof for Lemma 11 in \cite{ma2020optimal}, the first constraint in \eqref{eq: vaccine initial problem} is equivalent to 
\begin{equation} \label{eq: vaccine constraint 1}
	{\rm diag}(s(t_0) - \psi v)Ab_1\text{ is discrete-time stable.} 
\end{equation}
where $b_1 = \frac{\beta^{\text s}\epsilon + \beta^{\text a}(r^{\text s} - \alpha)}{(\epsilon + r^{\text a} - \alpha)(r^{\text s} - \alpha)}$. Let $A = \bar{A} {\rm diag}(N_i^*)$, where $\bar{A}$ is defined as in \eqref{eq: bar_A}, then constraint \eqref{eq: vaccine constraint 1} can be written as 
\[{\rm diag}(N^*_i)b_1{\rm diag}(s(t_0) - \psi v)\bar{A} \text{ is discrete-time stable,} \]
since the nonzero eigenvalues of a product of two matrices do not change after we change the order in which we multiply them. Next, apply part 3 of Lemma 7 in \cite{ma2020optimal}, we can write this constraint as
\begin{align}\label{eq: vaccine constrain 2}
	\bar{A} - {\rm diag} \left(\frac{1}{N_i^*b_1(s_i(t_0) - \psi v_i)}\right)\text{ is continuous-time stable.}   
\end{align}
To further simplify this constraint, we use the fact that $A \succcurlyeq B$ is equivalent to $B^{-1} \succcurlyeq A^{-1}$ if both matrices $A$ and $B$ are positive definite. Apparently,  $\bar{A}$ is a positive definite matrix. Then by using this proposition, we can write constraint \eqref{eq: vaccine constrain 2} as
\[{\rm diag}\left(N_i^*b_1(s_i(t_0) - \psi v_i)\right) - \bar{A}^{-1} \text{ is continuous-time stable}.\]
Thus, problem \eqref{eq: vaccine initial problem} can be denoted as
\begin{equation*}
	\begin{aligned}
		\min_{v_i}\quad & \sum_i N_i^*v_i\\
		s.t.\quad & {\rm diag}\left(N_i^*b_1(s_i(t_0) - \psi v_i)\right)\preccurlyeq \bar{A}^{-1}\\
		&0\leq v_i \leq s_i(t_0),~ i = 1,\ldots, n.
	\end{aligned}
\end{equation*}

\subsection{The reduction of the optimal stabilizing vaccine allocation problem \eqref{eq: vaccine problem general 0}}\label{sec: reduction2}~

By using part 2 of Lemma 7 in \cite{ma2020optimal}, we can write problem \eqref{eq: vaccine problem general 0} as the following bilinear programming problem
\begin{equation}\label{eq: vaccine problem general 1}
	\begin{aligned}
		\min_{v_i}\quad & \sum_i N_i^*v_i\\
		s.t.\quad & ({\rm diag}(s(t_0) - \psi v)Ab_1) d \leq d, \\
		& \sum_{i = 1}^{n} d_i =  1,\\
		&0\leq v_i \leq s_i(t_0),~ i = 1,\ldots, n,\\
		& d_i \geq 0, ~ i = 1,\ldots, n.
	\end{aligned}
\end{equation}
This problem has been studied in literature\cite{enayati2020optimal}, which solved the issue by using two mixed-integer programs. Problem \eqref{eq: vaccine problem general 1} can also be solved by using some nonlinear programming solver like SNOPT\cite{gill2005snopt}, BARON\cite{ts:05}, and GUROBI\cite{gurobi}.

\section{Empirical data analysis}\label{sec: data sources}
\subsection{Data \& parameters for COVID-19 model}\label{sec: data covid_19}
\quad

\textbf{Initial rates ($s(t_0), x^{\text a}(t_0), x^{\text s}(t_0), e(t_0), h(t_0)$).} We get the cumulative confirmed cases in county level of NY on Dec. 1st, 2020 from the dataset provided in \cite{susceptible}. Besides, according to CDC data\cite{reporting_rate}, 1 in 4.6 total COVID-19 infections were reported (or 0.217 reporting factor). Therefore we let the number of the cumulative confirmed cases of county $i$ be $\frac{I_i(t_0)}{0.217}$, where $I_i(t_0)$ represents the number of the cumulative confirmed cases of county $i$ at time point $t_0$. Moreover, we get the number of cumulative death cases in each county of NY on Dec. 1st, 2020  from New York Times \cite{death_rate}. We use $E_i(t_0)$ to denote the cumulative death cases of county $i$. For the truly recovered people of the COVID-19, unfortunately, we can not find specific numbers for each county in New York State. We also can not get the specific number of the 
asymptomatic cases and symptomatic cases of each county.

However, we learn from \cite{total_recovery} that the total number of recovered cases, death cases, and cumulative cases in USA on Dec. 1st, 2020 are 8333018, 276976, 14108490, respectively. Since the cumulative cases of each county also consists of the recovered cases, death cases, and active cases, we assume the ratio of the recovered cases to the sum of the recovered cases and the active cases of county $i$ is proportional to the number $8333018/(14108490 - 276976)$. Then we have
\[ H_i(t_0) = \left(\frac{I_i(t_0)}{0.217} - E_i(t_0)\right)\frac{8333018}{(14108490-276976)}.\]

Moreover, we learn from CDC \cite{symptom_rate} that 81\% COVID-19 patients have mild to moderate symptoms and 19\% patients have severe to critical symptoms. We also assume that the number of the asymptomatic cases and the symptomatic cases of county $i$ satisfies such ratio. Then we can get
\[X_i^{\rm a}(t_0) =  0.81\left(\frac{I_i(t_0)}{0.217} - E_i(t_0) - H_i(t_0)\right),\]
and
\[X_i^{\rm s}(t_0) =  0.19\left(\frac{I_i(t_0)}{0.217} - E_i(t_0) - H_i(t_0)\right),\]
where $X_i^{\rm a}(t_0)$, $X_i^{\rm s}(t_0)$ denotes the number of the asymptomatic cases and the symptomatic cases of county $i$ respectively. Finally, we let the initial rates of county $i$ be 
\begin{align*}
	s_i(t_0) &= 1- \frac{I_i(t_0)}{0.217N_i},\\
	e_i(t_0) & = \frac{E_i(t_0)}{N_i},\\
	h_i(t_0) & = \frac{H_i(t_0)}{N_i},\\
	x_i^{\rm a}(t_0) &= \frac{X_i^{\rm a}(t_0)}{N_i},\\
	x_i^{\rm s}(t_0) &= \frac{X_i^{\rm s}(t_0)}{N_i}.
\end{align*}

\textbf{Populations ($N_i$).} To define the populations of each node in the network, we adopt the 2010 Census Bureau data \cite{population} at the level of the counties in the New York state. 

\medskip

\textbf{Travel rate ($\tau_{ij}$).} To construct matrix $A$ for the COVID-19 model, we need travel rate matrix $\tau$, where $\tau_{ij}$ represents the rate at which an individual travels from location $i$ to location $j$. We use the Social Distancing Metrics dataset \cite{travel} from SafeGraph to generate $\tau$. This dataset was collected using a panel of GPS pings from anonymous mobile devices, and it is based on Census Block Group levels. For each device/individual, the dataset identifies a ``home'' CBG, and the median daily home-dwell-time is provided for each CBG. Additionally, this dataset provides the daily number of trips that the people go from their home CBG to various destination CBGs. 

In our empirical simulations, we only consider the network of New York State (i.e.,  we do not consider the trips to places outside the New York State). For each node, we aggregate the number of trips to the county level and obtain the number of trips from one node to another. We can also obtain the home-dwell-time of each node as the median of the home-dwell-time among all the CBGs (daily median home-dwell-time) in this county. Then, we define $\tau_{ij} = (1 - \frac{W_i}{1440})\cdot \frac{k_{ij}}{\sum_a k_{ia}}$, where $W_i$ is the home-dwell-time of node $i$ (measured in minutes), $k_{ia}$ is the number of trips from node $i$ to node $a$. We divide $h_i$ by $1440$ because the latter is the total number of minutes in a day. 

\textbf{Symptom rate $\epsilon$, recovery rate $r^{\text a}$, $r^{\text s}$, and mortality rate $\kappa$.} To obtain the disease parameters $\epsilon$, $r^{\text a}$, $r^{\text s}$, and $\kappa$ in our COVID-19 model, we first introduce some other important parameters in an epidemic. The infection mortality rate (IFR) is the ratio of mortality to the total infections \cite{levin2020assessing}. 
Besides, we use $d_A$, $d_S$ to denote the asymptomatic infectious period and symptomatic infectious period, respectively.

We use the estimation in \cite{levin2020assessing} to obtain the infection mortality rate of the COVID-19. For individuals of age $j$, we have
\begin{equation}\label{eq: mortality rate}
	\log_{10}({\rm IFR}_j) = -3.27 +0.0524 \times j,
\end{equation}
where ${\rm IFR}_j$ represents the infection mortality rate of age $j$. For COVID-19 model without considering te demographic structure, we let the infection mortality rate be the average of different ages,  which is ${\rm IFR}' = \sum_{j =0}^{89}{{\rm IFR}_j}/ 90 = 0.0242.$

We found the values of $d_A$, $d_S$ are different in different references. The detailed values of $d_A$, $d_S$ in  \cite{birge2020controlling,giordano2020modelling,bertozzi2020challenges,bubar2021model,luciana2020,bertsimas2020optimizing} are presented in Table S\ref{tb: d_A d_S}. We let $d_A$, $d_S$ be the median value of the estimation in these references, i.e., $d_A = 5.0025$, $d_S = 6.2475$.

Moreover, we learn from CDC\cite{symptom_rate} that 81\% COVID-19 patients have mild to moderate symptoms and 19\% patients have severe to critical symptoms. Then we can derive 
\begin{align}\label{eq: disease parameters}
	\epsilon + r^{\rm a} &= \frac{1}{d_A}, \nonumber\\
	r^{\rm s} + \kappa & = \frac{1}{d_S}, \nonumber\\
	\frac{\epsilon}{\epsilon + r^{\rm a}} &= \frac{19}{81}, \nonumber\\
	\frac{19}{81} \times \frac{\kappa}{\kappa + r^{\rm s}} &= {\rm IFR}'.
\end{align}
 for our COVID-19 model \eqref{COVID_19}.
Thus we can compute the values of the parameter $\epsilon$, $r^{\rm a}$, $r^{\rm s}$, $\kappa$ from
\eqref{eq: disease parameters} as
\[\epsilon = 0.0469, ~r^{\rm a} = 0.153, ~\kappa = 0.0165, ~r^{\rm s} =  0.1436.\]

\textbf{Transmission rate ($\beta^{\text a}$, $\beta^{\text s}$).} Since we already have the parameter $r^{\text a}$, $r^{\text s}$, $\epsilon$, and $\kappa$, we choose the transmission rate $\beta^{\text a}$, $\beta^{\text s}$ to match the effective reproduction number $R_t$ of NY. Our first step is to let  $\beta^{\text a} = \hat{\alpha} \beta^{\text s}$ and assume we can reuse $\hat{\alpha}$ from the existing literature \cite{giordano2020modelling}, as this scalar measures the transmission rate difference of symptomatic individuals and asymptomatic individuals. Thus we only need to decide how to choose $\beta^{\text s}$. Our second step is to choose $\beta^{\text s}$ to match the the effective reproduction number $R_t$ of COVID-19 model to the $R_t$ of NY, which we obtained from website \cite{Rt_live}. This is  $\lambda_{\rm max}(LD^{-1})$, where
\[L = \left(\begin{matrix}\beta^{\text a} {\rm diag}(s(t_0))A & \beta^{\text s} {\rm diag}(s(t_0))A \\
	0 & 0 \end{matrix} \right), \]
and
\[ D = \left(\begin{matrix}\epsilon + r^{\text a} & 0 \\
	-\epsilon & r^{\text s} + \kappa \end{matrix} \right) .\]
We can write $C(t_0) = L - D$,  where $C(t_0)$ is defined in \eqref{COVID_19}.

\textbf{Efficacy ($\psi$).} Currently, the majority of the available vaccines in the United States come from Pfizer-BioNTech and Moderna. According to literature, the Pfizer vaccine was 95\% effective in preventing COVID-19, while the Moderna vaccine was 94.1\% effective in preventing COVID-19. The efficacy of the two types of vaccines are essentially equivalent. In our simulations on NY data, we simply let the efficacy $\psi = 0.95$.

\subsection{Data \& parameters for COVID-19 model with demographic structures}\label{sec: data covid_19_demo}
\quad

\textbf{Contact matrix ($C$).} When we consider the demographic structure for COVID-19 model, we need use the contact matrix to quantify the contact intensities between different age groups. However, the empirical contact matrix obtained by gathering social contacts is not available for regions in the United States. Therefore, we use the estimated contact matrix for NY in literature \cite{mistry2021inferring},  where the population is divided into 85 partitions. In this work, we consider six age groups, our first step is to combine the 85 age groups into 6 age groups we consider and rewrite the contact matrix in \cite{mistry2021inferring} to a $6\times 6$ contact matrix $C$. Then we use the method in \cite{arregui2018projecting} to get the intrinsic connectivity matrix $\Gamma$ as follows
\begin{equation}\label{eq: NY_Gamma}
\Gamma = M {\rm diag}\left(\frac{N}{N_i}\right) =  \left(\begin{matrix}
	22.9768  &  15.3439 &   9.1141 &  11.3077  &  4.5509  &  3.2704\\
	15.3439 &  54.2639  &  9.7226  &  11.5955  &  8.6947  &  3.9597\\
	9.1141  &  9.7226  & 28.8528  &  14.7380 &  13.7316  &  5.1510\\
	11.3077 &  11.5955 &  14.7380 &  18.0776  & 12.9846  &  5.4702\\
	4.5509  &  8.6947 &  13.7316 &  12.9846 &  15.6485  &  6.3227\\
	3.2704  &  3.9597  &  5.1510  &  5.4702  &  6.3227  &  15.2828
\end{matrix}\right),
\end{equation}
where $N_i$ is the number of population of age group $i$, $N$ is the total population.

\textbf{Symptom rate $\epsilon$, recovery rate $r^{\text a}$, $r^{\text s}$, and mortality rate $\kappa$.} If we consider the demographic structure of the COVID-19 model, we assume that the infection mortality rate of different age groups are different. We still used the estimation in \cite{levin2020assessing} to generate the infection mortality rate. We computed
\[{\rm IFR}_a' = \frac{\sum_j {\rm IFR}_j}{l},\]
where ${\rm IFR}_a'$ is the infection mortality rate of age group $a$, $j$ are the ages corresponding to group $a$, IFR$_j$ is defined in \eqref{eq: mortality rate}, $l$ is the length of the age range of group $l$. In this case, ${\rm IFR}'$ is a vector rather than scalar. Next, we also used the equations \eqref{eq: disease parameters} to compute the values of $\epsilon$, $r^{\text a}$, $r^{\text s}$, $\kappa$, i.e.,
\[\epsilon = 0.0469, ~r^{\rm a} = 0.153, ~\kappa = \mathbf{1}\otimes \left( \begin{matrix}
	0.0000047\\
	0.000018\\
	0.000075\\
	0.00036\\
	0.0033\\
	0.0565
\end{matrix}\right), ~r^{\rm s} =  \mathbf{1}\otimes \left( \begin{matrix}
0.1601\\
0.1600\\
0.1600\\
0.1597\\
0.1568\\
0.1035
\end{matrix}\right),\]
where $\mathbf{1} \in \mathbb{R}^{n\times 1}$ is the all-ones vector.

\textbf{Transmission risk ($\beta^{\text a'}$, $\beta^{\text s'}$).} The choice of the parameters $\beta^{\text a'}$, $\beta^{\text s'}$ is similar to the transmission rate $\beta^{\text a}$, $\beta^{\text s}$. First, we let
\begin{align*}
	\beta^{\rm s'} &= \beta\left( \mathbf{1}\otimes\beta_0\right),\\
	\beta^{\rm a'} &= \hat{\alpha}\beta\left( \mathbf{1}\otimes\beta_0\right),
\end{align*}
where $\mathbf{1}\in\mathbb{R}^{n\times 1}$ is the all-ones vector, $\beta_0 \in \mathbb{R}^{6\times 1}$ measures the difference of the transmission risk between different age groups, $\beta$ is a scalar.
We still assume that we can reuse $\hat{\alpha}$ from the existing literature \cite{giordano2020modelling}. We get the value of $\beta_0$ from Extended Data Fig. 4 in \cite{davies2020age}. The division of the age groups in \cite{davies2020age} is different from ours. To deal with this, we assume all the ages in each age group of \cite{davies2020age} have the same mean value of the transmission risk, then we compute $\beta_{0}(a)$ as the average of the transmission rate in group $a$. In this case, we can get
\[\beta_0 = \left(\begin{matrix}
	0.400\\
	0.387\\
	0.790\\
	0.840\\
	0.830\\
	0.768
\end{matrix}\right).\]
Next we choose the scalar $\beta$ to match the effective reproduction number $R_t$ of NY \cite{Rt_live}, where $R_t$ can be computed with a similar method.

\textbf{Initial rates ($s(t_0), x^{\text a}(t_0), x^{\text s}(t_0), e(t_0), h(t_0)$).} The number of confirmed cases and death cases of each age group in county level for NY is not available, therefore we still use the data from CDC \cite{cdc_data_tracker} to estimate these numbers. First, we get the confirmed cases and death cases in county level of NY on Dec. 1st, 2020 from \cite{susceptible,death_rate}. Suppose $f_c(a)$ ($f_d(a)$) is the empirical ratio of confirmed (death) cases of age group $a$ to the total confirmed (death) cases we calculated from the data in \cite{cdc_data_tracker}. Then we let 
\begin{equation*}
\begin{aligned}
	I_{i,a}'(t_0) & = I_i(t_0) f_c(t_0), \\
	E_{i,a}'(t_0) & = E_i(t_0) f_d(t_0), 
\end{aligned}
\end{equation*}
where $I_{i,a}'(t_0)$ ($E_{i,a}'(t_0)$) is the number of confirmed (death) cases of age group $a$ at location $i$,  $I_i(t_0)$ ($E_i(t_0)$) is the number of confirmed (death) cases of location $i$.  Next, we use the similar method as in Sec. \ref{sec: data covid_19} to compute the initial rates $s(t_0), x^{\text a}(t_0), x^{\text s}(t_0), e(t_0), h(t_0)$.\\

All the other data and parameters used in COVID-19 model with demographic structure are the same as in Sec. \ref{sec: data covid_19}.

\section{Two-nodes network model}\label{sec: two nodes}
We now revisit the phenomenon we have observed in our analysis of NY,  which is that the optimal stabilizing vaccine allocation tends to assign zero vaccines to the counties in NYC, Long Island, and Mid-Hudson. To isolate this phenomenon in the simplest possible setting, we implement a simple synthetic experiment of a network with two nodes. 

We assume the number of the available vaccines is 10\% of the total population,  $R_t = 1.0697$, which equals to the value of $R_t$ of NY on Dec. 1st, 2021. Then we choose $\beta^{\text a}$, $\beta^{\text s}$ to match $R_t$ of the COVID-19 model to this value. All the other disease parameters ($\epsilon$, $\kappa$, $r^{\text a}$, $r^{\text s}$, $\psi$) choose the same values as in Sec. \ref{sec: data covid_19}.  For the choice of the travel rate matrix $\tau$, we choose a matrix that is similar to the NY data, but with rounder numbers; specifically,  we define $\tau_{ij} = (1- \frac{h_i}{1440})\cdot \frac{k_{ij}}{\sum_a k_{ia}}$, where $h_i$ is the home-dwell-time of node $i$, $k_{ia}$ is the number of trips from node $i$ to node $a$, and we let 
\[k = \left[\begin{matrix} 8000 & 200\\
	200 & 8000 \end{matrix} \right ].  \]
We consider four different cases:
\begin{itemize}
	\item Case 1: population$=[200,000 \quad 2,000]$,  $h =  [800 \quad 800]$, $s(t_0) = [0.9 \quad 0.9]$.
	
	\item Case 2: population$=[2,000 \quad 2,000]$,  $h =  [800 \quad 800]$, $s(t_0) = [0.7 \quad 0.9]$.
	
	\item Case 3: population$=[2,000 \quad 2,000]$,  $h =  [1,000 \quad 800]$, $s(t_0) = [0.9 \quad 0.9]$.
	
	\item Case 4: population$=[200,000 \quad 2,000]$,  $h =  [1,000 \quad 800]$, $s(t_0) = [0.7 \quad 0.9]$.  
\end{itemize}
Case 1,2,3 are designed to observe the impact of the population, initial susceptible rate, and the home-dwell-time to the value of optimal stabilizing vaccine allocation rate $v_i^*$. Case 4 is designed to mimic the situation in NY, where node 1 is similar to the counties in NYC, Long Island, and MidHudson which has larger values of population, home-dwell-time, and smaller values of initial susceptible rate. Then we apply the proposed algorithm to design the optimal stabilizing vaccine allocation policy for this two-nodes network model. 

The simulation results are presented in \stref{tab: two-nodes}. We can see the value of $v_i^*$ is sensitive to the home-dwell-time and the initial susceptible rate, but not sensitive to the population. We also see the same phenomenon as in our NY simulations in Case 4: the optimal stabilizing vaccine allocation policy gives priority to node 2, even though epidemics mainly localized in the node 1.

\subsection*{Sensitivity analysis.} To further study the impact of home-dwell-time, initial susceptible rate, and the population to the value of $v_i^*$, we implement sensitivity experiments. In each experiment, we vary the value of one parameter of node 1 while fix the values of the others. The normal values of the data are chose as:
\[\text{ population}=[2,000 \quad 2,000], \quad h =  [800 \quad 800], \quad s(t_0) = [0.7 \quad 0.9]. \]
All the other data and parameters are set as before. The simulation results are presented in Fig. S\ref{fig: small_network_sensitivity}. 

As expected, it can be observed from Fig. S\ref{fig: small_network_sensitivity} that the optimal stabilizing vaccine allocation policy assigns zero vaccines to location with larger value of home-dwell-time, or smaller value of the initial susceptible rate. Besides, the value of $v_i^*$ is not sensitive to the population at all.

\section{optimal stabilizing vaccine allocation design on model in \cite{bubar2021model}}\label{sec: comparion with bubar}
In this section, we present the details about the optimal stabilizing vaccine allocation design for model in literature \cite{bubar2021model}. The epidemic model considered in \cite{bubar2021model} can be summarized as
\begin{equation}\label{eq: model_bubar}
	\begin{aligned}
		\dot{S}_i & = -\lambda_i S_i\\
		\dot{S}_{x,i} &= -\lambda_i S_{x,i} \\
		\dot{S}_{v,i} &= 0\\
		\dot{E}_i & = \lambda_i S_i - d_E^{-1} E_i\\
		\dot{E}_{x,i} & = \lambda_i S_{x,i} - d_E^{-1} E_{x,i}\\
		\dot{E}_{v,i} & = d_E^{-1} E_{v,i}\\
		\dot{I}_i & = d_E^{-1} E_i -d_I^{-1} I\\
		\dot{I}_{x,i} & = d_E^{-1} E_{x,i} -d_I^{-1} I_{x,i}\\
		\dot{I}_{v,i} & = d_E^{-1} E_{v,i} -d_I^{-1} I_{v,i}\\
		\dot{R}_i & = d_I^{-1}(1 - IFR) I_i\\
		\dot{R}_{x,i} & = d_I^{-1}(1 - IFR) I_{x,i}\\
		\dot{R}_{v,i} & = d_I^{-1}(1 - IFR) I_{v,i}\\
		\dot{D}_i & = d_I^{-1} IFR I_i + d_I^{-1} IFR I_{x,i} + d_I^{-1} IFR I_{v,i},
	\end{aligned}
\end{equation}
where $S,~E,~ I, ~R, ~D\in\mathbb{R}^{d\times 1}$ ($d$ is the number of age groups) represents the susceptible, exposed, infectious, recovered, and died compartments; subscripts of $v$ and $x$ denote those who have been vaccinated with protection, and those who will either not be vaccinated (vaccine refusal or positive serotest) or have been vaccinated but without protection, respectively; scalar $d_E,~d_I$ represents the length of the latent period and th infectious period; $IFR\in \mathbb{R}^{d\times 1}$ represents the fatality rate; $\lambda_i$ is the force of infection for a susceptible individual in age group $i$, which is defined as
\[\lambda_i = u_i \sum_j c_{ij} \frac{I_j + I_{v_j} + I_{x_j}}{N_j - \Omega_j }, \]
where $u_i$ is the transmission risk of a contact with an infectious individual for an individual in age group $i$, $c_{ij}$ is the number of the age-$j$ individuals that an age-$i$ individual contacts per day, $N_j$ is the total population in group $j$, and $\Omega_j$ is the number of individuals from group $j$ who have died.

We will consider the all-or-nothing vaccine model in \cite{bubar2021model}, which assumes a fraction $\psi$ of vaccinated individuals are perfectly protected while the remaining $1- \psi$ individuals gain no protection. Under this assumption, after supplying vaccines to the population, we can get
\begin{equation*}
	\begin{aligned}
		S_i(t_0) &\rightarrow S_i(t_0) - v_i S_i(t_0),\\
		V_i(t_0) & \rightarrow V_i(t_0) + v_i S_i(t_0),\\
		S_{x, i}(t_0) &\rightarrow S_{x, i}(t_0) + (1- \psi)v_iS_i(t_0),
	\end{aligned}
\end{equation*}
where $V$ represents the compartment of vaccinated with protection, 
\[v_i = \frac{\text{\# vaccine doses for group i}}{S_i(t_0) + I_i(t_0) + R_i(t_0)}.\]

Next, we will use a similar analysis method to solve the optimal stabilizing vaccine allocation problem for model \eqref{eq: model_bubar}. As $\lambda_i$ in \eqref{eq: model_bubar} can be written in matrix form as follows
\[\lambda = A(I +  I_x + I_v),\]
where $A = D_u C D_{N - \Omega}$, $D_u = {\rm diag}(u_i)$, $C$ is the contact matrix, $D_{N- \Omega} = {\rm diag}\left(\frac{1}{N_i -\Omega_i}\right).$ Then we can write \eqref{eq: model_bubar}  in matrix form as
\begin{equation} \label{eq: matrix form bubar}
	\renewcommand{\arraystretch}{1.5}
	\left( 
	\begin{array}{c} 
\dot{S}\\ \dot{S}_x \\
\dot{S}_v\\
 \dot{E} \\ \dot{E}_x \\ \dot{E}_v \\ \dot{I} \\ \dot{I}_x \\ \dot{I}_v
	\end{array} 
	\right)= 
	\begin{tikzpicture}[baseline=-\the\dimexpr\fontdimen22\textfont2\relax ]
		\matrix (m)[matrix of math nodes, left delimiter=(,right delimiter=)]
		{ 0 & 0 & 0 & 0 & 0 & 0 & -{\rm diag}(S)A & -{\rm diag}(S)A & -{\rm diag}(S)A\\
		0 & 0 & 0 & 0 & 0 & 0 & -{\rm diag}(S_x)A & -{\rm diag}(S_x)A & -{\rm diag}(S_x)A\\
		0 & 0 & 0 & 0 & 0 & 0 & 0 & 0 & 0\\
		0 & 0 & 0 & -d_E^{-1} & 0 & 0 & {\rm diag}(S)A & {\rm diag}(S)A & {\rm diag}(S)A \\
0 & 0 & 0 &  0 & -d_E^{-1} & 0 & {\rm diag}(S_x)A & {\rm diag}(S_x)A & {\rm diag}(S_x)A \\
0 & 0 & 0 &  0 & 0 & -d_E^{-1} & 0 & 0 & 0 \\
0 & 0 & 0 & d_E^{-1} & 0 & 0 & -d_I^{-1} & 0 & 0\\
0 & 0 & 0 & 0 & d_E^{-1} & 0 & 0 & - d_I^{-1} & 0\\
0 & 0 & 0 & 0 & 0 & d_E^{-1} & 0 & 0 & -d_I^{-1}\\
		};
		\begin{pgfonlayer}{myback}
			\highlight[gray]{m-4-4}{m-9-9}
		\end{pgfonlayer}
	\end{tikzpicture}
	\left( 
	\begin{array}{c} 
		S\\S_x\\ S_v\\
E \\ E_x \\ E_v \\ I \\ I_x \\ {I}_v
	\end{array} 
	\right).
\end{equation}

We can see that the model \eqref{eq: matrix form bubar} has a similar form as the model \eqref{COVID_19}. By replacing the asymptomatic compartment in COVID-19 model as the exposed compartment in model \eqref{eq: matrix form bubar}, we can derive a similar conclusion as the Proposition 2 in \cite{ma2020optimal}, i.e., there exists a positive linear combination of $\{E_{i}, E_{x,i}, E_{v,i}, I_i, I_{x,i}, I_{v,i}\}$ which decays at rate $\alpha$ starting at time $t_0$ if $\lambda_{\rm max}(N(t_0))<\alpha$, where $N(t)$ is the submatrix outlined in a box in \eqref{eq: matrix form bubar}. Thus, we can formulate the optimal stabilizing vaccine allocation problem for model \eqref{eq: matrix form bubar} as minimizing the number of the vaccine doses while keeping the decay rate $\alpha$ fixed, that is
\begin{equation}\label{eq: bubar problem 1}
	\begin{aligned}
		\min_{v_i}\quad & \sum_i (S_i(t_0)+ I_i(t_0) +R_i(t_0))v_i\\
		s.t.\quad &\lambda_{\rm max}(N(t_0)) \leq -\alpha\\
		& 0\leq v_i \leq 1,~ i = 1,\ldots, n.
	\end{aligned}
\end{equation}
Next, we provide two versions of the reduction of problem \eqref{eq: bubar problem 1} depending the positive definiteness of the contact matrix $C$.

\begin{proposition}\quad \label{pro: reduction problem of bubar}
	\begin{enumerate}
		\item For general contact matrix $C$, problem \eqref{eq: bubar problem 1} can be written as the following bilinear optimization problem
		\begin{equation}\label{eq: bubar problem 2}
			\begin{aligned}
				\min_{v_i}\quad & \sum_i (S_i(t_0)+ I_i(t_0) +R_i(t_0))v_i\\
				s.t.\quad &b_1 {\rm diag}(S_i(t_0) + S_{x, i}(t_0) - \psi  v_i S_i(t_0))A d \leq d,\\
				& \sum_i d_i = 1,\\
				& 0\leq v_i \leq 1,~ i = 1,\ldots, n.\\
				& d_i > 0, ~ i = 1, \ldots, n.
			\end{aligned}
		\end{equation}
	\item If the matrix $\bar{A} = C D_{N- \Omega}$ is positive definite
	then problem \eqref{eq: bubar problem 1} can be written as the following SDP problem 
\begin{equation}   \label{eq: bubar problem 3}
	 \begin{aligned}
		\min_{v_i}\quad & \sum_i (S_i(t_0)+ I_i(t_0) +R_i(t_0))v_i\\
		s.t.\quad & {\rm diag}(b_1 u_i (S_x(t_0) + (1 - \psi v) S(t_0))) \preccurlyeq \bar{A}^{-1} 
		\\
		& 0\leq v_i \leq 1,~ i = 1,\ldots, n.
	\end{aligned}
\end{equation}

	\end{enumerate}
\end{proposition}

\begin{proof}[Proof of proposition \ref{pro: reduction problem of bubar}]
First, we will prove part (1) of Proposition \ref{pro: reduction problem of bubar}.  To make the first constraint in \eqref{eq: bubar problem 1} hold, we need that 
\begin{equation}\label{eq: bubar consraint 1}
N(t_0) - \alpha I \text{ is continuous time stable.} 
\end{equation}
Let us write \[N(t_0) - \alpha I = L - D,\] where
\begin{equation*}
	\scalemath{0.9}{
	L = \left(\begin{matrix} 
	0 & 0 & 0 & {\rm diag}(S_i(t_0) - v_i S_i(t_0))A & {\rm diag}(S_i(t_0) - v_i S_i(t_0))A & {\rm diag}(S_i(t_0) - v_i S_i(t_0))A \\
		0 & 0 & 0 & {\rm diag}(S_{x,i}(t_0) + (1-\psi)v_iS_i(t_0))A & {\rm diag}(S_{x,i}(t_0) + (1-\psi)v_iS_i(t_0))A & {\rm diag}(S_{x,i}(t_0) + (1-\psi)v_iS_i(t_0))A \\
		0 & 0 & 0 & 0 & 0 & 0 \\
		0 & 0 & 0 & 0 & 0 & 0\\
		0 & 0 & 0 & 0 & 0 & 0\\
		0 & 0 & 0 & 0 & 0 & 0
	\end{matrix} \right), }
\end{equation*}
and
\begin{equation*}
	\scalemath{0.9}{
	D = \left(\begin{matrix} 
		d_E^{-1} -\alpha & 0 & 0 & 0 & 0 & 0 \\
		0 & d_E^{-1}-\alpha & 0 & 0 & 0 & 0 \\
		0 & 0 & d_E^{-1}-\alpha & 0 & 0 & 0 \\
		-d_E^{-1} & 0 & 0 & d_I^{-1}-\alpha & 0 & 0\\
		0 & - d_E^{-1} & 0 & 0 & d_I^{-1}-\alpha &  0\\
		0 & 0 & - d_E^{-1} & 0 & 0 &  d_I^{-1}-\alpha 
	\end{matrix} \right). }
\end{equation*}
Since $L$ is nonnegative, the off-diagonal elements of $D$ is non-positive and its inverse is elementwise nonnegative, 
then according to Lemma 7, part (3) in \cite{ma2020optimal}, we can derive that \eqref{eq: bubar consraint 1} is equivalent to
\begin{equation}\label{eq: bubar constraint 2}
D^{-1}L \text{ is discrete time stable.} 
\end{equation}
As the nonzero eigenvalues of the product of two matrices keeps the same when the order of the product changes, we can also write \eqref{eq: bubar constraint 2} as
\[LD^{-1} \text{ is discrete time stable.} \]
Observing that
\[LD^{-1} = \left(
\begin{matrix}
	\frac{1}{(d_E^{-1} - \alpha)(d_I^{-1} - \alpha)}A_1 & \frac{1}{d_I^{-1} - \alpha}A_1 \\
	0 & 0
\end{matrix}
\right),\]
where
\[
	\scalemath{0.95}{A_1 = \left( \begin{matrix} {\rm diag}(S_i(t_0) - v_i S_i(t_0))A & {\rm diag}(S_i(t_0) - v_i S_i(t_0))A & {\rm diag}(S_i(t_0) - v_i S_i(t_0))A \\
	{\rm diag}(S_{x,i}(t_0) + (1-\psi)v_iS_i(t_0))A & {\rm diag}(S_{x,i}(t_0) + (1-\psi)v_iS_i(t_0))A & {\rm diag}(S_{x,i}(t_0) + (1-\psi)v_iS_i(t_0))A \\
	0 & 0 & 0
\end{matrix}\right).}\]
We can further reduce the first constraint in \eqref{eq: bubar problem 1} as
\begin{equation}\label{eq: bubar constraint 3}
	b_1A_1 \text{ is discrete time stable,}
\end{equation}
where $b_1 = 	\frac{1}{(d_E^{-1} - \alpha)(d_I^{-1} - \alpha)}.$ Obviously, matrix $A_1$ has the same nonzero eigenvalues as the submatrix
\[A_2 = \left( \begin{matrix} {\rm diag}(S_i(t_0) - v_i S_i(t_0))A & {\rm diag}(S_i(t_0) - v_i S_i(t_0))A \\
	{\rm diag}(S_{x,i}(t_0) + (1-\psi)v_iS_i(t_0))A & {\rm diag}(S_{x,i}(t_0) + (1-\psi)v_iS_i(t_0))A 
\end{matrix}\right). \]
Moreover, we can write 
\[A_2 = \left( \begin{matrix} 1 & 0 \\
	{\rm diag}\left(\frac{S_{x,i}(t_0) + (1-\psi)v_iS_i(t_0)}{S_i(t_0) - v_i S_i(t_0)}\right) & 1
\end{matrix}\right)
 \left( \begin{matrix} {\rm diag}(S_i(t_0) - v_i S_i(t_0))A & {\rm diag}(S_i(t_0) - v_i S_i(t_0))A \\
 	0 & 0
 \end{matrix}\right),\]
which has the same nonzero eigenvalues with matrix
\begin{align*}
&\left( \begin{matrix} {\rm diag}(S_i(t_0) - v_i S_i(t_0))A & {\rm diag}(S_i(t_0) - v_i S_i(t_0))A \\
	0 & 0
\end{matrix}\right)\left( \begin{matrix} 1 & 0 \\
	{\rm diag}\left(\frac{S_{x,i}(t_0) + (1-\psi)v_iS_i(t_0)}{S_i(t_0) - v_i S_i(t_0)}\right) & 1
\end{matrix}\right)\\
&= \left( \begin{matrix}{\rm diag}(S_i(t_0) + S_{x, i}(t_0) - \psi  v_i S_i(t_0))A & {\rm diag}(S_i(t_0) - v_i S_i(t_0))A \\
	0 & 0
\end{matrix}\right).	
\end{align*}
Thus, we can write the constraint \eqref{eq: bubar constraint 3} as
\begin{equation}\label{eq: bubar constrain 4}
b_1 {\rm diag}(S_i(t_0) + S_{x, i}(t_0) - \psi  v_i S_i(t_0))A \text{ is discrete time stable}.
\end{equation}
By applying Lemma 7, part (2) of \cite{ma2020optimal}, we can write the constraint above as there exists $d>0$, such that
\[ b_1 {\rm diag}(S_i(t_0) + S_{x, i}(t_0) - \psi  v_i S_i(t_0))A d \leq d.\]
To eliminate the scalar multiple of $d$, we add a normalization constraint $\sum_i d_i = 1$, and then we can get the reformulation \eqref{eq: bubar problem 2}. Thus, we complete the proof of part (1).

Next, we will prove part (2) of Proposition \ref{pro: reduction problem of bubar}. We  start our reduction from constraint \eqref{eq: bubar constrain 4}. As matrix 
\[A = D_uCD_{N - \Omega} = D_u \bar{A},\]
by applying Lemma 7, part (3) in \cite{ma2020optimal}, we can write the constraint \eqref{eq: bubar constrain 4} as
\begin{equation}\label{eq: bubar constraint 5}
\bar{A} - b_1 {\rm diag}\left(\frac{1}{b_1u_i(S_i(t_0) + S_{x, i}(t_0) - \psi  v_i S_i(t_0)))\bar{A}} \right) \text{ is continuous time stable}.
\end{equation}

When the contact matrix $C$ satisfies that the number of contact measured from $i$ to $j$ is equal to the number measured from $j$ to $i$, the matrix $\bar{A} = CD_{N- \Omega}$ is symmetric. Besides, the contact matrix obtained by gathering empirical social contacts usually shows a pattern \cite{mossong2008social, mistry2021inferring, Brittoneabc6810, prem2017projecting} that the diagonal elements are greater than off-diagonal elements, which ensures the positivity of the eigenvalues of matrix $\bar{A}$.
In other words, the positive definiteness of matrix $\bar{A}$ can be easily obtained.

If the matrix $\bar{A}$ is positive-definite, we can write the constraint \eqref{eq: bubar constraint 5} as
\[{\rm diag}(b_1u_i(S_i(t_0) + S_{x, i}(t_0) - \psi  v_i S_i(t_0))) - \bar{A}^{-1} \text{ is continuous time stable},\]
since $A \succcurlyeq B$ is equivalent to $B^{-1} \succcurlyeq A^{-1}$ if both matrices $A$ and $B$ are positive definite. In this case, we complete the proof of part (2).

\end{proof}

Since the contact matrices adopted in \cite{bubar2021model} can not satisfy the condition that the number of contacts measured from $i$ to $j$ is equal to the number measured from $j$ to $i$,  the matrix $\bar{A}$ is not positive definite. Therefore, we generate the optimal stabilizing vaccine allocation policy by solving the  problem \eqref{eq: bubar problem 2}. 

\clearpage







\clearpage

\section{Supplementary Figures}

\begin{figure}[!htb]
	\includegraphics[width=1.0\linewidth]{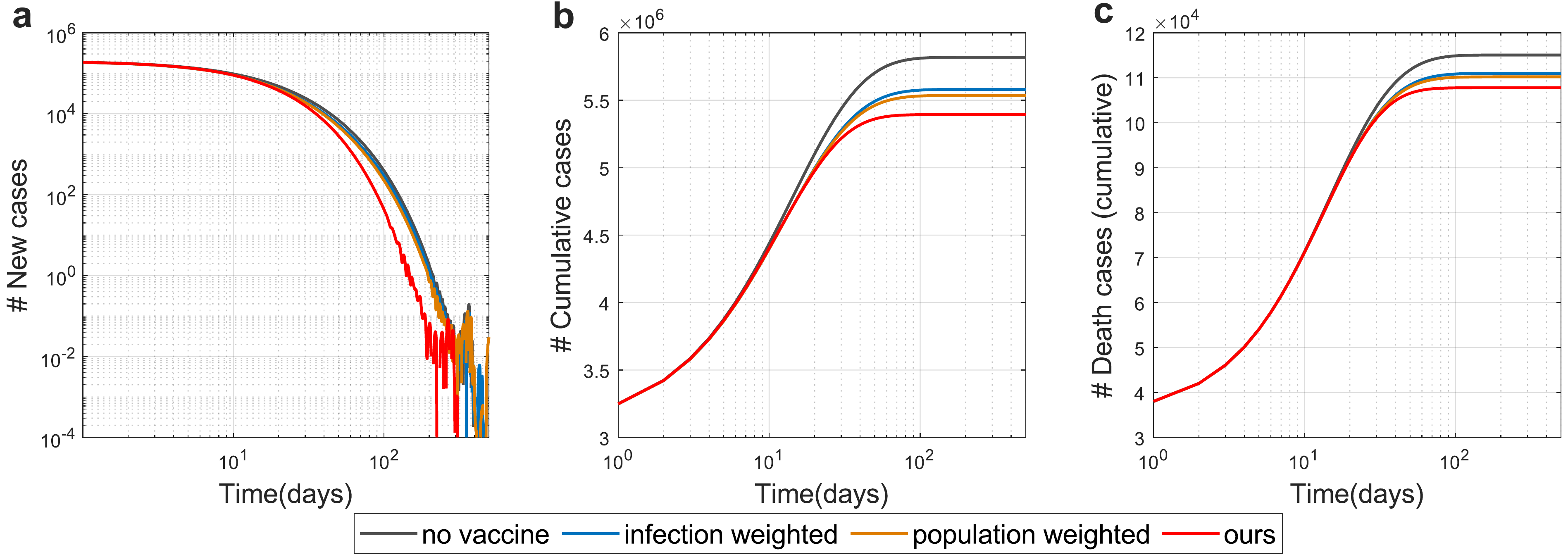}
	\centering
	\caption{
		{\bf When the vaccine supply is limited, and if we ignore the demographic structure of the population, our method outperforms all the other comparison methods.}  {\bf a}, the estimated number of daily new cases.  {\bf b}, the estimated number of cumulative cases. {\bf c}, the estimated number of cumulative death cases. The vaccines are supplied daily with
		a speed of $0.33\%$ of total population per day. The total number of the available vaccine doses for these policies are the same, that is $5\%$ of the total population. ``Population weighted", ``infection weighted", and ``no vaccine" polices are defined as in Fig. \ref{fig: demo_allocation_small}. The data applied is about COVID-19 outbreak in NY on December 1st, 2020.
	}\label{fig: location_curves_small}
\end{figure}
\clearpage
\begin{figure}[!htb]
	\includegraphics[width=1.0\linewidth]{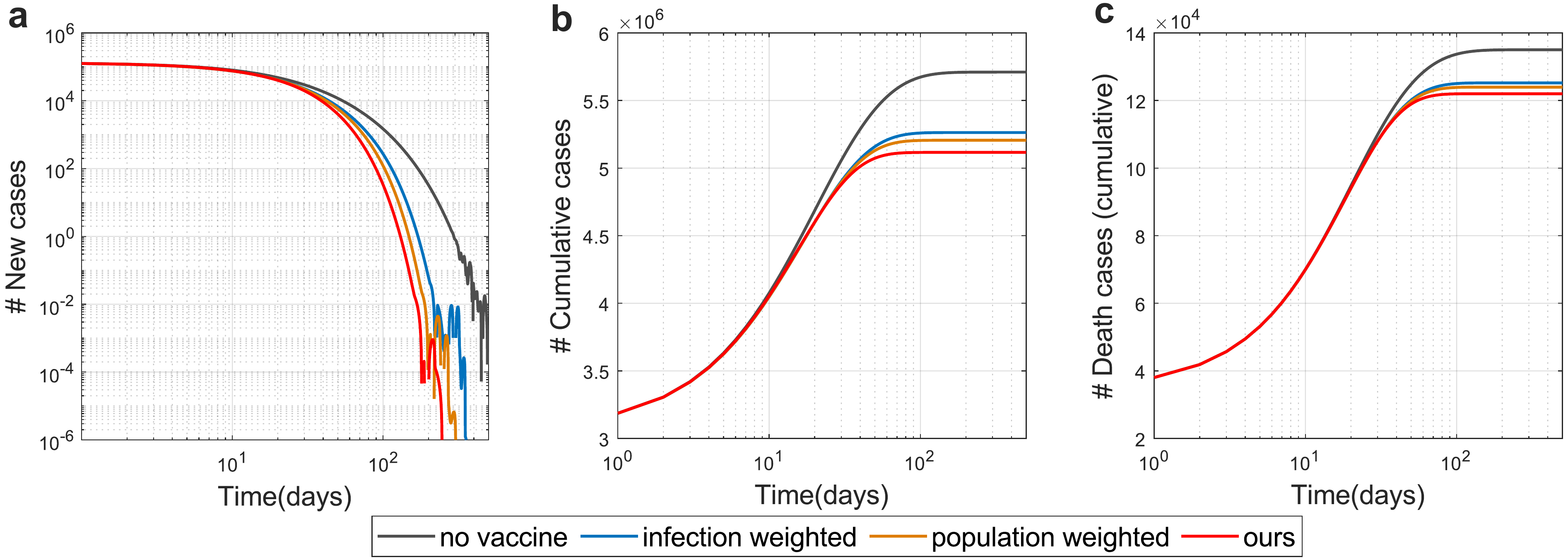}
	\centering
	\caption{
{\bf When the vaccine supply is unlimited, and if we ignore the demographic structure of the population, our method outperforms all the other comparison methods.}  {\bf a}, the estimated number of daily new cases.  {\bf b}, the estimated number of cumulative cases. {\bf c}, the estimated number of cumulative death cases. The vaccines are supplied daily with
a speed of $0.33\%$ of total population per day. The total number of the available vaccine doses for these policies are the same, that is $100\%$ of the total population. ``Population weighted", ``infection weighted", and ``no vaccine" polices are defined as in Fig. \ref{fig: demo_allocation_small}. The data applied is about COVID-19 outbreak in NY on December 1st, 2020.
	}\label{fig: location_curves_large}
\end{figure}
\begin{figure}[!htb]
	\centering
	\begin{minipage}[b]{0.95\linewidth}
		\centering
		\includegraphics[width=1.0\linewidth]{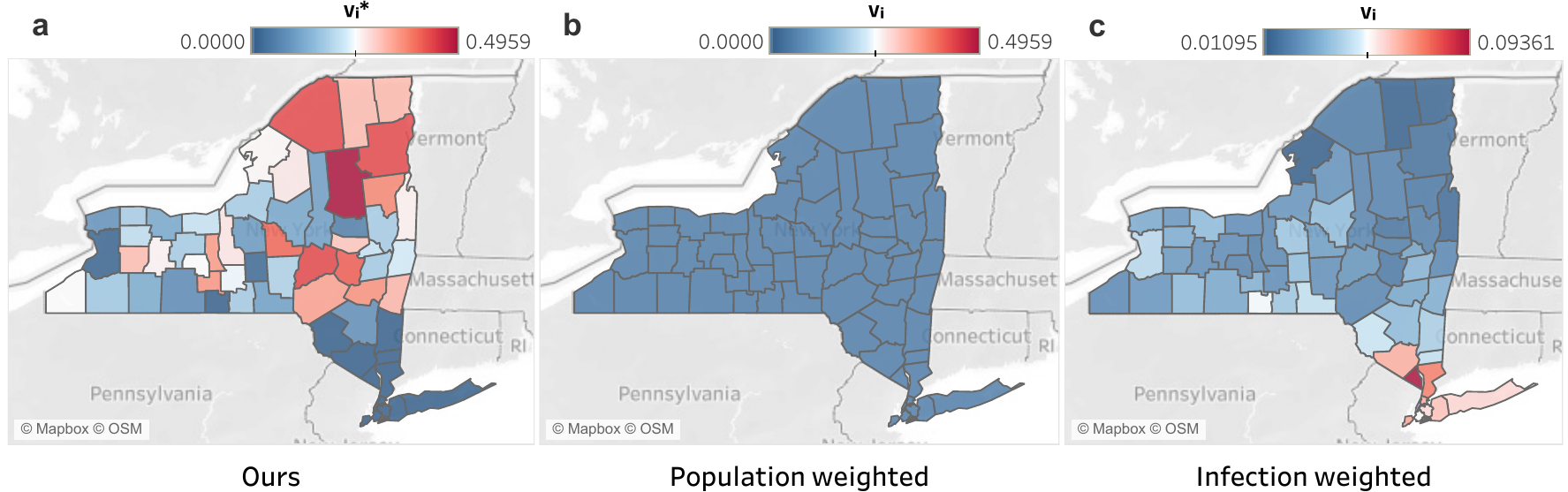}		\end{minipage}\hfill
	
	\vspace{3mm}
	
	\begin{minipage}[b]{0.95\linewidth}
		\centering
		\includegraphics[width=1.0\linewidth]{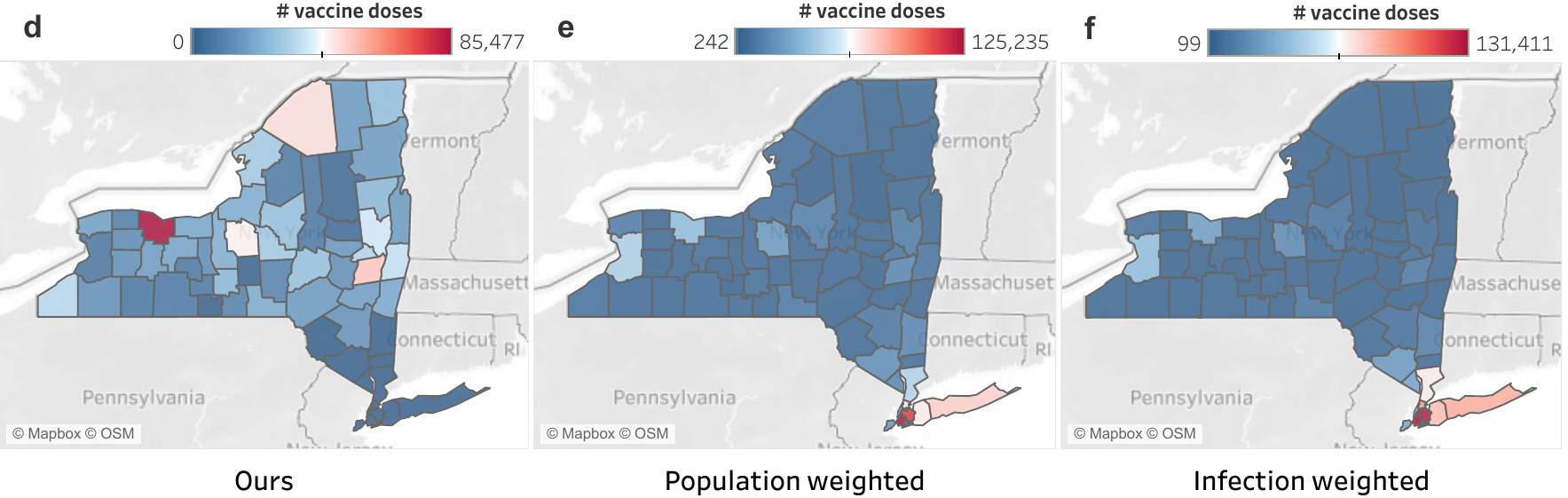}
	\end{minipage}%
	\centering
	\caption{{\bf When vaccine supply is limited, and if we ignore the demographic structure of the population, our method suggests to allocate zero vaccines to counties in NYC and Long Island.} {\bf a-c}, vaccine allocation rate $v_i$ of each county given by different policies. {\bf d-f}, the number of vaccine doses allocated to each county by different policies. The vaccines are supplied daily with
	a speed of $0.33\%$ of total population per day. The total number of the available vaccine doses for these policies are the same, that is $5\%$ of the total population. ``Population weighted", ``infection weighted", and ``no vaccine" polices are defined as in Fig. \ref{fig: demo_allocation_small}. The data applied is about COVID-19 outbreak in NY on December 1st, 2020. The values of $v_i^*$, $v_i$ in this figure corresponds to the results in Fig. S\ref{fig: location_curves_small}.
	}\label{fig: location_daily_small}
\end{figure}
\begin{figure}[!htb]
	\centering
	\begin{minipage}[b]{0.95\linewidth}
		\centering
		\includegraphics[width=1.0\linewidth]{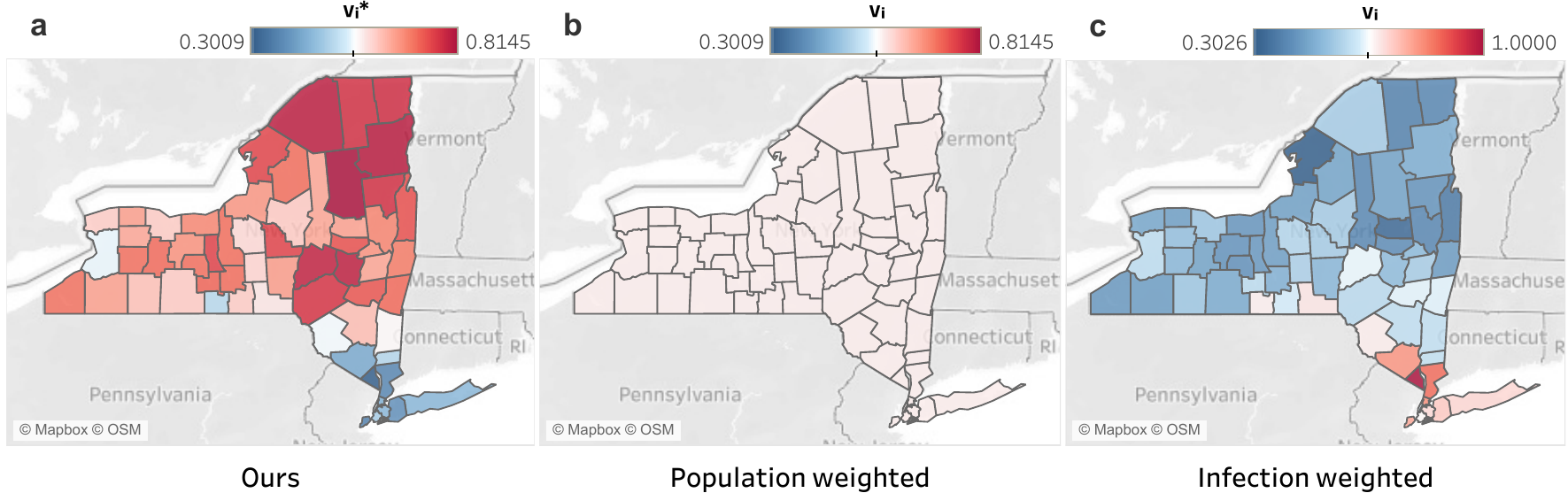}			
	\end{minipage}\hfill
	
	\vspace{3mm}
	
	\begin{minipage}[b]{0.95\linewidth}
		\centering
		\includegraphics[width=1.0\linewidth]{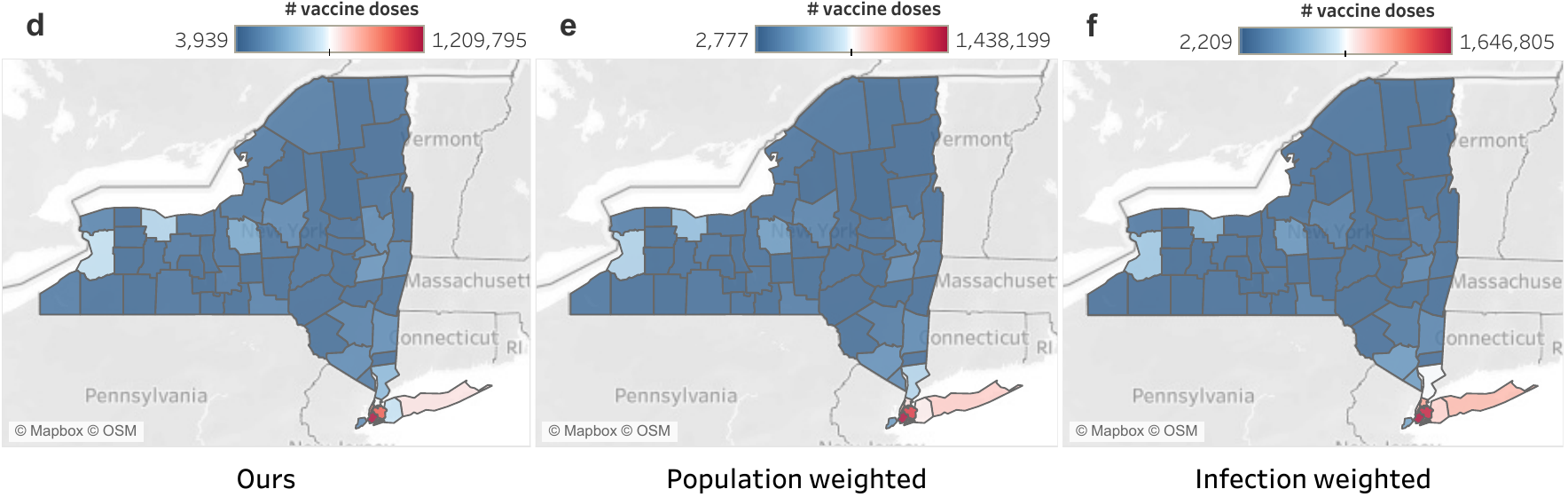}
	\end{minipage}%
	\centering
	\caption{{\bf When the vaccine supply is unlimited, and if we ignore the demographic structure of the population, our method suggests to allocate vaccines firstly to counties outside of the NYC and Long Island.} {\bf a-c}, vaccine allocation rate $v_i$ of each county given by different policies. {\bf d-f}, the number of vaccine doses allocated to each county by different policies. The vaccines are supplied daily with
	a speed of $0.33\%$ of total population per day. The total number of the available vaccine doses for these policies are the same, that is $100\%$ of the total population. In this figure, we only show the results before the disappearing of the epidemics (\# active cases is less than 1). After that, we will not distinguish between these policies, the leftover vaccines will be evenly distributed to the counties in each scenario. ``Population weighted", ``infection weighted", and ``no vaccine" polices are defined as in Fig. \ref{fig: demo_allocation_small}. The data applied is about COVID-19 outbreak in NY on December 1st, 2020. The values of $v_i^*$, $v_i$ in this figure corresponds to the results in Fig. S\ref{fig: location_curves_large}.
	}\label{fig: location_daily_large}
\end{figure}
\clearpage
\begin{figure}[!htb]
	\centering
	\begin{minipage}[b]{1.0\linewidth}
		\centering
		\includegraphics[width=0.8\linewidth]{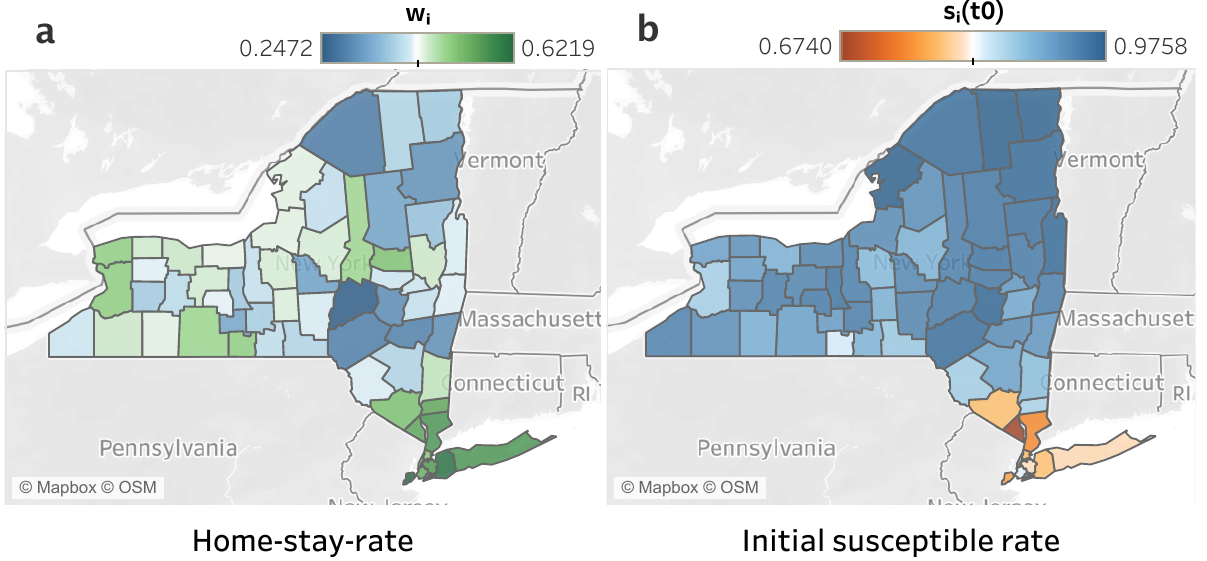}			
	\end{minipage}\hfill
	\caption{{\bf Counties in NYC and Long Island have relatively larger values of home-stay-rate and smaller values of initial susceptible rate.} {\bf a}, visualization of home-stay-rate $w_i$ of each county in NY, where $w_i = \frac{W_i}{1440}$, $W_i$ is the median of the home-dwell-time (in minutes) in a day for county $i$. {\bf b}, visualization of the initial susceptible rate $s_i(t_0)$ for each county in NY. The home-stay-rate is used to generate the travel rate matrix $\tau$, the details can be found in SI Sec. \ref{sec: data covid_19}. Both the home-stay-rate and initial susceptible rate are playing an import role to decide how to allocate the vaccines to each county in our method. From Fig. \ref{fig: demo_daily_small}, Fig. S\ref{fig: location_daily_small}, Fig. S\ref{fig: location_daily_large}, and Fig. S\ref{fig: demo_daily_large}, it can be observed that our method tends to allocate zero vaccines to county with large value of home-stay-rate and small value of initial susceptible rate.}\label{fig: home_rate}
\end{figure}
\clearpage
\begin{figure}[!htb]
	\includegraphics[width=0.70\linewidth]{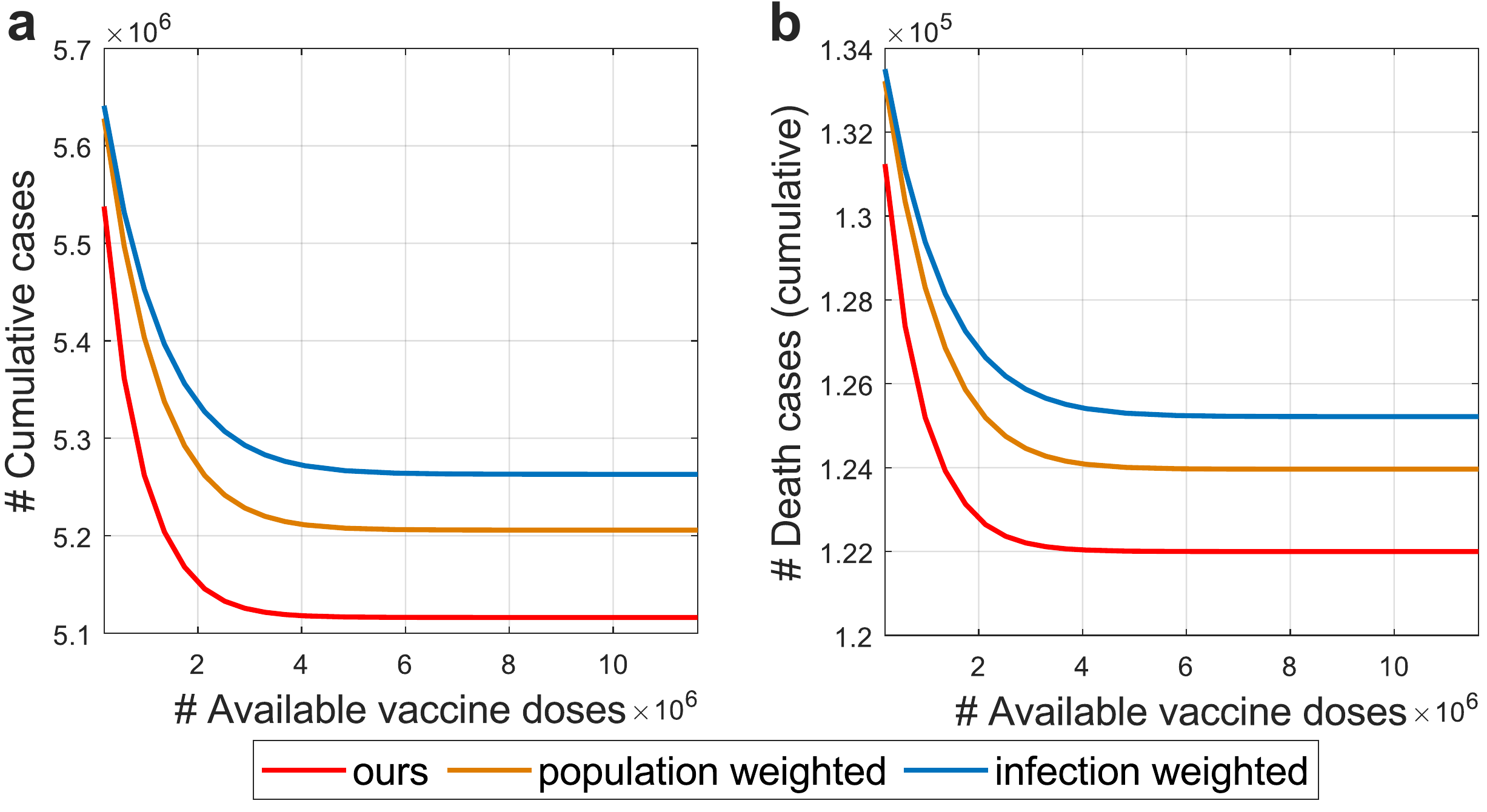}
	\centering
	\caption{ {\bf If we do not consider the demographic structure of the population, our method outperforms the other comparison methods regardless of the number of the available vaccine doses}. {\bf a}, the estimated number of final cumulative cases.  {\bf b}, the estimated number of final cumulative death cases. The vaccines are supplied daily with a speed of $0.33\%$ of total population per day. The total number of the available vaccine doses ranges from 1\% to 50\% of the population. We only observe to $50\%$ as the epidemic disappears after supplying around $50\%$ people with vaccines. ``Population weighted" and ``infection weighted" polices are defined as in Fig. \ref{fig: demo_allocation_small}. The data applied is about COVID-19 outbreak in NY on December 1st, 2020. It can be observed that the superiority of our method over the other two methods is more obvious when there are enough number of the available vaccine doses.
	}\label{fig: location_V_num_varying}
\end{figure}
%
\clearpage
\begin{figure}[!htb]
	\includegraphics[width=0.7\linewidth]{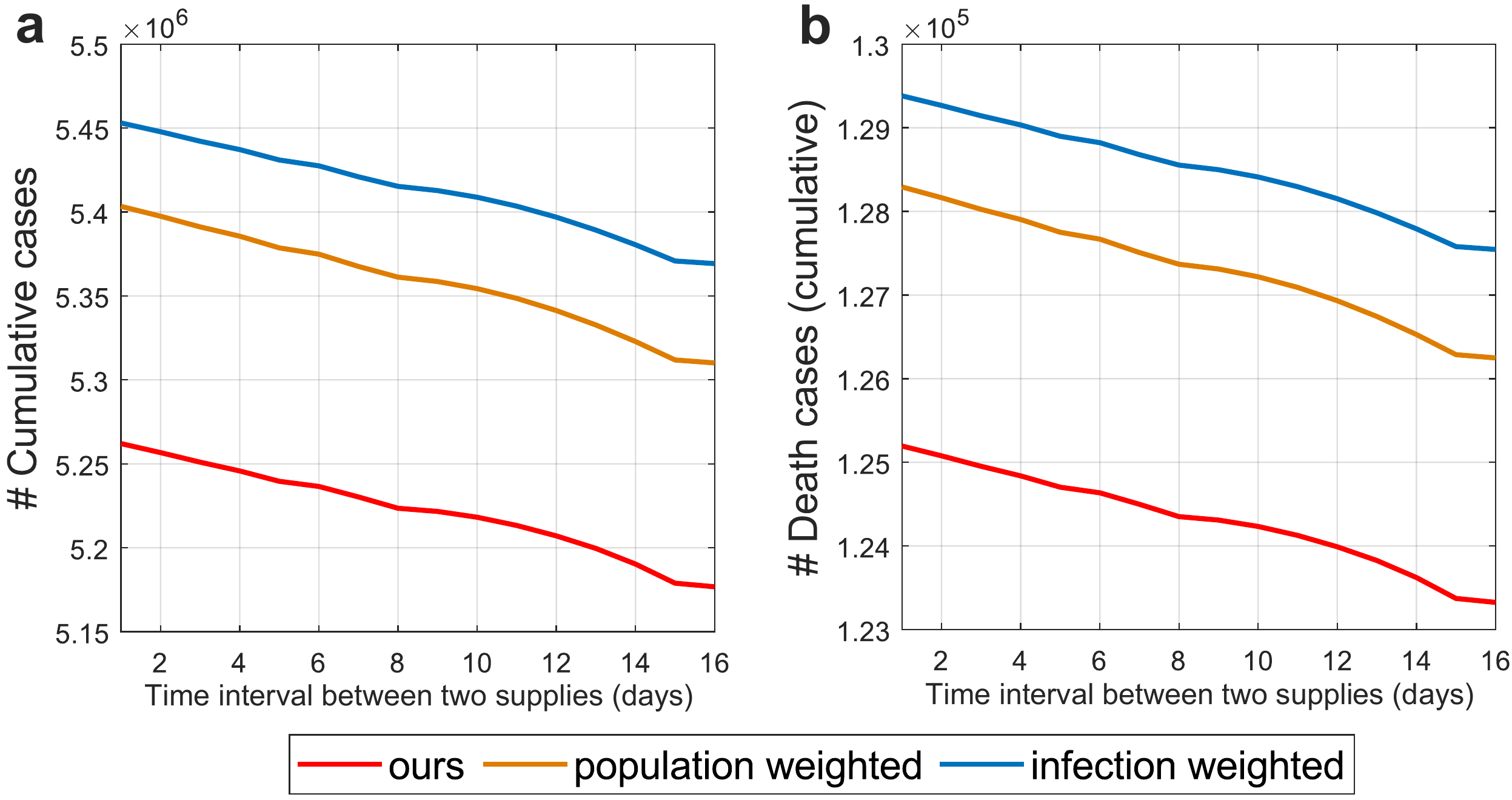}
	\centering
	\caption{ {\bf If we do not consider the demographic structure of the population, our method outperforms the other comparison methods regardless of the time interval between two vaccine supplies}. {\bf a}, the estimated number of final cumulative cases.  {\bf b}, the estimated number of final cumulative death cases. The vaccines are supplied in the beginning of each period, the time interval between two vaccine supplies ranges from 1 day to 16 days. The average number of the vaccine supply per day for these scenarios are the same, that is 0.33\% of the population. The total number of the available vaccine doses for these policies is $5\%$ of the total population.  ``Population weighted" and ``infection weighted" polices are defined as in Fig. \ref{fig: demo_allocation_small}. The data applied is about COVID-19 outbreak in NY on December 1st, 2020. The number of the cumulative cases and the death cases decreases when the time interval increases as more people will be vaccinated in the beginning.
	}\label{fig: location_period_varying}
\end{figure}
\clearpage

\begin{figure}[!htb]
	\includegraphics[width=0.7\linewidth]{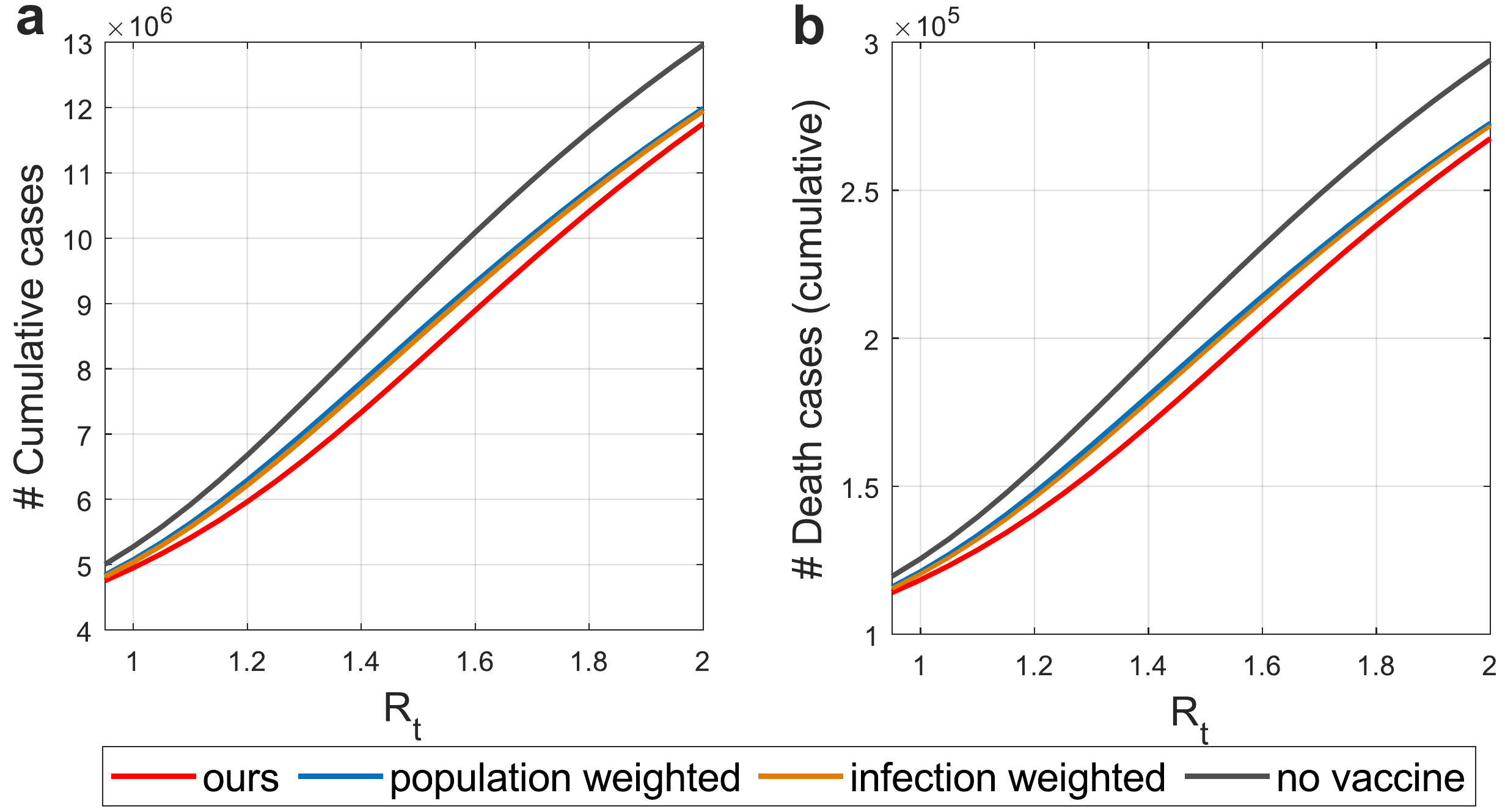}
	\centering
	\caption{{\bf If we do not consider the demographic structure of the population, our method outperforms the other comparison methods regardless of the the value of $R_t$}. {\bf a}, the estimated number of final cumulative cases.  {\bf b}, the estimated number of final cumulative death cases. In {\bf a} and {\bf b}, the value of $R_t$ ranges from 0.95 to 2.0. The vaccines are supplied daily with a speed of $0.33\%$ of total population per day. The total number of the available vaccine doses for these policies is $5\%$ of the total population.  ``Population weighted" and ``infection weighted" polices are defined as in Fig. \ref{fig: demo_allocation_small}. The data applied is about COVID-19 outbreak in NY on December 1st, 2020.
	}\label{fig: location_Rt_varying}
\end{figure}
\clearpage
\begin{figure}[!htb]
	\centering
	\begin{minipage}[b]{0.95\linewidth}
		\centering
		\includegraphics[width=1.0\linewidth]{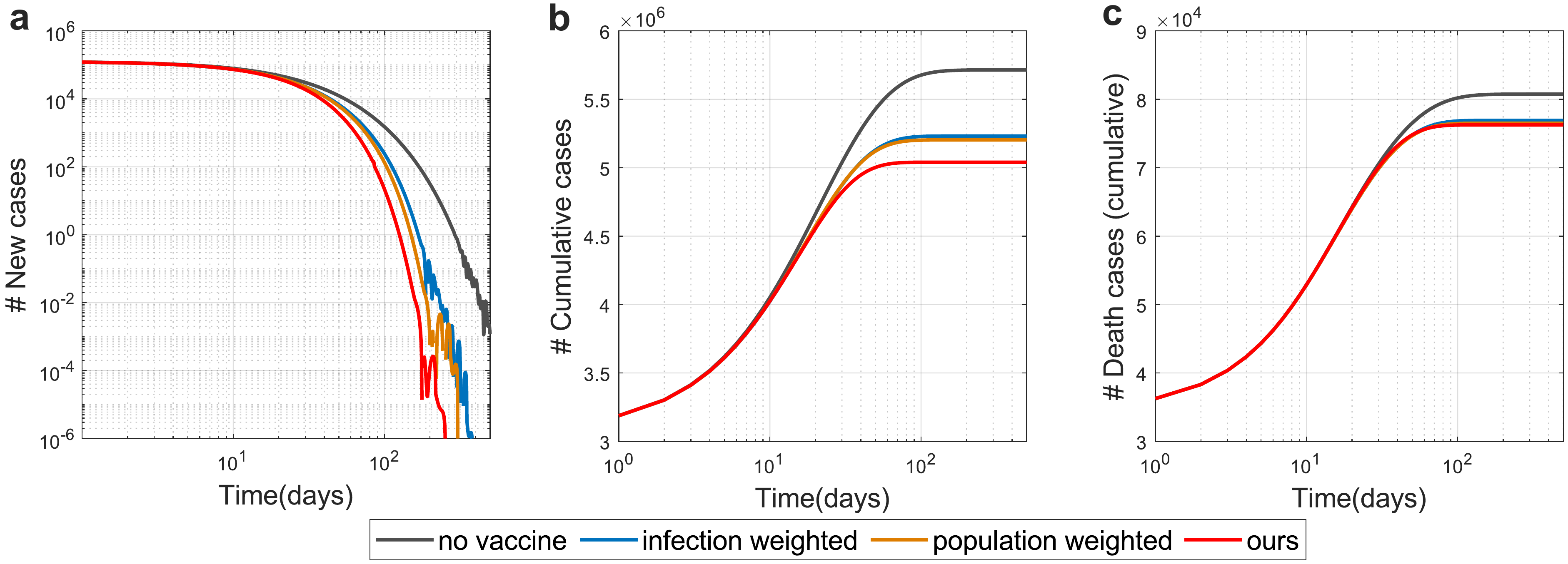}			
	\end{minipage}\hfill
	
	\vspace{3mm}
	
	\begin{minipage}[b]{0.6\linewidth}
		\centering
		\includegraphics[width=1.0\linewidth]{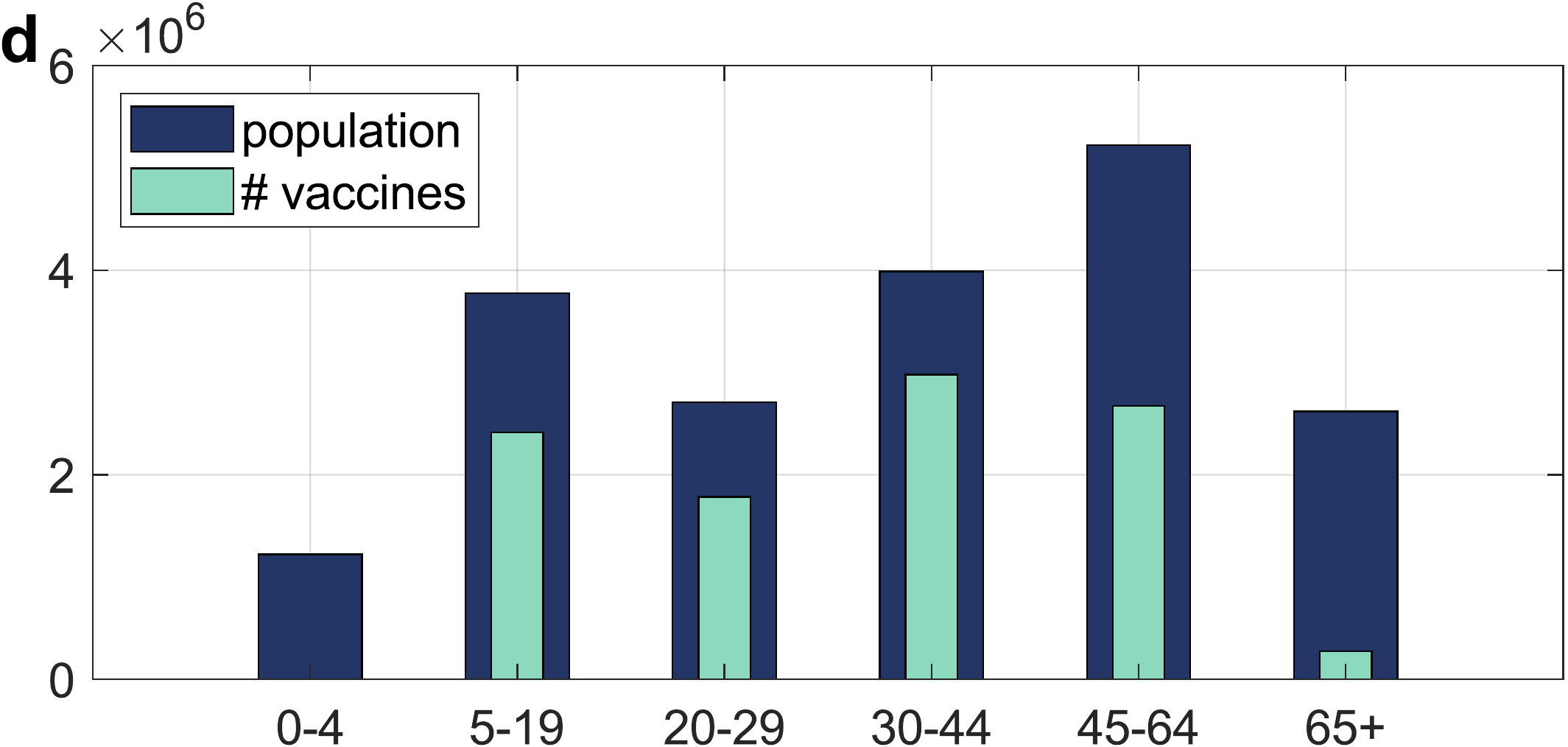}
	\end{minipage}%
	\centering
	\caption{{\bf When the vaccine supply is unlimited, our method suggests to allocate vaccines to adults in 20-64 and the school-age children (5-19) firstly, such strategy outperforms all the other comparison methods.}  {\bf a}, the estimated number of daily new cases.  {\bf b}, the estimated number of cumulative cases. {\bf c}, the estimated number of cumulative death cases. {\bf d}, the optimal stabilizing vaccine allocation number for each age group calculated by our method. The vaccines are supplied daily with
		a speed of $0.33\%$ of total population per day. The total number of the available vaccine doses for these policies are the same, that is $100\%$ of the total population. ``Population weighted", ``infection weighted", and ``no vaccine" polices are defined as in Fig. \ref{fig: demo_allocation_small}. The data applied is about COVID-19 outbreak in NY on December 1st, 2020.
	}\label{fig: demo_curves_large}
\end{figure}
%
%
%
%
%
%
%
\clearpage

\begin{figure}[!htb]
	\centering
	\begin{minipage}[b]{0.5\linewidth}
		\centering
		\includegraphics[width=0.8\linewidth]{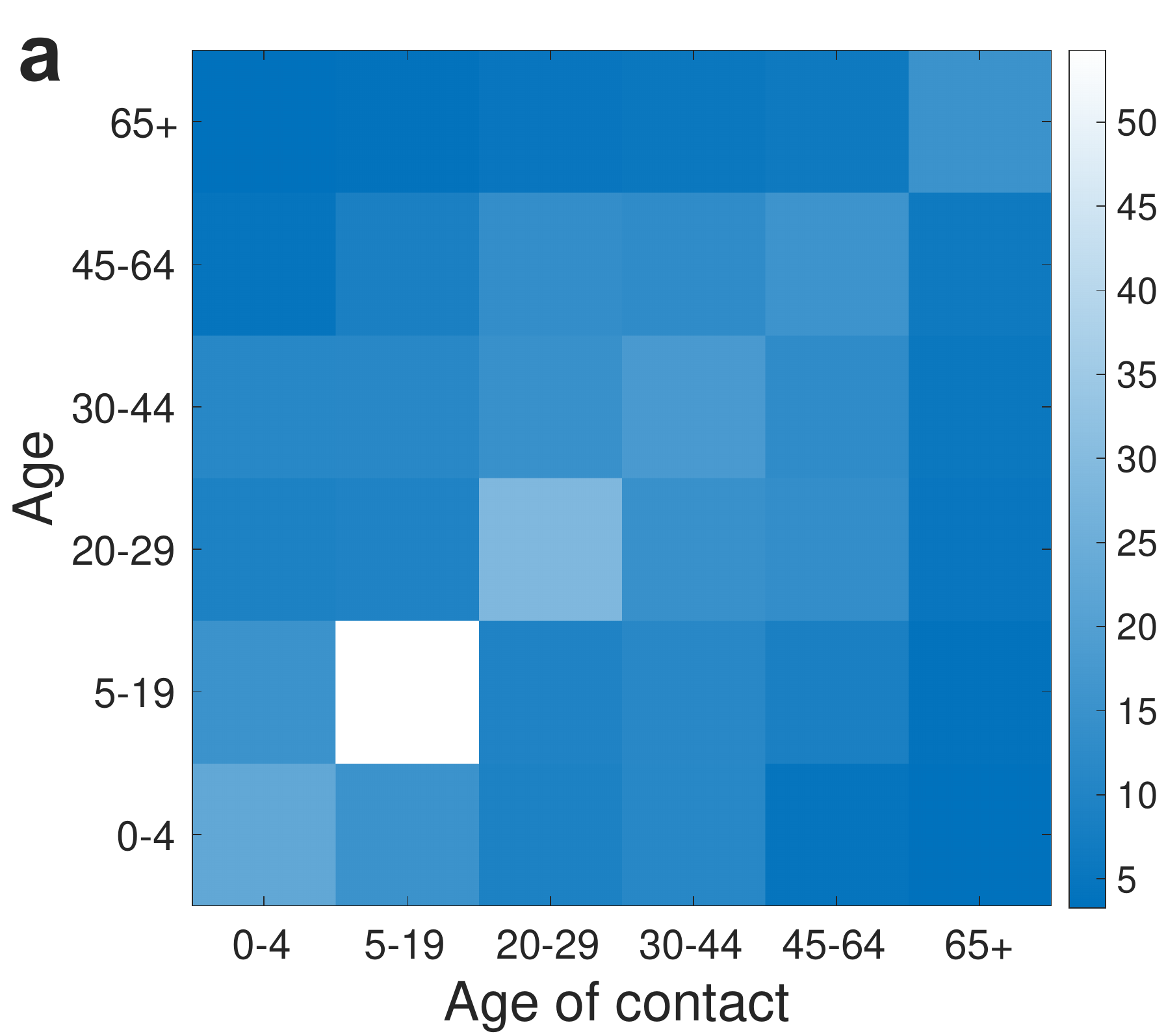}			
	\end{minipage}\hfill
	\begin{minipage}[b]{0.5\linewidth}
		\centering
		\includegraphics[width=0.8\linewidth]{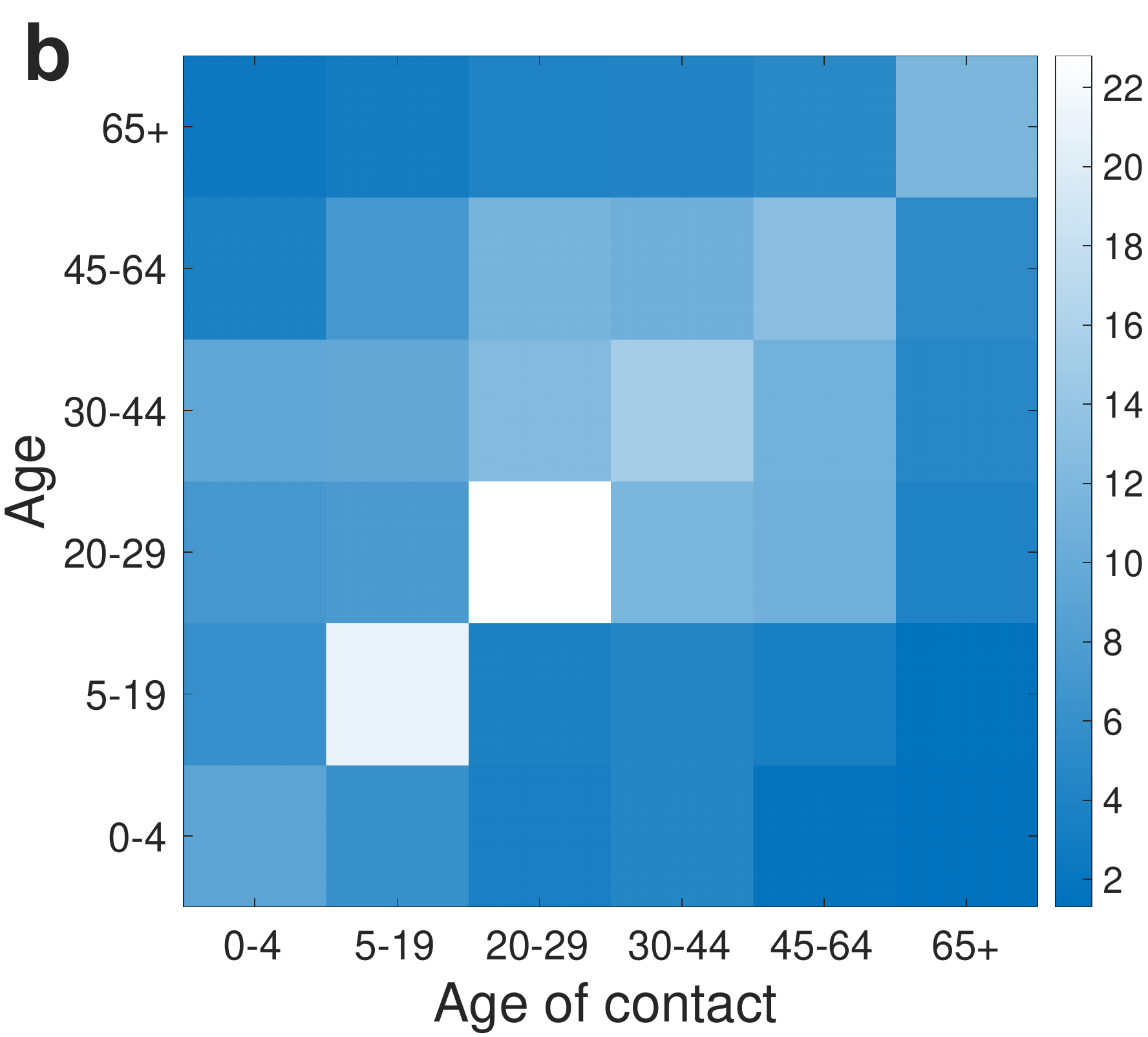}
	\end{minipage}%
	\centering
	\caption{{\bf School-age children (5-19) has the largest value of the contact rate, but young adults (20-29) has the largest value of the transmission rate.} {\bf a}, visualization of the intrinsic connectivity matrix $\Gamma$ for NY we used, which is defined in \eqref{eq: NY_Gamma}. Matrix $\Gamma$ measures the number of contacts between different age groups in each day. {\bf b}, visualization of the matrix  ${\rm diag}(\beta_0)\Gamma$, where $\beta_{0,a}$ represents the probability that an individual of age group $a$ be infected in each contact with others. Therefore $\{{\rm diag}(\beta_0)\Gamma\}_{ab}$ represents the probability that an individual from group $a$ be infected by people from group $b$ in each.}\label{fig: contact matrix}
\end{figure}

\clearpage
\begin{figure}[!htb]
	\centering
	\begin{minipage}[b]{0.95\linewidth}
		\centering
		\includegraphics[width=1.0\linewidth]{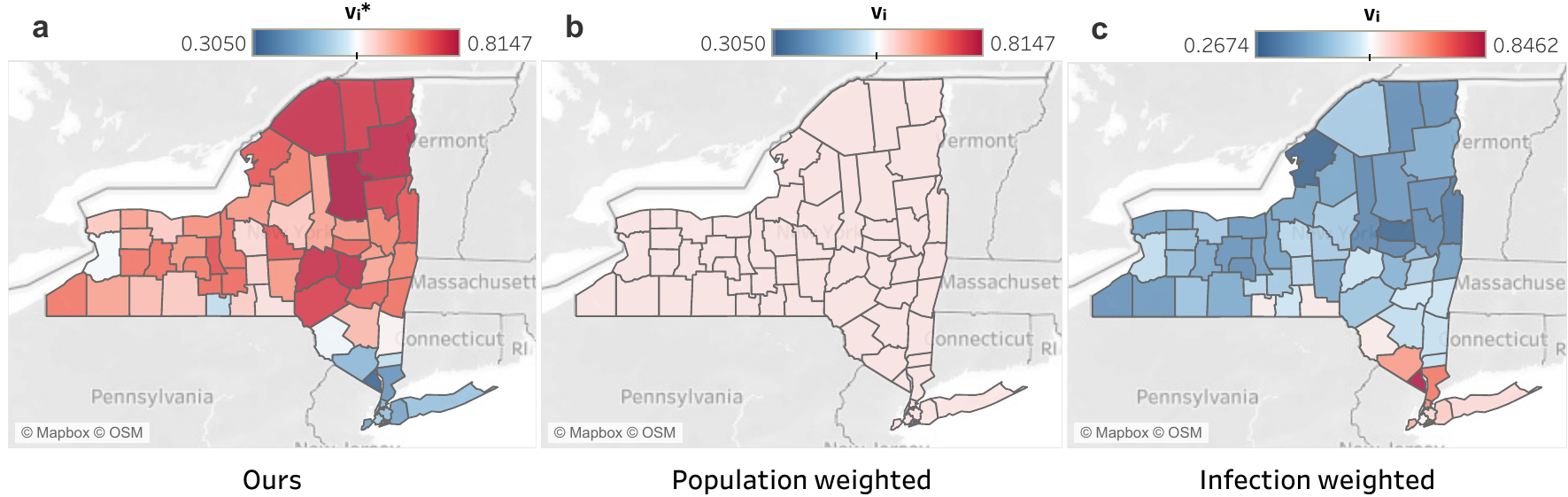}			
	\end{minipage}\hfill
	
	\vspace{3mm}
	
	\begin{minipage}[b]{0.95\linewidth}
		\centering
		\includegraphics[width=1.0\linewidth]{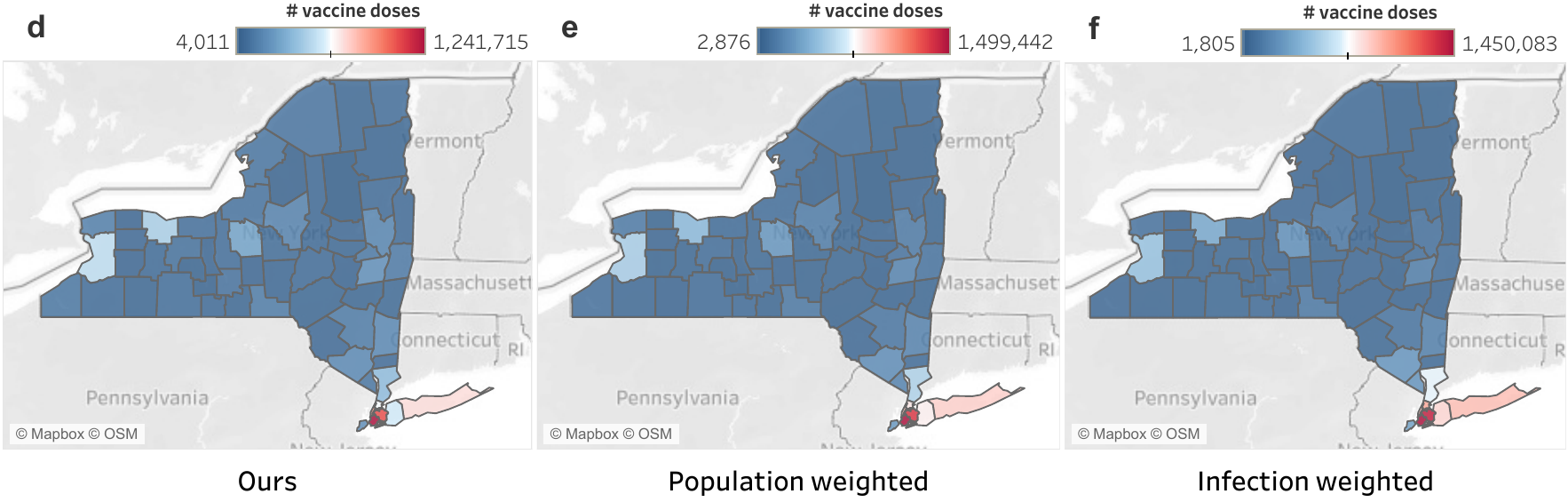}
	\end{minipage}%
	\centering
	\caption{
		{\bf When the vaccine supply is unlimited, our method suggests to allocate vaccines firstly to counties outside of the NYC and Long Island.} {\bf a-c}, vaccine allocation rate $v_i$ of each county given by different policies. {\bf d-f}, the number of vaccine doses allocated to each county by different policies. The vaccines are supplied daily with
		a speed of $0.33\%$ of total population per day. The total number of the available vaccine doses for these policies are the same, that is $100\%$ of the total population. In this figure, we only show the results before the disappearing of the epidemics (\# active cases is less than 1). After that, we will not distinguish between these policies, the leftover vaccines will be evenly distributed to the counties in each scenario. ``Population weighted", ``infection weighted", and ``no vaccine" polices are defined as in Fig. \ref{fig: demo_allocation_small}. The data applied is about COVID-19 outbreak in NY on December 1st, 2020. The values of $v_i^*$, $v_i$ in this figure corresponds to the results in Fig. S\ref{fig: demo_curves_large}.
	}\label{fig: demo_daily_large}
\end{figure}
%
%
%
%
%
%
\begin{figure}[!htb]
	\includegraphics[width=0.66\linewidth]{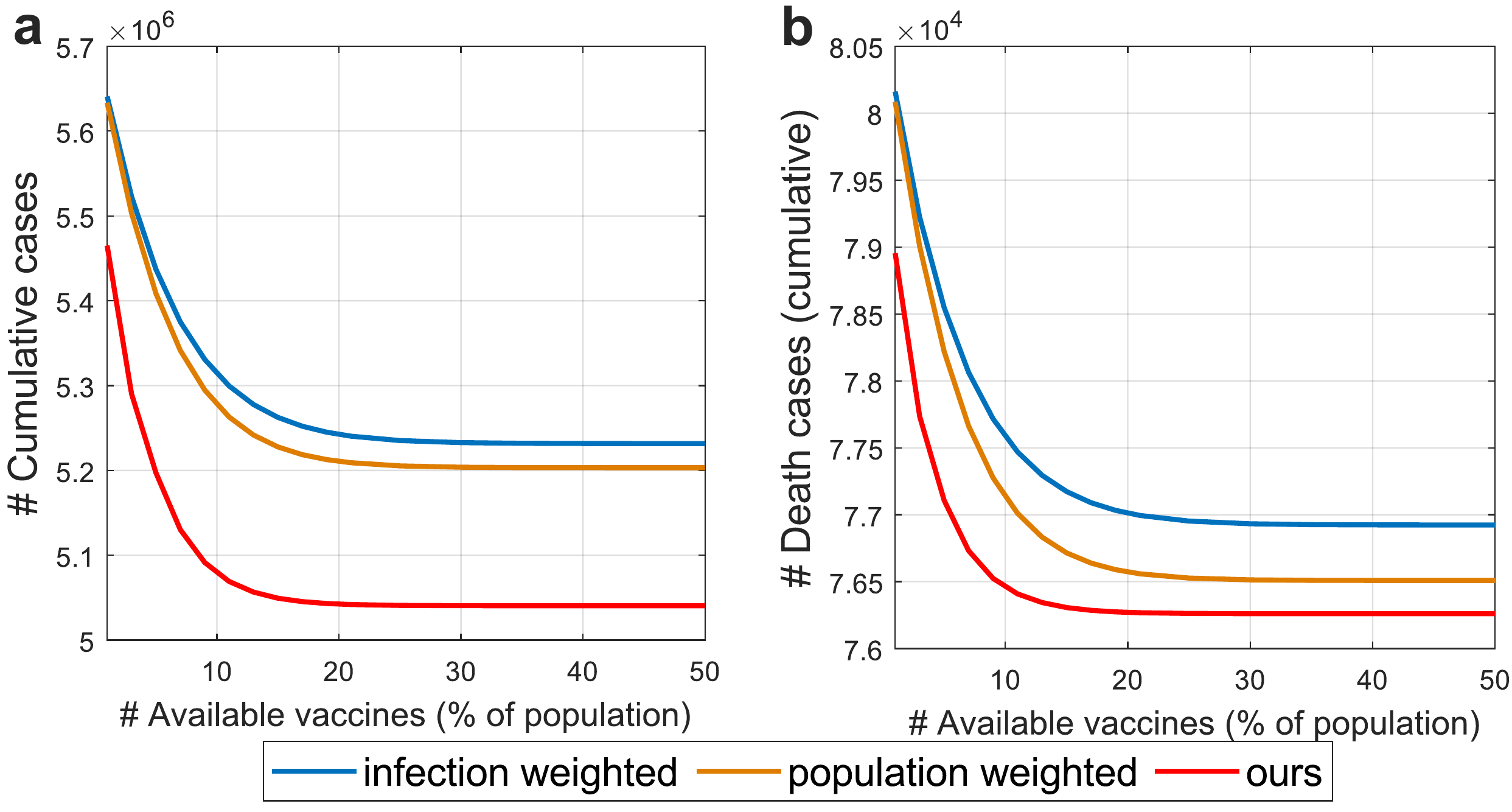}
	\centering
	\caption{{\bf When the demographic structure of the population is considered, our method outperforms the other comparison methods regardless of the number of the available vaccine doses}. {\bf a}, the estimated number of final cumulative cases.  {\bf b}, the estimated number of final cumulative death cases. The vaccines are supplied daily with a speed of $0.33\%$ of total population per day. The total number of the available vaccine doses ranges from 1\% to 50\% of the population. We only observe to $50\%$ as the epidemic disappears after supplying around $50\%$ people with vaccines. ``Population weighted" and ``infection weighted" polices are defined as in Fig. \ref{fig: demo_allocation_small}. The data applied is about COVID-19 outbreak in NY on December 1st, 2020. It can be observed that the superiority of our method over the other two methods is more obvious when there are enough number of the available vaccine doses.
	}\label{fig: demo_V_num_varying}
\end{figure}
\clearpage
\begin{figure}[!htb]
	\includegraphics[width=0.66\linewidth]{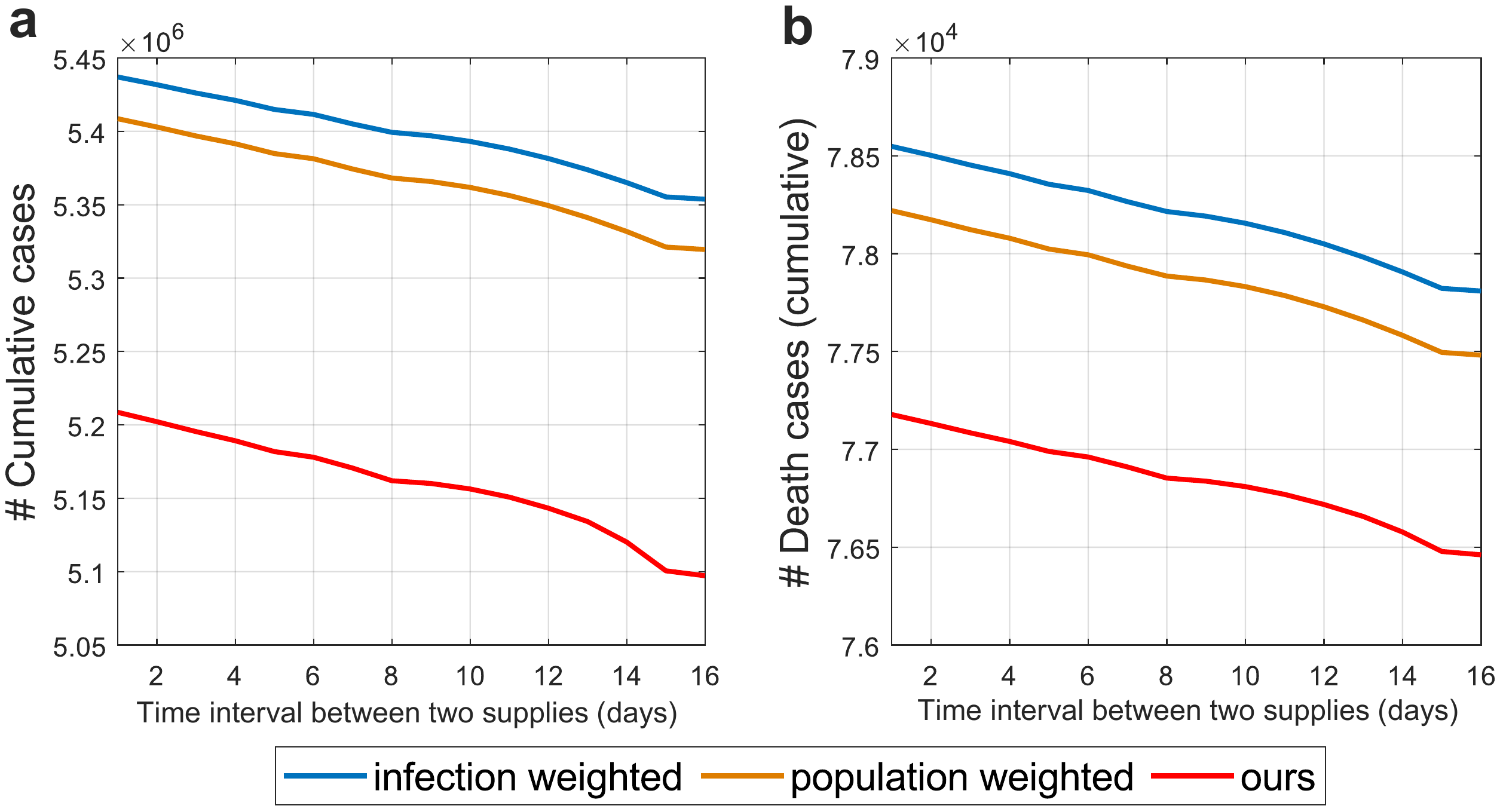}
	\centering
	\caption{{\bf When the demographic structure of the population is considered, our method outperforms the other comparison methods regardless of the time interval between two vaccine supplies}. {\bf a}, the estimated number of final cumulative cases.  {\bf b}, the estimated number of final cumulative death cases. The vaccines are supplied in the beginning of each period, the time interval between two vaccine supplies ranges from 1 day to 16 days. The average number of the vaccine supply per day for these scenarios are the same, that is 0.33\% of the population. The total number of the available vaccine doses for these policies is $5\%$ of the total population.  ``Population weighted" and ``infection weighted" polices are defined as in Fig. \ref{fig: demo_allocation_small}. The data applied is about COVID-19 outbreak in NY on December 1st, 2020. The number of the cumulative cases and the death cases decreases when the time interval increases as more people will be vaccinated in the beginning.
	}\label{fig: demo_period_varying}
\end{figure}
\clearpage
\begin{figure}[!htb]
	\includegraphics[width=0.66\linewidth]{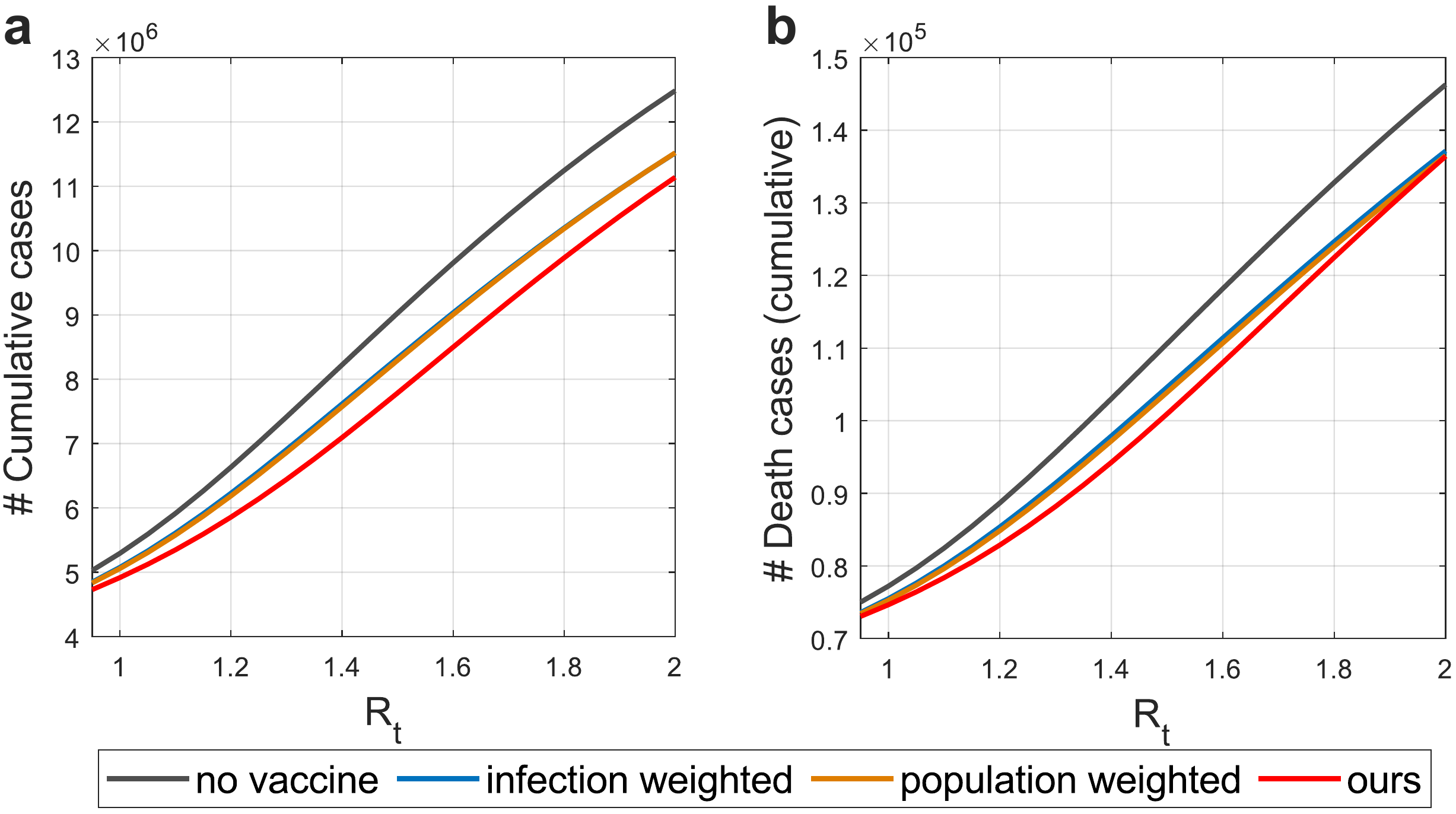}
	\centering
	\caption{{\bf When the demographic structure of the population is considered, our method outperforms the other comparison methods regardless of the the value of $R_t$}. {\bf a}, the estimated number of final cumulative cases.  {\bf b}, the estimated number of final cumulative death cases. In {\bf a} and {\bf b}, the value of $R_t$ ranges from 0.95 to 2.0. The vaccines are supplied daily with a speed of $0.33\%$ of total population per day. The total number of the available vaccine doses for these policies is $5\%$ of the total population.  ``Population weighted" and ``infection weighted" polices are defined as in Fig. \ref{fig: demo_allocation_small}. The data applied is about COVID-19 outbreak in NY on December 1st, 2020.
	}\label{fig: demo_Rt_varying}
\end{figure}

\begin{figure}[!htb]
	\begin{minipage}[b]{1.0\linewidth}
		\centering
		\includegraphics[width=1.0\linewidth]{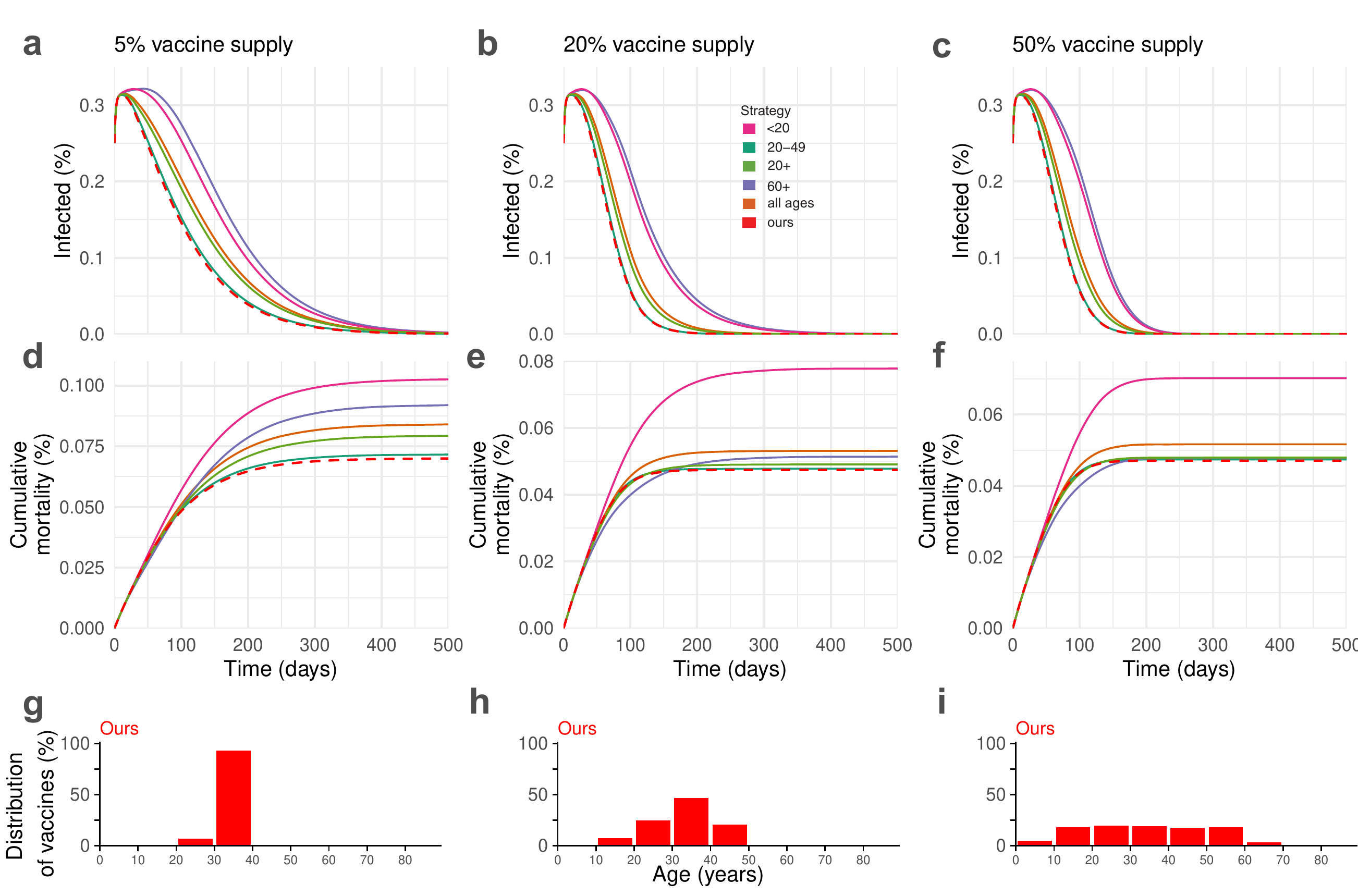}
	\end{minipage}%
	\centering
	\caption{{\bf Our method outperforms all the five age-stratified prioritization strategies in \cite{bubar2021model} when $R_0 = 1.05$.} {\bf a-f}, the estimated percentage of the infected cases and the cumulative mortality cases in the total population over 500 days. {\bf g-i}, distribution of vaccines of our method. In {\bf a-f}, the colors of the lines match with the polices, the red dashed line represents our policy. In {\bf a, d, g}, the total number of the available vaccine doses is 5\% of the population. In {\bf b, e, h}, the total number of the available vaccine doses is 20\% of the population. In {\bf c, f, i}, the total number of the available vaccine doses is 50\% of the population. All these results based on data and parameters for United States in \cite{bubar2021model}. The vaccines are supplied at 0.2\% of the total population per day. The vaccines are assumed to be all-or-nothing, transmission-blocking with 90\% efficacy.
	}\label{fig: R0_105_5_to_50}
\end{figure}

\begin{figure}[!htb]
	\begin{minipage}[b]{1.0\linewidth}
		\centering
		\includegraphics[width=1.0\linewidth]{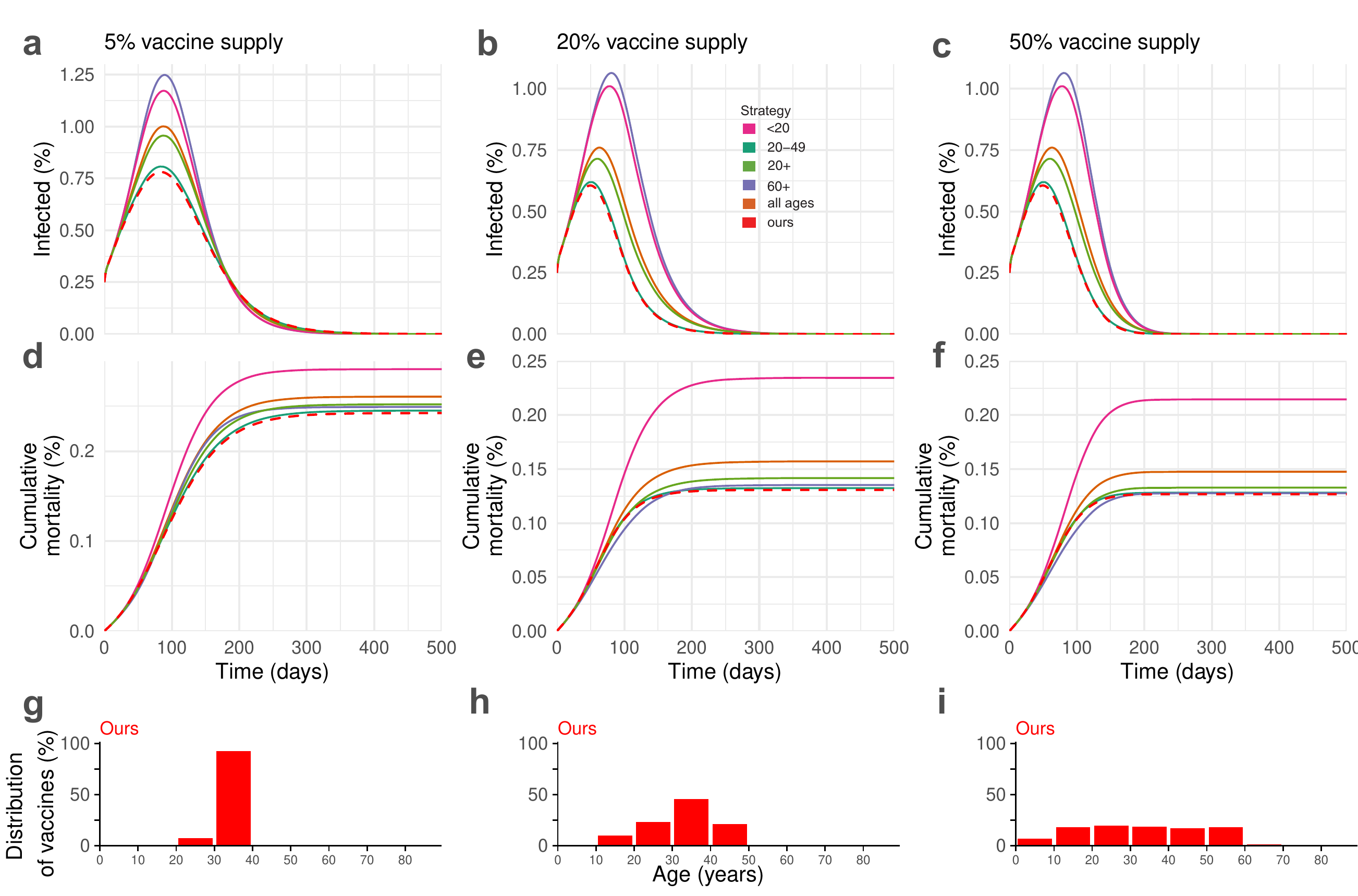}
	\end{minipage}%
	\centering
	\caption{{\bf Our method outperforms all the five age-stratified prioritization strategies in \cite{bubar2021model} when $R_0 = 1.25$.} {\bf a-f}, the estimated percentage of the infected cases and the cumulative mortality cases in the total population over 500 days. {\bf g-i}, distribution of vaccines of our method. In {\bf a-f}, the colors of the lines match with the polices, the red dashed line represents our policy. In {\bf a, d, g}, the total number of the available vaccine doses is 5\% of the population. In {\bf b, e, h}, the total number of the available vaccine doses is 20\% of the population. In {\bf c, f, i}, the total number of the available vaccine doses is 50\% of the population. All these results based on data and parameters for United States in \cite{bubar2021model}. The vaccines are supplied at 0.2\% of the total population per day. The vaccines are assumed to be all-or-nothing, transmission-blocking with 90\% efficacy.
	}\label{fig: R0_125_5_to_50}
\end{figure}

\begin{figure}[!htb]
	\begin{minipage}[b]{1.0\linewidth}
		\centering
		\includegraphics[width=1.0\linewidth]{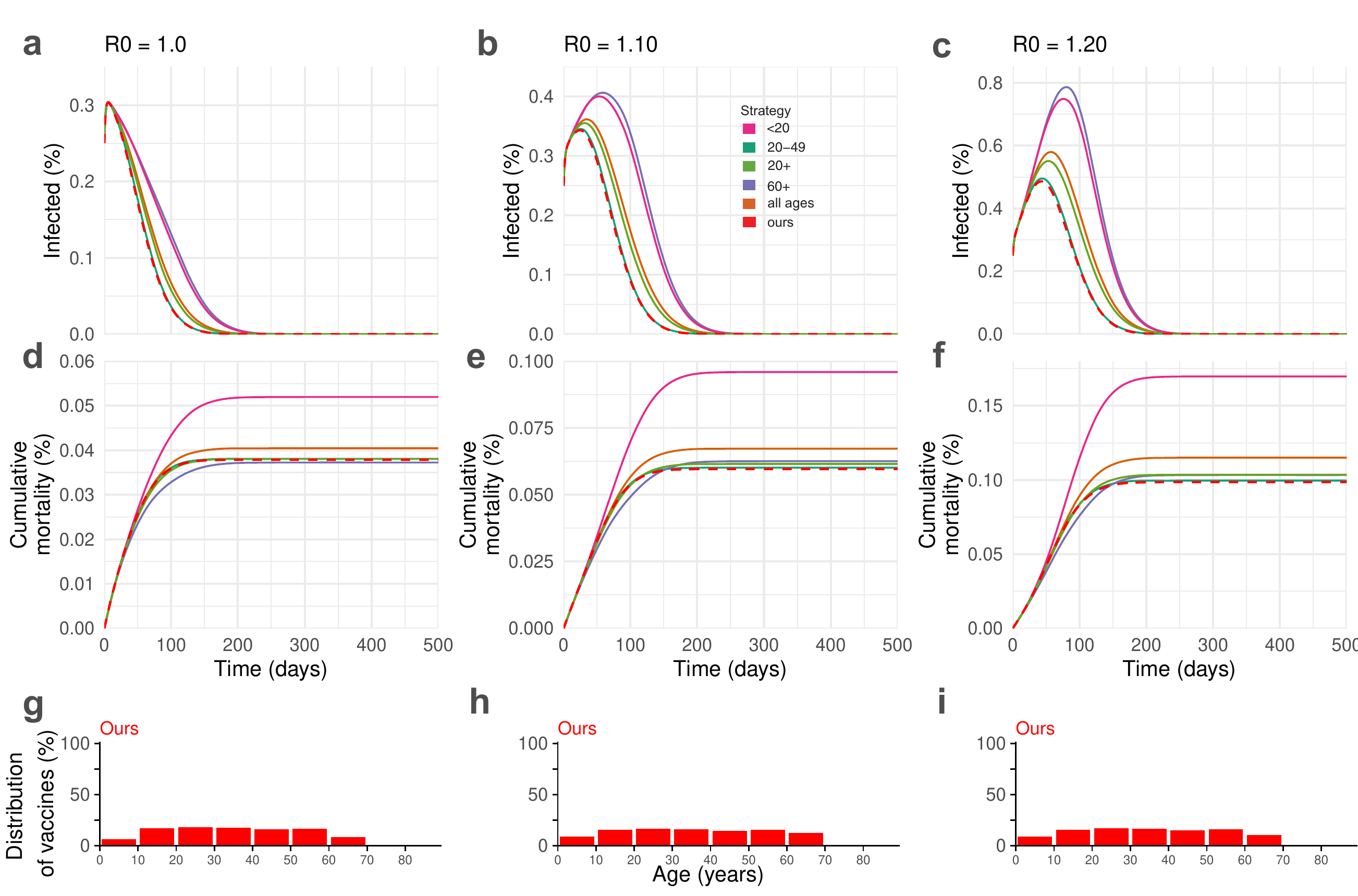}
	\end{minipage}%
	\centering
	\caption{{\bf If the vaccine supply is unlimited, and $R_0$ = 1.1, or 1.2, our method outperforms all the five age-stratified prioritization strategies in \cite{bubar2021model}.} {\bf a-f}, the estimated percentage of the infected cases and the cumulative mortality cases in the total population over 500 days. {\bf g-i}, the distribution of vaccines of our method. In {\bf a-f}, the colors of the lines match with the polices, the red dashed line represents our policy. In {\bf a, d, g}, $R_0$ = 1.0. In {\bf b, e, h}, $R_0$ = 1.1. In {\bf c, f, i}, $R_0$ = 1.2.  In this figure, the total vaccine supply is assumed to be 100\% of the population, in {\bf g-i}, we only show the distribution of vaccines before the disappearing of the epidemics, after that, the leftover vaccines would be evenly distributed to these age groups. In {\bf d}, our method slightly underperforms strategy that gives prioritization to seniors (60+).  All these results based on data and parameters for United States in \cite{bubar2021model}. The vaccines are supplied at 0.2\% of the total population per day. The vaccines are assumed to be all-or-nothing, transmission-blocking with 90\% efficacy.
	}\label{fig: R0_100_110_120_100_percent}
\end{figure}

\begin{figure}[!htb]
	\begin{minipage}[b]{0.66\linewidth}
		\centering
		\includegraphics[width=1.0\linewidth]{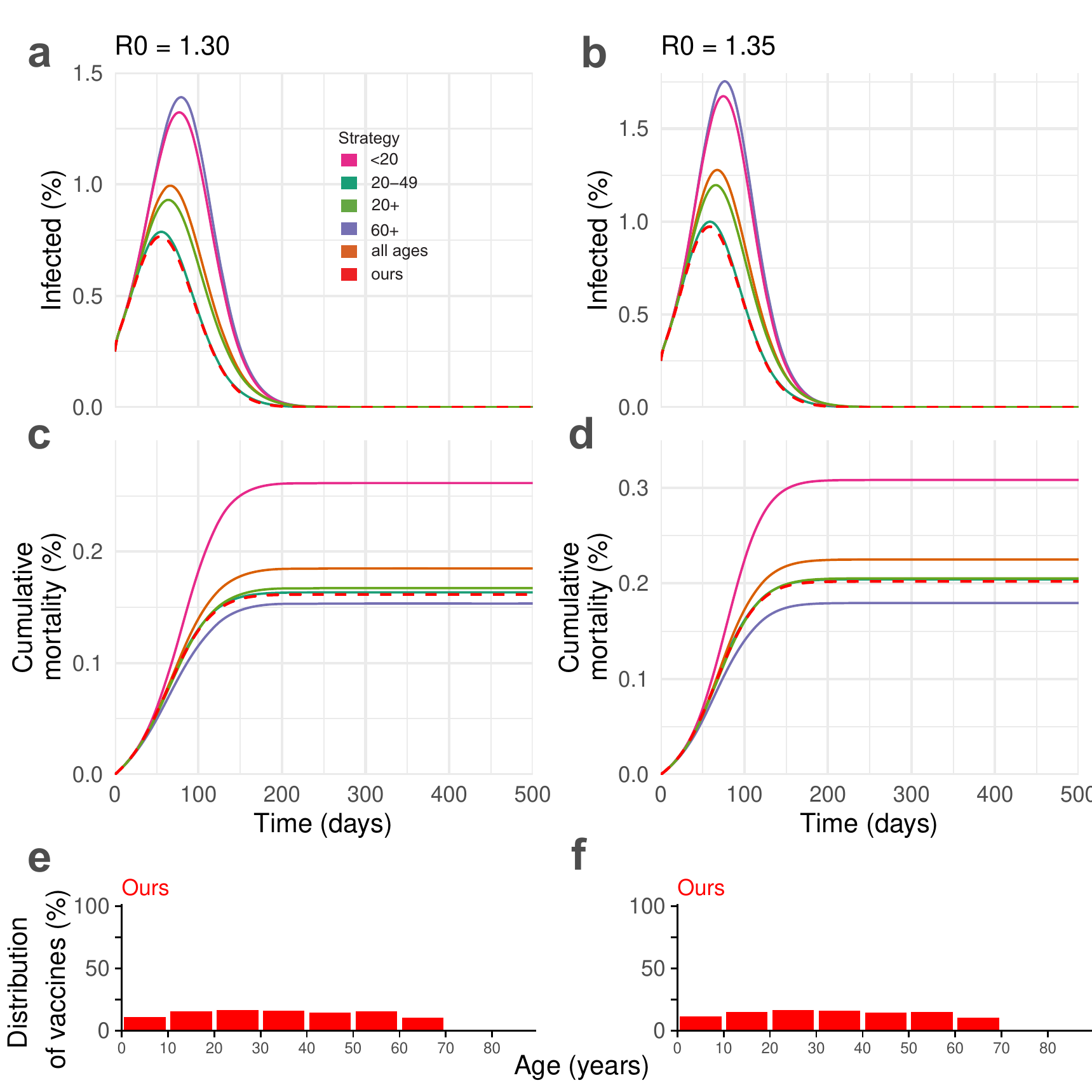}
	\end{minipage}%
	\centering
	\caption{{\bf If the vaccine supply is unlimited, and $R_0$ = 1.30, or 1.35, our method is the best in terms of the infected cases, the prioritization strategy for seniors (60+) is the best when consider the cumulative mortality cases.} {\bf a-d}, the estimated percentage of the infected cases and the cumulative mortality cases in the total population over 500 days. {\bf e, f}, the distribution of vaccines of our method. In {\bf a-d}, the colors of the lines match with the polices, the red dashed line represents our policy. In {\bf a, c, e}, $R_0$ = 1.30. In {\bf b, d, f}, $R_0$ = 1.35.  In this figure, the total vaccine supply is assumed to be 100\% of the population, in {\bf e, f}, we only show the distribution of vaccines before the disappearing of the epidemics, after that, the leftover vaccines would be evenly distributed to these age groups. All these results based on data and parameters for United States in \cite{bubar2021model}. The vaccines are supplied at 0.2\% of the total population per day. The vaccines are assumed to be all-or-nothing, transmission-blocking with 90\% efficacy.
	}\label{fig: R0_130_135_100_percent}
\end{figure}

\clearpage
\begin{figure}[!htb]
	\includegraphics[width=1.0\linewidth]{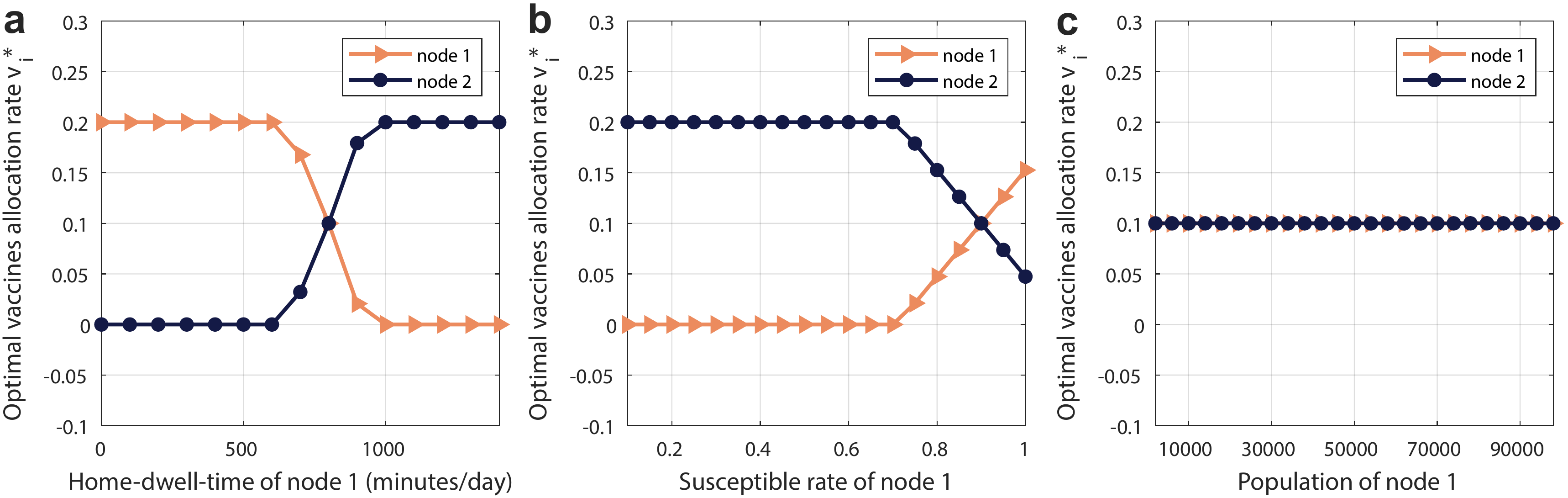}
	\centering
	\caption{{\bf For the two-nodes network model in SI Sec. \ref{sec: two nodes}, our optimal stabilizing allocation policy suggests to give zero vaccines to the node with larger value of home-dwell-time and smaller value of susceptible rate.} {\bf a}, the value of $v_i^*$ when the daily home-dwell-time of node 1 ranges from 0 to 1440 minutes. {\bf b}, the value of $v_i^*$ when the susceptible rate of node 1 ranges from 0.1 to 1. {\bf c}, the value of $v_i^*$ when the population of node 1 ranges from 2000 to 10000. It can be seen that the value of $v_i^*$ is sensitive to the home-dwell-time and the susceptible rate but not sensitive to the population. 
	}\label{fig: small_network_sensitivity}
\end{figure}

\clearpage

\clearpage

\section{Supplementary Tables}
\begin{table}[h]
	\centering
	\caption{Value of $d_A$ and $d_S$ in references}\label{tb: d_A d_S}
	\scalebox{0.9}{
		\begin{tabular}{c  p{1.3cm}<{\centering} p{1.4cm}<{\centering} p{1.5cm}<{\centering} p{1.3cm}<{\centering} p{1.3cm}<{\centering} p{1.3cm}<{\centering}}
			\toprule
		Reference &  \cite{birge2020controlling} & \cite{giordano2020modelling} & \cite{bertozzi2020challenges} & \cite{bubar2021model} & \cite{luciana2020} & \cite{bertsimas2020optimizing} \\
			\midrule  
			$d_A$ & 1/0.29 & 1/0.0034 & 5 & 3 & 5.1 & 10 \\
			$d_S$ & 1/0.29 & 1/0.017 & 5 & 5& 7.4 & 15 \\
			\bottomrule
		\end{tabular}
	}
\end{table}

\begin{table}[h]
	\centering
	\caption{Death rate $\kappa_i$ used in our simulations}\label{tb: death rate}
	\scalebox{0.9}{
		\begin{tabular}{c  p{1.3cm}<{\centering} p{1.3cm}<{\centering} p{1.4cm}<{\centering} p{1.3cm}<{\centering} p{1.3cm}<{\centering} p{1.3cm}<{\centering}}
			\toprule
			age group &  0-4 & 5-19 & 20-29 & 30-44 &45-64 & 65+ \\
			\midrule  
			$\kappa_i$ & 0.0002 & 0.00018 & 0.00036 & 0.0018 & 0.0094 & 0.0945 \\
			\bottomrule
		\end{tabular}
	}
\end{table}

\begin{table}[!htb]
	\centering
	\caption{The optimal stabilizing vaccine allocation rate $v_i^*$ for the two-nodes network model. \label{tab: two-nodes}
	}
	\scalebox{0.85}{
		\begin{tabular}{m{1.5cm}<{\centering} m{2.5cm}<{\centering} m{2cm}<{\centering} m{2cm}<{\centering} m{3.2cm}<{\centering} }
			\toprule Scenario & Case 1 & Case 2 & Case 3 & Case 4
			\\
			\midrule  
			$v_i^*$ &  [0.1\quad 0.0998] & [0\quad 0.2] & [0\quad 0.2] & [0.0923\quad 0.8744]  \\
			\bottomrule
		\end{tabular}
	}
\end{table}
\clearpage

\end{document}